\documentclass{amsart}
\usepackage{a4wide,tikz,hyperref,amsbsy,stmaryrd}

\usepackage{amsmath,amssymb,amsthm,latexsym,bm}
\usepackage{graphicx}
\usepackage{stmaryrd}		
\usepackage{xcolor}
\usepackage{hyperref}
\hypersetup{linktocpage, hidelinks}
\usepackage{tikz}
\usepackage{soul}
\usepackage{amsbsy}

\usepackage{amssymb}
\usepackage{amsthm}
\usepackage{thmtools}

\usepackage[normalem]{ulem}
\usetikzlibrary{arrows,automata}
\usetikzlibrary{decorations.pathreplacing}
\usepackage[colorinlistoftodos]{todonotes}

\def\cT{{\mathcal T}}

\newtheorem{lemma}{Lemma}[section]
\newtheorem{theorem}[lemma]{Theorem}

\newtheorem{corollary}[lemma]{Corollary}

\newtheorem{proposition}[lemma]{Proposition}
\theoremstyle{definition}
\newtheorem{definition}[lemma]{Definition}
\theoremstyle{remark}
\newtheorem{remark}[lemma]{Remark}
\newtheorem{example}[lemma]{Example}

\numberwithin{equation}{section}

\newcommand{\Z}{{\mathbb Z}}
\newcommand{\R}{{\mathbb R}}
\newcommand{\N}{{\mathbb N}}
\newcommand{\E}{{\mathbb E}}

\newcommand{\cL}{\mathcal{L}}
\newcommand{\cA}{\mathcal{A}}

\definecolor{vert}{cmyk}{0.8,0,1,0}

\newcommand{\tr}[1]{\vphantom{#1}^t #1}

\theoremstyle{definition}

\title{Coboundaries and eigenvalues of finitary $S$-adic systems}

\author[V.~Berth\'e]{Val\'erie Berth\'e}
\address{Universit\'e de Paris, IRIF, CNRS, F-75013 Paris, France}
\email{berthe@irif.fr}
\author[P.~Cecchi Bernales]{Paulina Cecchi Bernales}
\address{Departamento de Matemáticas,
Facultad de Ciencias,
Universidad de Chile,
Santiago, Chile}
\email{ pcecchi@uchile.cl}
\author[R.~Yassawi]{ Reem Yassawi}
\address{ School of Mathematical Sciences, Queen Mary University of London, United Kingdom}
\email{r.yassawi@qmul.ac.uk}
\thanks{This work was supported by the Agence Nationale de la Recherche and  the FWF Austrian Science Fund  through the project  ``SymDynAr''  (ANR-23-CE40-0024-01), and the
   EPSRC  grants, numbers  EP/V007459/2 and EP/S010335/1. The second author was supported by ANID/Fondecyt Postoctorado 3210746, and partially funded by Centro de Modelamiento Matemático (CMM), ACE210010 and FB210005, BASAL funds for centers of excellence from ANID-Chile and ECOS-ANID grant C21E04 (ECOS210033).
 }

\date{\today}



\begin{document}

\begin{abstract}

An $S$-adic system is a symbolic dynamical system  generated by iterating an infinite sequence of   substitutions or morphisms, called a directive sequence.
A {\em finitary} $S$-adic  dynamical system is one where the directive sequence consists of morphisms selected from a  finite set.
We study  eigenvalues and coboundaries for  finitary recognizable $S$-adic dynamical systems, i.e., those where points can be uniquely desubstituted using the given sequence of morphisms.
  To do this we identify the notions  of {\em straightness}  and  {\em essential} words, and use them to define a coboundary, inspired by  of Host's formalism, which  allows us to  
 express  necessary and    sufficient conditions that  a complex number must satisfy in order to be a continuous or measurable eigenvalue.
 We then  apply our results to   finitary  directive sequences   of substitutions of constant length, and show  how  to create constant-length $S$-adic shifts with non-trivial coboundaries. We show  that  in this case all continuous eigenvalues  are rational and we give a complete description of the rationals that can be an eigenvalue, indicating how this leads to a Cobham-style result for these systems.\end{abstract} 

\maketitle

\section{Introduction}

The Rokhlin-Kakutani lemma in ergodic theory allows one to approximate,  up to a set of small measure, a measure-preserving dynamical system by a partition of measurable sets, depicted as a tower, where the transformation consists of moving up the tower, see, e.g., \cite[Lemma 4.7]{Petersen:1983}.
{\em Recognizable} $S$-adic dynamical systems are a family of symbolic  systems which can be approximated by well-controlled Rokhlin towers,
where the instructions of how to construct finer Rokhlin towers from coarser ones are defined  by a {\em directive sequence} $(\sigma_n)_{n\geq 0}$ of {\em morphisms} or {\em substitutions}.
 In a previous work \cite{BSTY},
  recognizability for sequences of morphisms  was studied, and  conditions which guaranteed it were identified. In this article, we apply these results by
studying (dynamical)  eigenvalues  for  recognizable $S$-adic dynamical systems,  where the directive sequence consists of morphisms selected from a  finite set. 
 We will call such directive sequences {\em finitary}  (see Section \ref{preliminaries} for definitions of recognizability and substitutions).  
 
  There are various related notions in the measurable and topological dynamics literature that produce recognizable $S$-adic systems. In particular,   {\em cutting-and-stacking}  measurable systems on a Lebesgue space (see e.g.  \cite{Katok-Stepin-1967, Ferenczi:1997}) are built with recursive instructions which can be interpreted as  a directive sequence, and provided the latter is recognizable, the $S$-adic shift is a symbolic representation of the cutting-and-stacking system.   Also, under the condition that the morphisms $\sigma$  in the directive sequence are    {\em proper}, i.e., $\sigma(a)$ starts with the same letter  and ends with the same letter  for each  letter $a$, recognizable $S$-adic systems can be directly seen as  topological  Bratteli-Vershik systems, see  \cite{BSTY}, and \cite{Durand-Host-Skau} for the stationary case.    
  
  Moreover, the  model 
  theorem from  \cite[Theorem 4.7]{HPS:1992}  guarantees us  a proper Bratteli-Vershik representation for minimal systems, but  we cannot usually obtain such a representation by simple manipulations of the data we are given in the form of  the directive sequence: an $S$-adic system is not generally  topologically conjugate to the {\em natural} Bratteli--Vershik system associated to the directive sequence, although they are almost--conjugate (see Theorem \ref{thm:Bratteli-Vershik}  and \cite{BSTY} for details). The manipulations required to obtain a proper Bratteli-Vershik representation are already complicated even in the stationary case, see \cite{Durand-Host-Skau, Bezugly:2009}.    More generally, even  if finitary $S$-adic subshifts  are known to be  representable in as a finite rank Bratteli-Vershik system, and  therefore   represented   in an $S$-adic and proper way, see \cite{Espinoza, GH:22},  the construction allowing one to go  from a finitary $S$-adic representation to a proper $S$-adic one is not totally effective. 
 The approach of this article is to work directly with the given $S$-adic  description of the shift under study.  As there are several important results concerning the eigenvalues of the latter families of systems, we describe throughout how our findings intersect or extend the existing literature. 

 A {\em substitution shift} $(X_{\sigma},T)$ is an $S$-adic dynamical system where the  directive sequence  $(\sigma_n)_{n\geq 0}$ is constant, i.e.,  there exists a substitution $\sigma$ such that $\sigma_n=\sigma$ for each $n$.
Our departure point for this work is Host's article \cite{Host:1986}, where he identifies {\em coboundaries} as an important tool for describing  eigenvalues of a {\em primitive} substitution shift $(X_{\sigma},T)$, which allowed him moreover 
to prove   the striking result that   measurable eigenvalues are  continuous eigenvalues. We use here the term coboundary following Host's definition, which is not to be confused  with the usual meaning of a  coboundary, i.e.,  a function of the form $f- f\circ T$. A coboundary is basically partial information that identifies the relationship between the values that a putative eigenfunction  $f$ can take on the orbit of some select set of points $\mathcal D$ called {\em limit words}, see Definition \ref{def:limit-word}. In the substitution case, $\mathcal D$ is the finite set of {\em $ \sigma $-periodic points}, and primitivity ensures that the $T$-orbit of any of these points is dense in $X_{\sigma}$. For a primitive substitution $\sigma$, Host shows that any coboundary defines a continuous eigenvalue,   i.e., one with a continuous eigenfunction, and every measurable eigenvalue, i.e., one with a measurable eigenfunction,  defines a coboundary, and therefore is a  topological eigenvalue; see Theorem \ref{Host}.  In this proof, Host identifies that for $\lambda$ to be an eigenvalue, $\lambda^{|\sigma^n(a)|}$ must converge to a coboundary   for each letter $a$, but in fact he needs, and shows, the stronger fact that this convergence must be geometric.

In this paper,  we extend Host's notion of a coboundary to the finitary and {\em straight} $S$-adic setting.   When moving from a substitution shift to an $S$-adic shift, we move from working with one language to working with a sequence of languages; this is reflected in Definition \ref{def:Sadic-coboundary-one-sided} for the notion  of  a coboundary. 
  One other novelty that appears in our  definition of a coboundary is the requirement that the directive sequence be {\em straight}  (this notion is  introduced in Section \ref{sec:straight}.) This is simply so that there is a well-defined set of points $ \mathcal D$  of limit words as above, on which to partially define an eigenfunction.  Straight directive sequences abound, for example, all stationary directive sequences  can be uniformly  telescoped to be straight; in fact in \cite{Host:1986},  this is a property that is used. In Section \ref{subsec:ex} we show that many well-studied families of substitutions define straight directive sequences.

  We   investigate  the distinction    between continuous and measurable   eigenvalues  in the $S$-adic case,   a program which was   
    initiated  in \cite{Host:1986} for the substitutive case.  We  show in  Section \ref{sec:coboundary-directive}   (devoted to the continuous case),  and more particularly   in Theorems \ref{Host-continuous-bounded-onesided}  and  \ref{thm:sufficient-continuous-A}   that, with the given conditions, coboundaries define continuous eigenvalues.
  Moreover, in    Theorem \ref{thm:sufficient-continuous-A}, 
    we show that if the  convergence to the coboundary  is fast,   then the coboundary defines an eigenvalue,   and in  Section \ref{sec:Pisot}, we deduce  from  stronger convergence properties  that  eigenvalues associated to trivial coboundaries  are easy to compute, in the sense   that they are  	related to
measures of letter cylinders. In Section \ref{measurable}, Theorem \ref{thm:sufficient-mble}  gives  
a necessary condition for the existence of a measurable eigenvalue, to be compared to the condition in Theorem \ref{thm:sufficient-continuous-A}.

 Sufficient conditions  (both for the continuous and measurable case)  for  being an eigenvalue  are expressed  as summation-type  conditions; such conditions are  natural and classical,   and appear in several places  in the literature, for example in \cite[Corollary 15.57]{Nadkarni-1998}.
  In particular,  Theorem \ref{thm:sufficient-mble}  
 can be seen as an extension of \cite{CDHM:2003}, where the authors study finitary,  proper 
 directive sequences which define  {\em linearly recurrent }$S$-adic systems.  Note that the condition  that the morphisms are proper ensures that the only coboundary is the trivial one, which is why their result makes no mention of coboundaries whilst ours does;  we comment on why this is the case  in 
 Proposition \ref{prop:proper-implies-trivial-coboundary}. More generally, let us refer to the  important literature  devoted to the study of  dynamical eigenvalues of Cantor;
 see e.g.  \cite{Bressaud-Durand-Maass:2005,Bressaud-Durand-Maass:2010,CDHM:2003,Durand-Frank-Maass:2015,DFM:19}.

We then apply our results in Section \ref{constant-length-S-adic} to the special family of $S$-adic shifts defined by directive sequences  of constant-length substitutions. This is a natural family to consider in this context,   as Host's notion of a coboundary generalises the notion of {\em height} for constant-length substitutions, and indeed it is for this family that we most easily find non-trivial coboundaries. Here also, we show in Example \ref{ex:nontrivial height}
 that it is not difficult to create constant-length $S$-adic shifts with non-trivial coboundaries. In the case where the directive sequence is also finitary and satisfies some mild assumption, we show in  Theorem \ref{cor:fully-essential-constant-coboundarybis} that all continuous eigenvalues are rational, and we give a complete description of  the rationals that can be an eigenvalue. This allows us to conclude with  an $S$-adic  version of  Cobham's theorem in Corollary \ref{cor:cobham}.
 
Some results concerning sequences of constant-length substitutions exist in the finite rank literature, namely, work by Mentzen in \cite{Mentzen:1991}, who shows that {\em uniform exact} finite rank systems have only rational measurable eigenfunctions; see Section \ref{measurable} for the relevant definitions. Also, in the case where all the substitutions are  proper, we can look at the literature concerning eigenvalues of {\em Toeplitz} shifts, given  by Durand, Frank and Maass in \cite{Durand-Frank-Maass:2015}, with relevant results also in \cite{CDHM:2003} and \cite{Bressaud-Durand-Maass:2005}.  The results in the latter works for constant length $S$-adic shifts intersect ours in the family of Toeplitz shifts; see Section \ref{Toeplitz-S-adic} for definitions and a comparison of our results to these aforementioned works. We also give in Theorem \ref{thm:discrete-spectrum}  a generalisation of a result of  Kamae,  and Dekking   \cite{Kamae:1972, Dekking:1977}, which gives broad conditions that ensure that a constant-length $S$-adic system is Toeplitz.

 We illustrate our work with the running Example \ref{ex:first}, which gives an example of a finitary strongly primitive and non-proper $S$-adic shift which is not linearly recurrent, and  whose spectrum can be explicitly described using our methods. This example is from Durand's work    \cite{Durand:00a, Durand:00b}.

 The study of  eigenvalues for  substitutions has given rise to an abundant literature.  
 The approach developed in  \cite{Ferenczi-Mauduit-Nogueira} relies on  the notion of return words
 which avoids the  use of  coboundaries.  The series of papers  \cite{Bressaud-Durand-Maass:2005,Bressaud-Durand-Maass:2010,CortDP:16,CDHM:2003,Durand-Frank-Maass:2015,DFM:19,ADE:23} deals  with proper substitutions, here again avoiding coboundaries. See also \cite{DurGoy:19,BMY} for the constant-length case  and  \cite{Cassaigne-Ferenczi-Messaoudi:08}
for the Arnoux-Rauzy shifts. The $S$-adic Pisot case is handled in \cite{BST:20}  where continuous eigenvalues are considered.
Observe that similar results are obtained for flows and ${\mathbb R}$-actions (instead of ${\mathbb Z}$-actions, which is the viewpoint  developed here). See \cite{Solomyak:97,Solomyak:07}
for an extension of Host's result 
for substitution tilings and   see \cite{Fusion:14} for the constant-length case for   multidimensional tilings  of $S$-adic type, with an  approach based on return words  avoiding the use  of  coboundaries. Let us cite also   \cite{BS:20}    for  the study of the modulus of continuity of  spectral measures in the weakly mixing  case  via  a spectral cocyle.
Note  that in  the case  of flows,  towers have the same height, which  is reminiscent of the constant-length case. 
 
In conclusion, the papers we have cited above show situations where it is possible to characterize  topological  and measurable eigenvalues,  and to distinguish them.
 The present work is a preliminary step  towards the extension of the existing work in the proper case mentioned above, in order to get a  more complete characterization of   continuous vs. measurable eigenvalues in the $S$-adic setting,
by highlighting the role played by coboundaries.

Let us sketch the contents of this paper.
Preliminaries are recalled in Section \ref{preliminaries}. In Section \ref{sec:comspec} we introduce 
  the  combinatorial notions  of essential words and of straightness;
we also  discuss eigenvalues and coboundaries for substitutions.  In Section \ref{sec:coboundary-directive} we introduce $S$-adic coboundaries and discuss their connection to the existence of 
 continuous eigenvalues of $S$-adic shifts;  we also  investigate the relation between continuous eigenvalues and measures of letters in the case of a trivial coboundary.
Section \ref{measurable} deals with measurable eigenvalues in the  case of  so-called finite exact rank.
Lastly, in Section \ref{constant-length-S-adic} we focus on the  family of $S$-adic systems generated by constant-length directive sequences and 
in particular, on the Toeplitz $S$-adic shifts.

\subsection*{Acknowledgements}The authors warmly thank Karl Petersen for  fruitful discussions around eigenvalues of a one-sided shift and its natural extension, and in particular for providing  Proposition \ref{two-one-sided-continuous}, as well as  Dominique Perrin for  enlightening discussions around  one-sided recognizability
and Wolfgang Steiner for discussions on   the notion of transition words. Finally the authors thank the referees, who helped us to significantly  improve the previous version of this paper. In particular we thank the second referee for the stronger statement of Corollary \ref{cor:cobham}. We also thank  the Lorentz center 
for   the fruitful discussions that  have taken place
during the workshop ``Dynamics on zero-dimensional spaces: new connections".

\section{Preliminaries}\label{preliminaries}

This section recalls  the  notions that will be needed hereafter. Section \ref{subsec:subs}   first  recalls   basic definitions 
 from  symbolic dynamics, Section \ref{main_S_adic}   is devoted to  $S$-adic shifts, Section \ref{subsec:recog} to recognizability, 
and Section \ref{sec:generating-partitions}   discusses  the  partitions consisting of  Rokhlin towers associated to  recognizable $S$-adic shifts.

\subsection{Shifts, substitutions and eigenvalues}\label{subsec:subs}

 Let $\mathcal{A}$ be a finite set of symbols, also called an {\it alphabet}, and let $\mathcal{A}^\mathbb{Z}$ denote the set of  two-sided infinite sequences over $\mathcal{A}$. 
 Endowing $\mathcal{A}$  with the discrete topology, we equip  $\mathcal{A}^\mathbb{Z}$ with the metrizable product topology.
 In this work we consider shift dynamical systems, or {\em shifts} $(X,T)$, where $X$ is a closed $T$-invariant set of $\mathcal A^\mathbb Z$ and $T:\mathcal A^\mathbb Z\to \mathcal A^\mathbb Z$ is the left shift map  $(x_n)_{n\in\mathbb{Z}} \mapsto (x_{n+1})_{n\in\mathbb{Z}}$.  We call these invertible shifts {\em two-sided}.  We use letters $x,y,z$ to denote points in two-sided shift spaces $\mathcal{A}^\mathbb{Z}$.
  In some cases we will discuss  non-invertible, or one-sided, shifts  $(\tilde{X},T)$, where $\tilde{X}\subset \mathcal{A}^\mathbb{N}$  stands for the set of one-sided,  or right-infinite sequences over $\mathcal A$.   A  shift is {\em minimal} if it  has no non-trivial closed shift-invariant subset.   
    We say that $x \in \mathcal{A}^\mathbb{Z}$ is \emph{periodic} if $T^k(x) = x$ for some $k \ge 1$, \emph{aperiodic} otherwise.
The~shift $(X,T)$ is said to be \emph{aperiodic} if each $x \in X$ is aperiodic.   For basics on
  continuous and measurable dynamics, see  e.g. \cite{cornfeld-fomin-sinai,Petersen:1983,Walters}.

  Given a finite alphabet  $\mathcal{A}$, let  $\mathcal{A}^*$ be the free monoid of all (finite) words over $\mathcal{A}$ under the operation of concatenation, and let $\mathcal{A}^+$  be the set of all non-empty words over $\mathcal{A}$. We let  $|w|$ denote   the length of a finite word~$w$ and  let $|\mathcal A|$ denote the cardinality of  the set $\mathcal A$. A {\em subword}   of  a word or a sequence $x$  is    a finite word $x_{[i,j)}$, $i \le j$, with $x_{[i,j)} := x_i x_{i+1}\dots x_{j-1}$.  (We use here  the term   subword instead of factor, which we use only  for  topological or measurable factors.)  For any word $w\in \cA ^+$ and any $0\leq j<|w|$, the notation $T^jw$ refers to the word $w_{[j,|w|)}$.
  A {\em language} is a collection of words in $\mathcal{A}^*$. 
The \emph{language~$\mathcal{L}_x$} of $x = (x_n)_{n\in \mathbb Z} \in \mathcal{A}^\mathbb{Z}$ is the set of all its subwords. 
The language~$\mathcal{L}_X$ of a one- or two-sided shift $(X,T)$ is the union of the languages of all $x \in X$; it is closed under the taking of subwords and every word $w$ in $\mathcal{L}_X$ is \emph{left- and right-extendable}  to a word in $\mathcal{L}_X$, i.e., there  exist letters $a,b$ such that $awb \in \mathcal{L}_X$.
Conversely,  a language~$\mathcal{L}$ on~$\mathcal{A}$ which is closed under the taking of subwords and such that each word is both left- and right-extendable defines a one- or two-sided shift $(X_\mathcal{L}, T)$, where $X_\mathcal{L}$ consists of   the set of points all of whose subwords belong to~$\mathcal{L}$, so that $\mathcal L_{X_\mathcal{L}} = \mathcal L$.   
We remark that the two-sided shift defined by $\mathcal L$ is the natural extension of the one-sided shift defined by $\mathcal L$; for a definition and details, see \cite{cornfeld-fomin-sinai}.

Let $X$ be a one-sided  (respectively two-sided)  minimal shift. If $\mu$ is a $T$-invariant probability Borel measure on $X$, then
 $T$ induces an isometric (respectively unitary) Koopman operator on $L^2(X,\mu)$,  namely $f\mapsto f\circ T$.  The  {\it eigenvalues} of $(X,T,\mu)$ are by definition the eigenvalues  
of this operator, and the {\it eigenfunctions} of $(X,T,\mu)$ are defined as being its eigenvectors.  The set of eigenvalues is an invariant for  conjugacy  in measure  and is called the {\em discrete spectrum of $(X,T,\mu).$}  If there exists an orthonormal basis of  $L^2(X,\mu)$ consisting of eigenfunctions,  then  $(X,T,\mu)$ is  said to have \emph{discrete spectrum}. One can also consider the {\it continuous} eigenvalues of $(X,T)$: these are values $\lambda $ such that there exists a continuous function $f:X\rightarrow \mathbb C$ with $f\circ T = \lambda f$.  Eigenvalues belong to the set  $ \mathbb S^{1}$ of  complex numbers  with modulus $1$.   If $t$ is such that $e^{2i \pi t}$ is a (continuous)    eigenvalue, then $t$ is said to be an {\em additive (continuous)  eigenvalue}.  Abusing  language, we say that  $\lambda=e^{2i \pi t}$ is rational  whenever 
$t$ is rational, i.e., whenever $\lambda$ is a root of unity. We  denote by  $E(X,T)$  the  additive  group of additive continuous eigenvalues of $(X,T)$.
It
is a topological invariant for $(X,T)$ for conjugacy.   Continuous eigenvalues  are measurable
if  the invariant measure is a Borel  measure.
An example of a family for which $E(X,T)$ equals the set of measurable eigenvalues is  the family of primitive substitution shifts (see  \cite{Host:1986} and  also see below), but in general the set of measurable eigenvalues strictly contains the continuous ones.  See for example \cite{Downarowicz-Lacroix:1996} and the comments at the end of Section \ref{Toeplitz-S-adic}.

 Let $v, w$ be two words. The cylinder $[v.w]$ is defined  as the set
$\{ x \in X \mid x_{[-|v| , |w|)} = vw \}$  and the cylinder   $[w]$ is the set
$\{ x \in X \mid x_{[0 , |w|)} = w \}$.  Cylinders are clopen sets. If $w$ is a letter,  the cylinder is   called a {\em one-letter cylinder},  and the same holds for two letters.
 Let ${\mathcal M} (X,T)$ be the set of all $T$-invariant probability measures on $(X,T)$.  Let  $I$ be  the  additive group  generated by   the measures of  cylinders, i.e., 
$$I(X,T) = \bigcap_{ \mu \in {\mathcal M} (X,T)} \left\{ \int f d\mu \,: \, f \in C(X,{\mathbb Z})\right\}.
$$
   One has   $E(X,T) \subset I(X,T)$ by   \cite[Proposition 11]{CortDP:16} and   \cite[Corollary 3.7]{GHH:18}. 
In particular, if $(X,T)$ is uniquely ergodic with unique $T$-invariant probability  measure $\mu$, then $
 I(X,T) =  \left\{ \int f d\mu \,: \, f \in C(X,{\mathbb Z})\right\} = \left\langle \{\mu([w]): w\in\mathcal{L}(X)\}\right\rangle$. Sufficient conditions 
for  the  additive group generated by  measures of one-letter cylinders being included in $E$      under the assumption of unique ergodicity 
are given in Section \ref{sec:Pisot}.

 The following result relating eigenvalues for one-sided and two-sided shifts is  graciously provided by Karl Petersen.
 Note that a  measure $\tilde{\mu}$ on a one-sided shift space defines  $\mu$ on the natural extension. Both measures give the same mass to cylinders, the only difference is that they are defined on different spaces; see \cite[Chapter 10, Section 4]{cornfeld-fomin-sinai} 
 for details.

\begin{proposition}\label{two-one-sided-continuous}
Let $(\tilde{X},T)$ be a one-sided minimal  shift, and  let 
$(X,T)$ be its natural extension. Then
\begin{itemize}\item
  $E(\widetilde X,T)=E(X,T)$.
 \item Let  $\tilde{\mu}$ be a shift invariant measure on $(\tilde{X},T)$ and  $\mu$ the corresponding measure on $(X,T)$. Then   $\lambda$ is a measurable eigenvalue  for   $(\tilde{X},T, \tilde{\mu})$ if and only if it is also one for  $(X,T, \mu)$.\end{itemize}
\end{proposition}
\begin{proof}

One easily checks that $E(\widetilde X,T)\subset E(X,T)$. Conversely, suppose that $\lambda$ is a continuous eigenvalue of $(X,T)$ and let $f\in C(X)$ be a corresponding eigenfunction with $||f||_{\infty}=1$.  Given $\tilde{x}\in \tilde{X}$, let $x\in X$ agree with $\tilde{x}$ on the non-negative indices. Choose such an $x$ and fix it. Define $\tilde {f}:\{T^n(\tilde{x}): n\in \N\}\rightarrow \mathbb C$ by $\tilde {f}(T^n(\tilde{x})) = \lambda^n(f(x))$. If we can show that $\tilde f$ is uniformly continuous on $\{T^n(\tilde{x}): n\in \N\}$, then we are done as  we can then extend $\tilde f$ to a continuous eigenfunction on $\tilde{X}$.  (Note that we only need the existence of a  point with a dense orbit.)

As $f$ is continuous on a compact space, so it is uniformly continuous. Thus for each $\varepsilon>0$ there is a $k$ such that for all $y,z\in X$, if $y_{[-k,k] }=z_{[-k,k] }$, then $|f(y)-f(z)|<\varepsilon$.
Now if $T^n(\tilde{x})_{[0,2k+1]}= T^m(\tilde{x})_{[0,2k+1]} $, then by definition  $T^{n+k}(x)_{[-k,k]}= T^{m+k}(x)_{[-k,k]}$, so that  one has $|f(T^{n+k}(x))-f(T^{m+k}(x))| < \varepsilon$. Therefore
\begin{align*}  |\tilde{f}(T^n(\tilde{x}))-\tilde{f}(T^{m}(\tilde{x}))|= |\lambda^k\tilde{f}(T^n(\tilde{x}))-\lambda^k\tilde{f}(T^{m}(\tilde{x}))|&=  |\lambda^{k+n}f(x)-\lambda^{k+m}f(x)|\\&=  |f(T^{n+k}(x))-f(T^{m+k}(x))| < \varepsilon,
\end{align*}
 and  $\tilde f$ is uniformly continuous on the orbit of $\tilde x$.

As $\tilde{\mu}$ and $\mu$ give the same mass to cylinders, 
the second statement then follows from \cite[Theorem 1]{Sz-Nagy-Foias-1958}, combined with the fact that  the unitary operator $f\mapsto f\circ T$ on $L^2(X,T,\mu)$ is the minimal unitary dilation of the isometric  operator on $L^2(\tilde{X},T,\mu)$ \cite[page 385]{Campbell-1986}.

\end{proof}
  
      A dynamical system $(Z,F)$ is called {\em equicontinuous} if the family $\{F^n : n\in\Z\}$ is equicontinuous. 
An equicontinuous system $(Z,F)$ which is minimal must be a  minimal rotation, that is, there is a continuous abelian group structure on $Z$ and an element $g\in  Z$ such that the homeomorphism $F$ is given by adding $g$, i.e.,  $F(z) = z+g$. Moreover, the orbit $\{ng: n\in\Z\}$ is dense in $Z$, i.e., $g$ is a topological generator of $Z$. Recall that $\pi: (X,T)\rightarrow (Z,F) $ is a {\em factor map} if it is continuous, onto, and $\pi\circ T  = F\circ \pi$.
If $(Z,+g)$ is equicontinuous and there is a  factor map $\pi :(X,T)\rightarrow (Z,+g)$, we say that $(Z,+g)$ is an equicontinuous factor of $(X,T)$. Any $\Z$-action $(X,T)$ admits a maximal equicontinuous factor $\pi:(X,T)\to (Z,+g)$.  
This equicontinuous factor must be maximal in the sense that any equicontinuous factor of $(X,T)$ factors through it, and it encodes all continuous eigenvalues of $(X,T)$.

Let $\sigma:\, \mathcal{A}^*\to \mathcal{A}^+$ be a morphism, also  called a \emph{substitution}. Note that  the image of any letter is a non-empty word. We will abuse notation and write  $\sigma:\, \mathcal{A} \to \mathcal{A}^+$.  Using concatenation, we extend $\sigma$ to~$\mathcal{A}^\mathbb{N}$ and~$\mathcal{A}^\mathbb{Z}$.  The finiteness of $\mathcal A$ guarantees that  $\sigma$-periodic points, i.e., points $x$ such that $\sigma^k(x)=x$ for some  positive $k$, exist.
The \emph{incidence matrix} of the substitution $\sigma$ is the $|\mathcal{A}| \times |\mathcal{A}|$ matrix~$M_\sigma=(m_{ij})$ with $m_{ij}$ being the number of occurrences of~$i$ in $\sigma(j)$.    A substitution is  \emph{primitive} if its incidence matrix admits  a  power  with positive entries. 
A substitution $\sigma$ is {\em left- (right-) proper} if for each letter $a$,  $\sigma(a)$ starts (ends) with the same letter, and it is {\em proper} if it is both left- and right-proper.
Given a substitution $\sigma:\mathcal{A} \rightarrow \mathcal{A}^+$, the  language   $\mathcal{L}_{\sigma}$    defined by  $\sigma$ is
\[
\mathcal{L}_{\sigma} = \big\{w \in \mathcal{A}^*:\, \mbox{$w$ is a subword of $\sigma^n(a)$ for some  $a\in\mathcal{A}$ and $n\in \mathbb N$}\big\}.
\]
If $\sigma$ is primitive, then each word in $\mathcal{L}_{\sigma}$ is left- and right-extendable, and $\mathcal{L}_{\sigma}$ is closed under the taking of subwords, so we  define
$X_{\sigma}:=X_{\mathcal L_\sigma}$.
We call $(X_\sigma, T)$ a \emph{substitution shift} and the language  of a  substitution shift is called a substitution language. The substitution $\sigma$ is \emph{aperiodic} if $(X_\sigma,T)$ is aperiodic.

\begin{definition}[Transition and return words]\label{def:return-word}

Let $\mathcal L$ be a language on $\mathcal A$, and let $a$, $b \in {\mathcal A}$.   
Suppose that the word $w\in \mathcal L$  starts with $a$ and is such that $wb$ belongs to $\mathcal L$. Then  $w$ 
  is  called a {\em  transition word} from $a$ to $b$.
 Furthermore
if $wb$ contains exactly one occurrence of each of $a$ and $b$,  then we say that $w$ is a {\em strict transition word} from $a$ to $b$. 
If $a=b$,  $w$ is also called  a  return word or a  strict return word to $a$.

\end{definition}

  For example,
  if $accd  b \in \mathcal L$,   then  $accd$ is a  strict transition word from $a$ to $ b$. Also, if $aa   \in \mathcal L$,   then  $a$ is a strict  return word  to $ a$.

 A subshift  $(X,T)$ is  {\em linearly recurrent} if there exists a constant $L$ such that for any word $w$ belonging to $\mathcal L_X$,  the length of any  strict
 return word to $ w$  is of length at most $L| w|$. Examples of linearly recurrent shifts are primitive substitution shifts \cite{Durand:00a}.

\subsection{$S$-adic shifts } \label{main_S_adic}
We  recall basic definitions concerning $S$-adic shifts. They 
are obtained by replacing the iteration of a single substitution by the iteration of a sequence of {\em morphisms}, which are defined like substitutions, except that letters in $\mathcal A$ are mapped to  non-empty words on a possibly different alphabet  $\mathcal B$.  In this article we restrict to the case of an $S$-adic shift defined over one alphabet $\mathcal A$, however in discussing other works, we refer to morphisms if that is their context.

Let $\boldsymbol{\sigma} = (\sigma_n)_{n\ge0}$ be a sequence of substitutions with $\sigma_n:\, \mathcal{A}\to \mathcal{A}^+$; we call   
 $\boldsymbol{\sigma}$ a {\em directive sequence}.
 For $0\leq n<N$, let 
\[
\sigma_{[n,N)} = \sigma_n \circ \sigma_{n+1} \circ \dots \circ \sigma_{N-1}.
\]
For $n\geq 0$, define  
\[
\tilde{\mathcal{L}}_{\boldsymbol{\sigma}}^{(n)} = \big\{w \in \mathcal{A}^*:\, \mbox{$w$ is a subword of $\sigma_{[n,N)}(a)$ for some $a \in\mathcal{A}$, for some $N>n$}\big\}.
\]

 Define
\[ X_{\boldsymbol{\sigma}}^{(n)}:=\{x\in \mathcal A^{\Z}: \mbox{ for each $k\leq \ell, x_{[k,\ell)}\in \tilde{\mathcal{L}}_{\boldsymbol{\sigma}}^{(n)}$} \}\]
and note that $  \mathcal{L}_{\boldsymbol{\sigma}}^{(n)}      := \{ w:  w \mbox{ appears as some subword in some } x\in X_{\boldsymbol{\sigma}}^{(n)}    \}$
 is a subset of   $\tilde{\mathcal{L}}_{\boldsymbol{\sigma}}^{(n)}$, but it can be a proper subset.
 We call  $(\mathcal{L}_{\boldsymbol{\sigma}}^{(n)})_{n\geq 0}$ the  sequence of languages associated to~$\boldsymbol{\sigma}$.

We say that ~$\boldsymbol{\sigma}$ is \emph{primitive} if for each $n\ge0$ there is an $N>n$ such that  the incidence matrix $M_{[n,N)}:= M_{\sigma_n}M_{\sigma_{n+1}}\dots M_{\sigma_{N-1}}$ of $\sigma_{[n,N)}$ is a positive matrix. Under the assumption of primitivity for $\boldsymbol{\sigma}$,
each word in $ \mathcal{L}_{\boldsymbol{\sigma}}^{(n)}$ is left- and right-extendable. Note that if  $\boldsymbol{\sigma}$ is primitive, then each letter appears in each language $\mathcal L^{(n)}_{\boldsymbol\sigma}$. If $\boldsymbol{\sigma}$ is primitive, then $(X_{\boldsymbol{\sigma}}^{(n)},T)$ is minimal  for all~$n$, 
by \cite[Lemma 7]{Durand:00a}, and     $\tilde{\mathcal{L}}_{\boldsymbol{\sigma}}^{(n)}= \mathcal{L}_{\boldsymbol{\sigma}}^{(n)}$ for each $n$.  As we only work with minimal shifts in this article, even if we do not always need to  work with primitive directive sequences, {\em we assume throughout that     $\boldsymbol{\sigma}$ is primitive,} and hence that         each letter in $\mathcal A$ appears in each language $\mathcal L^{(n)}_{\boldsymbol\sigma}$. Nevertheless, some of our results hold with the weaker assumption of minimality of   $(X_{\boldsymbol{\sigma}},T)$.

To abbreviate notation, we set $X_{\boldsymbol{\sigma}} = X_{\boldsymbol{\sigma}}^{(0)}$ and call $(X_{\boldsymbol{\sigma}},T)$ the \emph{$S$-adic shift} generated by the directive sequence ~$\boldsymbol{\sigma}$.
The directive sequence 
${\boldsymbol{\sigma}}$ is  {\em everywhere growing} if  for each $a \in \mathcal{A}$,  
 $|\sigma_{[0,n)}(a)|\rightarrow \infty$   as $n$ tends to infinity. If a directive sequence is primitive, then it is everywhere growing.
We also  say  that  ${\boldsymbol \sigma}$ is {\em strongly primitive} if there exists $r$ such that, for each $n$, 
$\sigma_{[n,n+r)}$ has a positive incidence matrix. 

  We consider  particular families of directive sequences.
\begin{definition}\label{def:directive-families} Let $\boldsymbol{\sigma}= (\sigma_n)_{n\geq 0}$ be a directive sequence. We say that    $\boldsymbol{\sigma}$ is 
\begin{itemize}
\item
  {\em finitary}   if there is  a  finite set $\mathcal S$ such that $\sigma_n\in \mathcal S$ for each $n$;
 \item
  {\em stationary }  if  there  exists a substitution $\sigma$ such that 
  $\sigma_n= \sigma$ for all $n$;
  \item
 {\em constant-length, with length sequence $(q_n)_{n\geq 0}$} if for each $n$, $\sigma_n$ has constant-length $q_n$
(note that the sequence $(q_n)$ is not necessarily constant);
\item
 {\em unimodular} if for each $n$,  the incidence matrix $M_n$  of   $\sigma_n$ is unimodular, i.e., $|M_n|=\pm 1$.
\end{itemize}
\end{definition}

\subsection{Recognizability}\label{subsec:recog}
We  first start with   the notion   of recognizability which expresses the idea of performing a  ``desubstitution''.
\begin{definition}[Dynamical recognizability, $\sigma$-representations  and recognizable directive sequences] \label{def:recog}
Let $\sigma:\, \mathcal{A} \to \mathcal{A}^+$ be a substitution and $y \in \mathcal{A}^\mathbb{Z}$.
If $y = T^k \sigma(x)$ with $x=(x_n)_{n\in \mathbb Z}
\in \mathcal{A}^\mathbb{Z}$, and  $0 \leq k < |\sigma(x_0)|$,  then we say that $(k,x)$ is a \emph{(centred) $\sigma$-representation} of~$y$. 
For $X \subseteq \mathcal{A}^\mathbb{Z}$, we say that the $\sigma$-representation $(k,x)$ \emph{is in~$X$} if $x \in X$.

Given a nonempty $X \subseteq \mathcal{A}^\mathbb{Z}$ and $\sigma:\, \mathcal{A} \to \mathcal{A}^+$, we say that $\sigma$ is \emph{recognizable in~$X$} if each $y \in \mathcal{A}^\mathbb{Z}$ has at most one centered $\sigma$-representation in~$X$. 
A~directive sequence $\boldsymbol{\sigma}$ is \emph{recognizable at level~$n$} if $\sigma_n$ is recognizable in~$X_{\boldsymbol{\sigma}}^{(n+1)}$. 
The sequence~$\boldsymbol{\sigma}$, or  the $S$-adic shift $X_{\boldsymbol{\sigma}}$ if the sequence $\boldsymbol{\sigma}$ is given, is \emph{recognizable} if each  $x \in X_{T^n( \boldsymbol{\sigma} )}$ has exactly one $\sigma_n$-centred representation in $X_{T^{n+1}(\boldsymbol{\sigma})}$.

\end{definition}

In this article, recognizability is a standing  assumption that is necessary for most of our results. Here, we are talking about two-sided dynamical recognizability.
Note that two-sided  recognizability does not  in general imply one-sided recognizability.  
 Also,   we are concerned with dynamical recognizability, and not Moss\'{e}'s combinatorial version  \cite{Mosse:92, Mosse:96} (based on cutting points for finite words), and while dynamical recognizability implies combinatorial recognizability \cite[Theorem 2.5]{BSTY}, the converse is false. 
 A two-sided aperiodic primitive substitution shift is recognizable in both senses, but not necessarily  one-sided recognizable \cite{Mosse:92, Mosse:96}.
 There even exist substitutions that are constant-length and injective  on letters, that are combinatorially one-sided recognizable, but not dynamically one-sided recognizable as in the following example provided to us by Dominique Perrin (see also \cite{BBPR}).

\begin{example}
Consider 
\begin{align*}
   a &  \stackrel{\sigma}{\mapsto}  ba\\
    b &  \stackrel{\sigma}{\mapsto}  aa. 
 \end{align*}
Since $\sigma$ is a primitive, aperiodic and constant-length substitution which is injective on the letters, it is one-sided recognizable in the  combinatorial sense of Moss\'e  \cite{Mosse:92}. However, the sequence $x=a\sigma(a)\sigma^2(a)\sigma^3(a)\dots \in \tilde{X}_\sigma$ has two different centered $\sigma$-representations in $\tilde{X}_{\sigma}$, namely $
x= T(\sigma(ax))=T(\sigma(bx))$, thus it is not dynamically one-sided recognizable. 
\end{example}

In what follows,
when we write that $\boldsymbol{\sigma}$ is recognizable, we mean that the corresponding two-sided shift is (dynamically) recognizable.

 A morphism $\sigma:\mathcal{A}\to \mathcal{C}^{+}$ is {\em elementary} if it cannot be written as $\sigma= \gamma \circ \delta$, with $\delta: \mathcal A \rightarrow \mathcal B^{+}$ and $|\mathcal B|<|\mathcal A|$. 
 A morphism $\sigma:\mathcal A\rightarrow \mathcal B^{+} $ is \emph{left- (right-) permutative} if the first (last) letters of $\sigma(a)$ and $\sigma(b)$ are different for all distinct $a,b \in \mathcal{A}$.   Two morphisms $\sigma, \tilde{\sigma}:\, \mathcal{A} \to \mathcal{B}^+$ are \emph{rotationally conjugate} if there is a word $w \in \mathcal{B}^*$ such that $\sigma(a) w = w \tilde{\sigma}(a)$ for all $a \in \mathcal{A}$ or $w \sigma(a) = \tilde{\sigma}(a) w$ for all $a \in \mathcal{A}$.
We state here  conditions on morphisms in a directive sequence $(\sigma_n)_{n\ge0}$  of morphisms, which guarantee recognizability.  This is one of the results on recognizability from \cite[Theorem 3.1]{BSTY},  with the more general result concerning elementary morphisms obtained earlier by 
Karhum\"{a}ki, Man\v{u}ch and Plandowski  in \cite{Karhumaki-Manuch-Plandowski:2003}, as discussed in \cite{Beal-Perrin-Restivo}. Note that there is a decision procedure to conclude whether a constant-length substitution generates an aperiodic fixed point  \cite{Allouche-Rampersad-Shallit-2009}.
Since we work with morphisms on a fixed alphabet, the case that concerns us is  $\mathcal{A}_{n} = \mathcal A$ for each $n$. 
\begin{theorem}\label{c:rec}
Let $\boldsymbol{\sigma} = (\sigma_n)_{n\ge0}$ be a primitive sequence of morphisms with $\sigma_n:\, \mathcal{A}_{n+1}\to \mathcal{A}_n^+$. Suppose that $X_{\boldsymbol{\sigma} }$ is aperiodic.
If each morphism $\sigma_n$ is elementary, then  $\boldsymbol{\sigma}$ is  recognizable. 
In particular, if 
 each morphism $\sigma_n$
satisfies one of 
\begin{itemize}
\item  $\mathrm{rk}(M_{\sigma_{n}}) = |\mathcal{A}_{n+1}|$, or
\item $|\mathcal{A}_{n+1}| = 2$, or
\item  $\sigma_n$ is (rotationally conjugate to) a left- or right-permutative morphism,
\end{itemize}
then $\boldsymbol{\sigma}$ is  recognizable. 
\end{theorem}

\begin{example}\label{ex:running-first}
We will use this running example to clarify  definitions and results.
Consider  the substitutions  $S=\{\sigma, \tau\}$ with
 \begin{align*}
   a   \stackrel{\sigma}{\mapsto}  acb  \ \  & a    \stackrel{\tau}{\mapsto}  abc
   \\  b  \stackrel{\sigma}{\mapsto}  bab \ \  & b  \stackrel{\tau}{\mapsto}  acb\\
    \ c  \stackrel{\sigma}{\mapsto}  cbc
 \ \ & c   \stackrel{\tau}{\mapsto} aac.
\end{align*}

We will work  throughout this  example with the directive sequence $ {\boldsymbol{\sigma}}$,  where
\[  {\boldsymbol{\sigma}}=\sigma, \tau, \sigma, \sigma, \tau, \sigma, \sigma , \sigma, \tau , \sigma, \sigma ,\dots \,. \]
This directive sequence is   introduced  by Durand in ~\cite{Durand:00b} 
 to produce   a  finitary  {\em strongly  primitive}  constant-length directive sequence  whose   associated shift $(X_ {\boldsymbol{\sigma}},T)$  is minimal,    {\em  not linearly recurrent}, and hence aperiodic,  and  has linear subword complexity\footnote{  The  {\em subword complexity} of  a subshift $(X,T)$ is the  function $ n \mapsto p_X(n)$  that counts the number of  words of length $n$ that
 belong to its language.}.  This directive sequence  is  also  recognizable, applying Theorem \ref{c:rec}  to the aperiodic shift $(X_ {\boldsymbol{\sigma}},T)$. To see this, note that $\sigma$ is left-permutative, and $\tau$ is rotationally conjugate to a left-permutative substitution: its middle column contains all letters, and its first column consists of one letter.

 \end{example}

The relevance of recognizability is that it gives us a framework within which to approximate our dynamical system, in terms of generating partitions, see Section \ref{sec:generating-partitions}. We end this section  with a statement of  \cite[Theorems 6.5 \& 6.7] {BSTY}, stated below as Theorem \ref{thm:Bratteli-Vershik}, which we will use when comparing our results to those in the literature, specifically those concerning Bratteli-Vershik systems and also finite-rank (cutting-and-stacking) systems. We assume that the reader is familiar with those systems, and refer them to  \cite[Section 6]{BSTY}, or \cite{BKMS:13}  for definitions and terminology concerning Bratteli-Vershik systems and  \cite{Ferenczi:1997} concerning finite rank systems,  recalling only a few definitions,  used mainly for Theorem \ref{thm:Bratteli-Vershik}.

A~\emph{Bratteli diagram} is an infinite graph $B=(V,E)$ such that the vertex set $V=\bigcup_{n\geq 0}V_n$ and the edge set $E = \bigcup_{n\geq 0} E_n$ are partitioned into pairwise disjoint, non-empty subsets $V_n$ and~$E_n$, where
\renewcommand{\theenumi}{\roman{enumi}}
\begin{enumerate}
\item
$V_0=\{v_0\}$ is a single point;
\item
$V_n$ and $E_n$ are finite sets;
\item
there exists a range map~$r:E\rightarrow V$ and a source map~$s:E\rightarrow V$ such that $r(E_n)= V_{n+1}$ and  $s(E_n)= V_n$ for each $n\geq 0$.
\end{enumerate}

A~finite or infinite sequence of edges~$(e_n)$ with $e_n\in E_n$ such that $r(e_n)=s(e_{n+1})$ is called a \emph{finite} or \emph{infinite path}, respectively. An {\it ordered Bratteli diagram} is a Bratteli diagram together with a linear ordering on $r^{-1}(v)$ for each $v\in V\setminus V_0$.
Given a directive sequence~$\boldsymbol{\sigma}$, we can define an associated \emph{natural}  ordered Bratteli diagram, as follows. 
For $n\geq 1$, $V_n$~is a copy of~$\mathcal{A}_{n-1}$. Since we work with a uniform alphabet, this means  that $V_n$ is fixed for each $n$.
Given a vertex $v \in V_{n+1}$ labelled by the letter $a \in \mathcal{A}_{n}$, we order the edges with range~$a$ as follows: if $b$ is the $j$-th letter in $\sigma_{n-1}(a)$, then we label an edge with source~$b$ and range~$a$ with $j{-}1$.  Such an ordered Bratteli diagram  allows the definition of a measurable
Bratteli-Vershik dynamical system (see  \cite[Section 6]{BSTY}). Note that the ordered Bratteli diagram associated to $\boldsymbol{\sigma}$ has much in common with prefix-suffix automata, with labels of edges in the $n$-level of the Bratteli diagram being replaced by prefixes of the words $\sigma_n(a)$; see for example \cite{Canterini-Siegel:2001}.

We say that a transformation $\Phi:\, (X,T) \to (Y,S)$ is an \emph{almost-conjugacy} (also called essential-conjugacy) if there is a $T$-invariant set $\mathcal{D} \subset X$ and an $S$-invariant set $\mathcal{E} \subset Y$, with $\Phi:\, X \setminus \mathcal{D} \to Y \setminus \mathcal{E}$ a continuous bijection satisfying $\Phi \circ T= S \circ \Phi$, and such that~$\mathcal{D}$  (resp.~$\mathcal{E}$)     has zero measure for \emph{every} fully supported invariant probability measure on $(X,T)$ (resp.~$(Y,S)$).
Thus, if $(X,T)$ and $(Y,S)$ are almost conjugate, $\nu$~is any fully supported probability measure on~$X$ that is preserved by~$T$,  $\mu$~is any  fully supported probability measure on~$Y$ that is preserved by~$S$,  and $\Phi$ maps $\nu$ to $\mu$, then $(X,T,\nu)$ and $(Y,S,\mu)$ are conjugate in measure. We can therefore apply \cite[Lemma 6.4]{BSTY} to obtain the following, which is \cite[Theorem 6.5, Corollary 6.7]{BSTY}. We refer to  \cite{Katok-Stepin-1967, Ferenczi:1997} for details on cutting-and-stacking measurable transformations.

\begin{theorem}\label{thm:Bratteli-Vershik}
Let $\boldsymbol{\sigma}$ be an everywhere growing recognizable  directive sequence defined on $\mathcal A$ and $(X_B, \varphi_\omega)$ the natural Bratteli-Vershik dynamical system associated to~$\boldsymbol{\sigma}$.  We assume that $(X_{\boldsymbol{\sigma}},T)$ aperiodic. Then\begin{itemize}
\item
 the system $(X_{\boldsymbol{\sigma}},T)$ is  almost-conjugate to $(X_B, \varphi_\omega)$, and
 \item
if
$|\mathcal A|=k$, and
$\nu$ is any 
fully supported
 probability measure such that $(X_{\boldsymbol{\sigma}}, T,\nu)$ is measure preserving,
then $(X_{\boldsymbol{\sigma}}, T,\nu)$ is measurably conjugate to a cutting-and-stacking measurable transformation of a Lebesgue space which is of measurable rank at most~$k$.
\end{itemize}
\end{theorem}

The model theorem of \cite{HPS:1992} tells us  that any  minimal $S$-adic shift  has a  proper
 Bratteli-Vershik  representation, i.e., a Bratteli-Vershik representation where $X_B$ has only one maximal and one minimal path, so that  the successor map $\varphi_\omega$ is a homeomorphism.    If in addition,  this proper   representation defines  a   recognizable directive sequence, 
then it  defines   a proper $S$-adic  representation. Note though that 
 this proper $S$-adic representation may no longer
 be finitary, even  if the original directive sequence is.   By choosing to work  with a  non-proper representation   we  take advantage of  the finitary nature of the directive  sequence. 

\subsection{Generating partitions}\label{sec:generating-partitions}

Substitution dynamical systems have received much attention in the past decades, and this is in no small  part due to the fact that aperiodic substitution shifts possess a     natural sequence of partitions $(\mathcal Q_n)$, consisting of   Rokhlin towers,   which generates in measure (see below for a definition). This is a consequence of the fact that 
 $\sigma$ is recognizable
   \cite{Mosse:92,Mosse:96, Bezugly:2009}.
We refer the reader to \cite{Fog02,Queffelec:10} for expositions of the basic aspects of these systems and all the undefined terms we use below.
This extends to the $S$-adic case and 
we  present 
a sequence of partitions $({\mathcal Q}_n)$  relying  here also on  the recognizability property, such as described  in Section \ref{subsec:recog}.

Given a primitive directive sequence~$\boldsymbol{\sigma}$ such that $X_{\boldsymbol{\sigma} }$ is aperiodic, and $n\in \mathbb N$, for a letter $a\in \mathcal A$ define
\[B_n(a):=\sigma_{[0,n)}([a])\]
 and for a word $w\in \mathcal A^{*}$, define \begin{equation}\label{eq:noth}
 h_n(w) := |\sigma_{[0,n)}(w)|. 
 \end{equation}
Let
\begin{equation}\label{eq:partisadic}
\mathcal{Q}_n = \{ T^k \sigma_{[0,n)}([a]):\,  0 \leq k<h_n(a) \}.
\end{equation}

Note that for all $n$, $\mathcal{Q}_n$ is a cover of $X_{\boldsymbol \sigma}$. Indeed, for every 
$\ell \in {\mathbb N}$
$$
X_{\sigma}^{(\ell)} =\{T^k \sigma_\ell  (x) : x \in  X_{\sigma}^{(\ell+1)}, 0   \leq    k< |\sigma_\ell (x_0 )|\}$$
(this equality is true without assuming recognizability), thus by iterating
$$
X_{\sigma} =\{T^k  \sigma_{[0,n)} (x): x  \in  X_{\sigma}^{(n)},  0   \leq   k<  h_n(x_0)\}.$$
Then, using the partition of  $X_{\sigma}^{(n)}$  by cylinders $  [a]$, $a \in  {\mathcal A}$, we get
that  $\mathcal{Q}_n $ is a cover.

We call   $\mathcal T_n(a):=\cup_{   0\leq k<h_n(a)      } T^k \sigma_{[0,n)}([a])$ the {\em  $n$-tower defined by $a$},  we call $T^k \sigma_{[0,n)}([a])$ the {\em $k$-th level} of this tower, and     $B_n(a)$  its {\em base}.   Thus the elements of $\mathcal{Q}_n$ can be arranged to form  $  |\mathcal A|$ towers. 
 If
${\boldsymbol{\sigma}}$ is everywhere growing,  then 
   the height $h_n(a)$ of each $n$-tower  $\mathcal T_n(a)$  increases to $\infty$ as $n$ grows.
   
    We say that a sequence $(k_n, a_n )_{n\in\mathbb N}$ is a \emph{$(\mathcal{Q}_n)$-address} for $x \in X_{\boldsymbol{\sigma}}$  if   $x\in T^{k_n} \sigma_{[0,n)}([a_n])$ for each $n\in \mathbb N$, where $0\leq k_n <  |\sigma_{[0,n)}( b_n)|$.
    Each point $x \in X_{\boldsymbol{\sigma}}$ has at least one $(\mathcal{Q}_n)$-address. 
    
Let $\mu$ be  a shift invariant probability measure on~$X_{\boldsymbol{\sigma}}$.
The covers $(\mathcal{Q}_n)_{n=0}^{\infty}$ are \emph{generating in  measure} if $\mu$-almost every $x \in X_{\boldsymbol{\sigma}}$ has a $(\mathcal{Q}_n)$-address that uniquely determines~$x$. 

We next identify the importance of recognizability in ensuring that $(\mathcal Q_n)$ is a sequence of partitions.
Indeed it is straightforward to check that if $\boldsymbol{\sigma}$ is recognizable, then each $\mathcal{Q}_n$ is a partition of $X_{\boldsymbol{\sigma}}$, i.e.,  each point has exactly one $(\mathcal{Q}_n)$-address. Note  also  that since $\sigma_n(X_{\boldsymbol{\sigma}}^{(n+1)})\subseteq X_{\boldsymbol{\sigma}}^{(n)}$, the sequence $(\mathcal{Q}_n)_{n=0}^{\infty}$ is \emph{nested}, i.e.,  for each $n\in\N$, every element of $\mathcal{Q}_{n+1}$ is a subset of some element of $\mathcal{Q}_n$. We summarise this as follows.

\begin{lemma}
 Let $\boldsymbol{\sigma}$ be a  sequence of substitutions defined on $\mathcal A$. If $\boldsymbol{\sigma}$ is recognizable,  then the sequence of covers $(\mathcal{Q}_n)_{n=0}^{\infty}$ defined in \eqref{eq:partisadic} is a nested sequence   of partitions of $X_{\boldsymbol{\sigma}}$. 
\end{lemma}
We will see in  Lemma \ref{partition-generates-mble} that the sequence  $(\mathcal Q_n)$ generates in measure, and  therefore  we can use it when we investigate measurable eigenvalues in Section \ref{measurable}.

Incidence matrices  of substitutions  allow a partial description of  towers. Indeed, the incidence matrix $M_{\sigma_{n-1}}= (m_{ij})$ for $\sigma_{n-1}$ gives us partial information about how elements from $\mathcal{Q}_{n}$ are built from  elements of $\mathcal{Q}_{n-1}$. For, recalling that   $m_{ab}$  equals the number of occurrences of the  letter $a$ in the image of the letter $b$ under $\sigma_{n-1}$, 
$\mathcal T_{n}( a)$  consists of $m_{ba}$ copies of a subtower of  $\mathcal T_{n-1}(b)$.  What the incidence matrix does {\em not} tell us is in which order we stack these subtowers;  the order of arrangement of the subtowers in $\mathcal T_{n}(a)$ is given by $\sigma_{n-1}(a)$. Indeed,  define
$ t_n(a,b): = \{ 0\leq t<  h_n(a): T^t (B_n(a))\subset B_{n-1}(b)\}.$ 
One has   $T^t (B_n(a)) \subset  B_{n-1}(b)$ if and only if $b$ occurs  
  as the $j$-th letter of $ \sigma_{n-1}(a)$ and $t=h_{n-1}(\sigma_{n-1}(a)_{[0,j)})$.  This implies that   $ t_n(a,b)\neq \emptyset$ if and only if $b$ appears in $\sigma_{n-1}(a)$,
and that the  cardinality of $t_n(a,b)$ equals the $(b,a)$ entry  $m_{ba}$ of $M_{\sigma_{n-1}}$.  See Figure \ref{fig:partitionQ}  for an illustration, and \cite[Chapter 6]{CANT} for an exposition.

\begin{figure}[h]

 \begin{center}
  \begin{tikzpicture}[scale=0.75]

   \draw [very thick] (6,0) rectangle (12,12);
   \draw [very thick] (6,1) -- (12,1);
   \draw [very thick] (6,4) -- (12,4);
   \draw [very thick] (6,5) -- (12,5);
   \draw [very thick] (6,6) -- (12,6);
   \draw [very thick] (6,7) -- (12,7);
   \draw [very thick] (6,8) -- (12,8);
   \draw [very thick] (6,9) -- (12,9);
   \draw [very thick] (6,11) -- (12,11);
   \draw [very thick] (6,12) -- (12,12);
      \node at (9,0.5) {$B_n(a)$};
   \node at (9,4.5) {$T^{h_{n-1}(c)-1}B_n(a)$};
   \node at (9,5.5) {$T^{h_{n-1}(c)}B_n(a)$};
   \node at (9,7.5) {$T^{h_{n-1}(cb)-1}B_n(a)$};
     \node at (9,8.5) {$T^{h_{n-1}(cb)}B_n(a)$};
   \node at (9,11.5) {$T^{h_{n-1}(cbd)-1}B_n(a)$};

   \draw [dashed,very thick] (6.1,0.1) rectangle (11.9,4.9);
   \draw [dashed,very thick] (6.1,5.1) rectangle (11.9,7.9);
   \draw [dashed,very thick] (6.1,8.1) rectangle (11.9,11.9);

   \node at (9, 2.5) {$\vdots$};
  \node at (9, 6.7) {$\vdots$};
  \node at (9, 10) {$\vdots$};

  \draw [->, >=stealth] (5.8,0.5) -- (5.8,1.5);
  \node [black] at (5.5,1) {$T$};
\draw [decorate,decoration={brace,amplitude=10pt,mirror,raise=4pt},yshift=0pt]
(12.2,0.2) -- (12.2,4.8) node [black,midway,xshift=1.5cm] {$\subseteq \cT_{n-1}(c)$};
  
\draw [decorate,decoration={brace,amplitude=10pt,mirror,raise=4pt},yshift=0pt]
(12.2,5.2) -- (12.2,7.8) node [black,midway,xshift=1.5cm] {$\subseteq \cT_{n-1}(b)$};

\draw [decorate,decoration={brace,amplitude=10pt,mirror,raise=4pt},yshift=0pt]
(12.2,8.2) -- (12.2,11.8) node [black,midway,xshift=1.5cm] {$\subseteq \cT_{n-1}(d)$};

\node (e0) at (5.7,0) {};
\node (e1) at (5.7, 5) {};
\node (e2) at (5.7, 12) {};

  \end{tikzpicture}
 \end{center}
 \caption{In this example, we  construct  $\mathcal T_n(a)$ with $\sigma_{n-1}(a)= cbd$, by concatenating portions of the towers of   level $n-1$ for $c$, $b$ and then $d$.
}\label{fig:partitionQ}
 
\end{figure}

\section{From combinatorics  to spectral theory}\label{sec:comspec}
We  first introduce in Section \ref{sec:straight} two key combinatorial  notions here, namely  the notions of  essential words and of  straightness   for  discussing limit words.
We then provide some families of examples in Section \ref{subsec:ex}  and   discuss  in Section \ref{subsec:ecs}
eigenvalues  and coboundaries for  stationary directive  sequences,  i.e., for substitutions.  Lastly, Section \ref{subsec:height} deals with the notion of height for constant-length substitutions.

\subsection{Limit words and straightness}\label{sec:straight}

We identify some distinguished points in $X_{\boldsymbol{\sigma}}$, namely limit words, which are analogues of substitution fixed points.

\begin{definition}[Essential and fully essential  words]
We say that a word $w$   is {\em essential} (for $ \boldsymbol{ \sigma} $) if it occurs in $\mathcal{L}_{\boldsymbol{\sigma}}^{(n)}$ for  infinitely many $n$. An essential word is {\em fully essential} if it occurs in  $\mathcal{L}_{\boldsymbol{\sigma}}^{(n)}$  for each $n$.
\end{definition}

\begin{example}
Consider  the substitutions
 \begin{align*}
   a  \stackrel{\tau_1}{\mapsto} aaaba,  \ \ &   a  \stackrel{\tau_2}{\mapsto}  bbbab
   \\ b \stackrel{\tau_1}{\mapsto}  ababa, \ \ &   b  \stackrel{\tau_2}{\mapsto}   babab.\\
\end{align*}
The words $ab$ and $ba$ are fully essential  for any directive sequence. The word $aa$ appears in  $\mathcal{L}_{\boldsymbol{\sigma}}^{(n)}$ if and only if 
$\sigma_n=\tau_1$, and   $bb$ appears in  $\mathcal{L}_{\boldsymbol{\sigma}}^{(n)}$ if and only if 
$\sigma_n=\tau_2$.
 Now take any directive sequence $ \boldsymbol{ \sigma} $ from the set $\{\tau_1, \tau_2\}$. Then $aa$ (respectively $bb$) is an essential word for $ \boldsymbol{ \sigma} $   if and only if we  see $\tau_1$ (respectively $\tau_2$) infinitely often in $ \boldsymbol{ \sigma}$.  Note that 
 $aa$ (respectively $bb$) appears in $\mathcal{L}_{\boldsymbol{\sigma}}^{(0)}$    if and only if
$\sigma_0=\tau_1$ (respectively $\sigma_0=\tau_2$).  In particular, essential words do not necessarily appear in  $\mathcal{L}_{\boldsymbol{\sigma}}^{(0)}$.
\end{example}

\begin{example}\label{ex:running-second}
  Recall the directive sequence \[ \sigma, \tau, \sigma, \sigma, \tau, \sigma, \sigma , \sigma, \tau , \sigma, \sigma ,\dots \, \]
 from Example~\ref{ex:running-first}.
It can be verified that $\{ab,ac,ba, bc, ca ,cb \}$ are all fully essential. 
For example, we show that $ca$ is fully essential. 
We  have  $ab  \in \mathcal{L}_{\boldsymbol{\sigma}}^{(k)}$ for each $k$. Therefore, if $\sigma_{k-1}= \tau$, we have $ca \in  \mathcal{L}_{\boldsymbol{\sigma}}^{(k-1)}$, since $\tau$ occurs as an isolated letter in the directive sequence $\boldsymbol{\sigma}$.  Also, if $ca\in  \mathcal{L}_{\boldsymbol{\sigma}}^{(k-1)}$  and $\sigma_{k-2}= \sigma$, then $ca \in \mathcal{L}_{\boldsymbol{\sigma}}^{(k-2)}$. 
Hence  $ca \in \mathcal{L}_{\boldsymbol{\sigma}}^{(j)}$ for all $j\leq k-1$. Since $\tau$ appears infinitely often, the claim follows.
 
 Otherwise, the word $bb$ is essential, but does not appear at any level 
  $\mathcal{L}_{\boldsymbol{\sigma}}^{(k)}$ where $\sigma_k=\tau$, and $aa$ is essential, only appearing in  $\mathcal{L}_{\boldsymbol{\sigma}}^{(k)}$ when $\sigma_k=\tau$.  Finally $cc$ does not belong to any   $\mathcal{L}_{\boldsymbol{\sigma}}^{(k)}$.

\end{example}

{\em Telescoping} a directive sequence $(\sigma_n)_{n\geq 0}$ means taking a sequence $(n_k)_{k\geq 1}$ and considering instead the directive sequence $(\tilde \sigma_k)$ where $\tilde \sigma_0 = \sigma_{[0,n_1)}$ and $\tilde \sigma_k = \sigma_{[n_{k},n_{k+1})}$ for $k\geq 1$. Telescoping a directive sequence does not change the dynamics, i.e.,    $X_{\boldsymbol{\sigma}}^{(0)} =  X_{\boldsymbol{\tilde \sigma}}^{(0)}$. Note that the telescoped sequence 
${\boldsymbol{\tilde  \sigma}}$ may not have the same set of essential words as  the original directive sequence, in particular it may have fewer essential words.

Suppose that  the word $ab$ is essential. If there is a sequence $(n_k)_{k\geq 1}$  such that for each $k$, $ab\in \cL_{\boldsymbol{\sigma}}^{(n_k)}$,  $\sigma_{[0,n_k)}(a)$  shares a common prefix with  $\sigma_{[0,n_{k+1})}(a)$ of  increasing length, and $\sigma_{[0,n_k)}(b)$  also shares a common suffix with $\sigma_{[0,n_{k+1})}(b)$ of increasing length, then the sequence of finite words $(\sigma_{[0,n_k)} (a)\cdot   \sigma_{[0, n_k)} (b))_{k\geq 1}$ converges to a bi-infinite sequence in $X_{\boldsymbol{\sigma}}$. We denote  it by $\lim_k \sigma_{[0,n_k)} (a)\cdot   \lim_k  \sigma_{[0, n_k)} (b).$ Here the indices to the right of the radix point $``\cdot"$ start at 0, i.e., $x\cdot y= \dots x_{-1} y_0 y_1 \dots$.   The same  convergence property holds similarly    for one-sided words.  Recall that we assume that     each letter in $\mathcal A$ appears in each language $\mathcal L^{(n)}_{\boldsymbol\sigma}$.

\begin{definition} [Limit word]\label{def:limit-word} We say that $x$ in the two-sided shift  $ X_{\boldsymbol{\sigma}}$ is a  {\em limit word} if $x= \lim_{k} \sigma_{[0,n_k)}(a)\cdot  \sigma_{[0,n_k)}(b)$ for some essential word $ab$ that belongs to $  \mathcal{L}_{\boldsymbol{\sigma}}^{(n_k)}   $ for each $k$.
We say that  an element $x$ in the one-sided shift $ \tilde{X}_{\boldsymbol{\sigma}}$ is a  {\em limit word} if $x= \lim_{k} \sigma_{[0,n_k)}(a)$ for some sequence $(n_k)$.
 \end{definition}

For each $n$, two-sided limit words of recognizable directive sequences belong to the base of some $n$-tower in the partition $\mathcal Q_n$  from (\ref{eq:partisadic}). 
Note that if the word $ab$ is essential, and $ \boldsymbol{ \sigma} $ is everywhere growing,  then compactness implies that  there is at least one sequence $(n_k)$ such that 
\[    \lim_k \sigma_{[0,n_k)} (a)\cdot   \lim_k  \sigma_{[0, n_k)} (b)\] belongs to $X_{\boldsymbol{\sigma}}$.
Different sequences $(n_k)$ may lead to different limit words; this motivates the definition of straightness below. We define 

\[
{\boldsymbol{ a}}\cdot {\boldsymbol{ b}}:= \{   u   \in {\mathcal A}^{\mathbb Z}: \exists (n_k) \mbox{ such that } 
ab\in  \mathcal{L}_{\boldsymbol{\sigma}}^{(n_k)} \mbox{ for each }k \mbox{ and }  u=    \lim_k \sigma_{[0,n_k)} (a)\cdot   \lim_k  \sigma_{[0,n_k)} (b)
\} 
\]

and similarly
 \[
 {\boldsymbol{ a}}:=\{   u \in {\mathcal A}^{\mathbb N}: \exists (n_k) \mbox{ such that }  u=   \lim_k  \sigma_{[0,n_k)} (a)\}.
\]


 If  ${\boldsymbol{ a}}\cdot {\boldsymbol{ b}}$  consists of a unique limit word, we use ${\boldsymbol{ a}}\cdot {\boldsymbol{ b}}$     to denote this   limit word.  Similarly if ${\boldsymbol{ a}}$ is a singleton we use it to denote that limit word.  We will work mainly with one-sided limit  words in the following.

We note that there are alternative definitions of limit words, such as discussed e.g.   in \cite[Section 4]{BSTY}. Indeed a second  and natural definition of a  limit word 
consists in  considering elements of   $\bigcap_{n\in\mathbb{N}} \sigma_{[0,n)}(\mathcal{A}^\mathbb{Z})$  (this is case for instance in 
\cite{Arnoux-Mizutani-Sellami:2014,Pytheas-2020}).
 Observe that this definition of a limit word can yield a different shift $X_{\boldsymbol{\sigma}}$ to ours, already in  the non-minimal substitutive case such as stressed in \cite{AubSab}. 
Indeed   take the constant sequence $\boldsymbol{\sigma}$ taking the constant value~$\sigma:\, 0 \mapsto 00,\, 1 \mapsto 11$  on the alphabet $\{0,1\}$. Then 
 $\dots0011\dots$ does not belong to~$X_{\boldsymbol{\sigma}}$, but does belong to  $\bigcap_{n\in\mathbb{N}} \sigma^n (\{0,1\}^\mathbb{Z})$. However, we do have the following lemma, which partly links these two notions of a limit word.
 
 \begin{lemma}
Let $\boldsymbol{\sigma}$ be an everywhere growing directive sequence defined on  the alphabet ${\mathcal A}$. 
 If $ u^{(0)}\in \bigcap_{n\in\mathbb{N}} \sigma_{[0,n)}\left( X_{\boldsymbol{\sigma}}^{(n)} \right)$,  then there exist $a,b \in \mathcal A$ such that 
  $u^{(0)}\in \boldsymbol{ a\cdot b}$.
  Conversely, 
   if $u^{(0)} \in \boldsymbol{ a\cdot b}$ for $a,b \in \mathcal A$,  then there is a sequence $(m_k)$ such that  $u^{(0)}\in \bigcap_{k} \sigma_{[0,m_k)}\left( X_{\boldsymbol{\sigma}}^{(m_k)} \right)$.
 \end{lemma}
          \begin{proof}
By hypothesis, for each $n$ there exists  $u^{(n)}\in  X_{\boldsymbol{\sigma}}^{(n)}$ such that $u^{(0)}=\sigma_{[0,n)}( u^{(n)}   )$.
           We can pass to a subsequence $(n_k)$, so that there exist letters $a,b$ with $ab$ such that 
          $u^{(n_k)}_{[-1,0]}= ab$ in $ \mathcal{L}_{\boldsymbol{\sigma}}^{(n_k)}$ for each $k$. In particular  $ab$ is essential.
                 The first part of the  result follows using the assumption that $\boldsymbol{\sigma}$ is everywhere growing.
                                                              
Conversely, if $u^{(0)} \in \boldsymbol{ a\cdot b}$,  then by definition, there is a sequence $(n_k)$ such that $ab$  belongs to   $\mathcal{L}_{\boldsymbol{\sigma}}^{(n_k)}$ for each $k$ and   $u^{(0)} =    \lim_k \sigma_{[0,n_k)} (a)\cdot   \lim_k  \sigma_{[0,n_k)} (b)$. 
                 Let   $u^{(n)}=  \lim_k \sigma_{[n,n_k)} (a)\cdot   \lim_k  \sigma_{[n,n_k)} (b)$. One has $u^{(n)} \in   X_{\boldsymbol{\sigma}}^{(n)}$. The result follows. 
                           \end{proof}
          
The next lemma shows that there are finitely many  limit words for the systems we consider.
\begin{lemma}\label{lem:finitely many limit words}
 Let $\boldsymbol{\sigma}$ be   a 
  sequence of substitutions defined on $\mathcal A$. Then for each 
 letter, the set ${\boldsymbol{ a}}$  contains at most $|\mathcal A|$ points, and for each
  essential word $ab$, the set 
 ${\boldsymbol{ a}}\cdot { \boldsymbol{ b}}$ contains at most $|\mathcal A|^2$ points.
\end{lemma}

\begin{proof}

 Fix $a\in \mathcal A$.  Consider an element $ \lim_k \sigma_{[0,n_k)} (a)$ of $\boldsymbol{ a}$. For each $m$, if $n_k>m$,  the word $\sigma_{[0,n_k)}(a)$  admits as a prefix one of the $|\mathcal A|$ words $\sigma_{[0,m)}(\alpha)$, $\alpha \in \mathcal A$. Thus there is a set of one-sided sequences of cardinality at most $|\mathcal A| $, such that whenever  $n_k\rightarrow \infty$, $\lim_{n_k\rightarrow \infty}\sigma_{[0,n_k)}(a)$ exists and is a sequence, then it must belong to this set. Hence $\boldsymbol{ a}$ contains at most $|\mathcal A| $ points. Similarly,   ${\boldsymbol{ a}}\cdot { \boldsymbol{ b}}$ contains at most $|\mathcal A| ^2$ points.

\end{proof}
\begin{example}
Consider the Thue-Morse substitution 
\begin{align*}
   a &  \stackrel{\sigma}{\mapsto}  ab\\
    b &  \stackrel{\sigma}{\mapsto}  ba. 
 \end{align*}
Then  ${\boldsymbol{ a}}\cdot {\boldsymbol{ b}}$ consists of two points, both of which equal the $\sigma$-fixed point starting with $b$ on the right; however on the left we can see 
 the left infinite words  fixed by  $\sigma^2$   defined either  by  $a$,  or by  $b$, i.e.,
 \[         {\boldsymbol{ a}}\cdot {\boldsymbol{ b}}= \{\lim_{k\rightarrow \infty} \sigma_{[0,2k)}(a) \cdot   \lim_{n\rightarrow \infty} \sigma_{[0,2k)}(b),   \lim_{k\rightarrow \infty} \sigma_{[0,2k+1)}(a) \cdot   \lim_{n\rightarrow \infty} \sigma_{[0,2k+1)}(b)    \}.  \]
 \end{example}

A  necessary condition for most of  our results is that the directive sequence be {\em straight} (see Definition \ref{def:straight} below and  Definition  \ref{def:strongly-straight} for a  stronger notion). The terminology is  borrowed   from \cite{shimomura:20}, where it is defined for  directive sequences which are stationary.
 We thus  first   state  it  for substitutions
by introducing the  notion of   strong       straightness (see Definition \ref{def:straight-substitution}).   It  corresponds to  the  case where  the initial period equals  $1$  in   \cite[Section 2.1]{Host:1986}, which  allows us to  work with $p=1$ in the statement of Host's  theorem  from   \cite[Section 1.4]{Host:1986}.

\begin{definition}[Strongly straight substitution] \label{def:straight-substitution}
A  substitution $\sigma$ is {\em strongly straight} if  it is primitive and  for every letter $a$,  whenever $b$ is  the first letter of  $\sigma(a)$, then $\sigma(b)$ starts with $b$.

 \end{definition}

 \begin{example}\label{ex:straight-not-strongly}
 The  substitution $\tau_1$ below is strongly straight.  However $\tau_2$ is not, as $\tau_2(c)$ starts with $a$ but $\tau_2(a)$ does not start with $a$:
 \begin{align*}
   a   \stackrel{\tau_1}{\mapsto}  bca  \,\,\,\,\hspace{1em}  & a    \stackrel{\tau_2}{\mapsto}  bc
   \\  b  \stackrel{\tau_1}{\mapsto}  bbc \,\,\,\, \hspace{1em} & b  \stackrel{\tau_2}{\mapsto}  bbc\\
    \ c  \stackrel{\tau_1}{\mapsto} cac
 \,\,\,\, \hspace{1em}& c   \stackrel{\tau_2}{\mapsto} abc.
\end{align*}
\end{example} 
 Every  primitive substitution has a  power which is strongly  straight. As a  shift defined by a primitive  substitution $\sigma$ equals the shift defined by any of its powers, then we may assume, by taking a power of $\sigma$ if needed, that it is strongly straight.  
If $\sigma$ is strongly straight, then any one-sided $\sigma$-periodic point (i.e., any one-sided fixed point for  some power $\sigma^n$,  with $n$ positive)   is a   fixed point  of  the substitution $\sigma$ and 
for every letter
$a$,   $\lim \sigma^n (a)$ exists and is  a fixed point of $\sigma$.

Definition \ref{def:straight-substitution}  
 should be compared to the stronger Definition \ref{def:strongly-straight} that we introduce later for directive sequences.
  The next definition  for directive sequences is a less restrictive notion than that of strong straightness in the case that the directive sequence is stationary.
By this we mean that if $\sigma$ is strongly  straight, then it is straight,  and  that   straight substitutions  are not necessarily   straight. For instance, the substitution $\tau_2$ of Example \ref{ex:straight-not-strongly} is straight but not strongly straight.

\begin{definition}[Straightness] \label{def:straight}
A  directive sequence  $(\sigma_n)_{n\geq 0}$   is   {\em straight} if it is primitive and for each letter $a\in \cA$, the  set   $ \boldsymbol{ a}$ is a singleton. In other words,   $(\sigma_n)_{n\geq 0}$   is {\em straight} if  and only if   $\lim_{n\rightarrow \infty}\sigma_{[0,n)}(a)$ exists for each letter $a\in \cA$.
\end{definition}

Consider the bases $B_n(a)$  of the towers of  the generating  partition ${\mathcal Q_n}$  from Section \ref{sec:generating-partitions}. If $(\sigma_n)_{n\geq 0}$ is straight, then the right-infinite  part of the  two-sided word $\cap_{n_k\geq 0} B_{n_k}(a ) $ is equal to $ \boldsymbol{ a}  $  for any  increasing sequence $(n_k)_k$, and this will be useful when we  build eigenfunctions as limits of simple functions in later sections.


Note that we can always telescope a primitive sequence $(\sigma_n)_{n\geq 0}$ to a sequence \[\tilde\sigma_0:=\sigma_{[0,n_1)}, \tilde \sigma_1:= \sigma_{[n_1,n_2)}, \tilde\sigma_2 := 
\sigma_{[n_2,n_3)}, \dots\]
which is straight. However, if ${\bf \sigma}$  is finitary, we may lose this property when we telescope to  $(\tilde\sigma_n)$, and some of our later results are sensitive to this condition.    Take for example  two  substitutions $\sigma, \tau$ such that  $\sigma(a)$ starts with $b$, $\sigma(b)$ starts with $a$,  $\tau(a)$ starts with $a$ and $\tau(b)$ starts with $b$, and take ${\boldsymbol \sigma}= (\sigma_n)_{n\geq 0}$ where $\sigma_n= \tau$ except when $n\in \{2^k:k\in \N\}$, in which case $\sigma_n= \sigma$. Assume ${\boldsymbol \sigma}$ primitive  and ${\boldsymbol a}\neq {\boldsymbol b}$. In this case, to obtain straightness, we must telescope to levels $n_k$ where $n_{k+1}-n_{k}\rightarrow \infty$,  and the resulting directive sequence is not finitary.  
 
We say that  $(\sigma_n)_{n\geq 0}$
can be {\em boundedly telescoped} to a straight sequence  if we can telescope  via a syndetic sequence $(n_{k})$  
to a straight directive sequence, where a sequence is  syndetic if  its letters occur with bounded gaps.

\subsection{Examples}\label{subsec:ex}
We now illustrate the  notions of straightness and limit words  with   examples of families  of substitutions, 
namely  the Cassaigne-Selmer, Arnoux-Rauzy, Jacobi-Perron and Brun substitutions, which  are associated to
 multidimensional continued  fractions and  are used to  construct  shifts  with prescribed  eigenfunctions (see e.g. \cite{BST:20}). 
For all these families of substitutions, straightness usually holds. 

\begin{example}\label{ex:running-third}

Consider the directive  sequence
 \[ \boldsymbol{ \sigma}= \sigma, \tau , \sigma, \sigma, \tau,  \sigma, \sigma , \sigma, \tau,  \sigma, \sigma,\dots \]
 from Example~\ref{ex:running-second}.
 We have 
     \[   {\boldsymbol{ a}}=   {\boldsymbol{ b}}= {\boldsymbol{ c}}           = \left\{     \lim_{n\rightarrow \infty} \sigma_{[0,n)}(a)    \right\},\]   
so that  $\boldsymbol{\sigma}$ is  straight.

Similarly, any directive sequence from $\{\sigma, \tau \}$ where there are    infinitely many occurrences of the substitution $\tau$ and bounded runs of  consecutive occurrences of either the substitution $\sigma$, or the substitution $\tau$,
can be boundedly telescoped to a straight sequence.  

 Note that \[u=\lim_{n\rightarrow \infty} \sigma_{[0,n)}(b) \cdot \lim_{n\rightarrow \infty} \sigma_{[0,n)}(a), \mbox{ and } v=\lim_{n\rightarrow \infty} \sigma_{[0,n)}(c) 
 \cdot \lim_{n\rightarrow \infty} \sigma_{[0,n)}(a)\]
are distinct  two-sided limit words,
and  this directive sequence cannot be telescoped to a  directive sequence with a unique two-sided limit word. In other words, although $(X_{\boldsymbol \sigma},T)$ is guaranteed a proper Bratteli-Vershik representation, by the model theorem of  (Thm 4.7, \cite{HPS:1992}) for minimal systems, we do not know how to  obtain such a representation by simple manipulations of the data we are given.

\end{example}

\begin{example}
Consider the Cassaigne-Selmer substitutions (discussed e.g. in  \cite{CLL:17,CLL:21,BST:20})
\begin{align*}
  & a   \stackrel{\gamma_1}{\mapsto}  a   \hspace{1.3em} a    \stackrel{\gamma_2}{\mapsto}  b
   \\&  b  \stackrel{\gamma_1}{\mapsto}  ac \hspace{1em}b  \stackrel{\gamma_2}{\mapsto}  ac\\
    &c  \stackrel{\gamma_1}{\mapsto}  b
 \hspace{1.5em} c   \stackrel{\gamma_2}{\mapsto} c.
\end{align*}
A directive sequence from $ \{ \gamma_1, \gamma_2\}$ generates an aperiodic shift as soon as it is primitive, by \cite[Proposition 6]{CLL:17}.
As $\gamma_1$ is right-permutative and $\gamma_2$ is left-permutative, any primitive directive sequence from $\{ \gamma_1, \gamma_2\}$ is recognizable. We consider now a  primitive directive sequence $(\sigma_n)_n$.   By  \cite[Lemma 1]{CLL:17}, for all $N$, there exists  $i$ such that $\sigma_{N +2i}\neq  \sigma _{N+2i+1}$. In particular $ \gamma_1$ and  $ \gamma_2$  appear infinitely often,  and this implies that there  is only one  one-sided limit word; hence the directive sequence is straight, as every   $\boldsymbol{\alpha}$ consists of that unique limit word. To see this, we will show that the unique right-infinite limit word is $\lim_n\sigma_{[0,n)}(a)$. Every time $\sigma_n=\gamma_1$, the first letter of  $\sigma_{[n,N)}(a)$ and $\sigma_{[n,N)}(b)$ equals $a$, while  $\sigma_{[n,N)}(c)$ starts with $a$ or $b$.  Furthermore if $N$ is large enough so that  $\sigma_{[n,N)}$ contains two occurrences of $\gamma_1$, then  $\sigma_{[n,N)}(c)$ also starts with $a$. Now let $n$ grow, this means that each of  $\sigma_{[0,N)}(\alpha)$ starts with  $\sigma_{[0,n)}(a)$, where $n\rightarrow \infty$ as $N\rightarrow \infty$. The claim follows.
 \end{example}

 Therefore in the case where $\gamma_1$ and $\gamma_2$ each appear infinitely often the above example gives directive sequences whose corresponding $S$-adic systems have one left-infinite limit word and one right-infinite limit word, i.e.,  the corresponding Bratteli-Vershik diagrams have one maximal and one minimal path, although the substitutions are not proper.  Thus directive sequences with one unique left-infinite limit word and one unique right-infinite limit word are not in general proper. Furthermore, depending on the given  directive sequence, one may have to telescope unboundedly  in order to satisfy needed conditions in articles which consider eigenvalues of such systems via their Bratelli-Vershik representation, such as   \cite{CDHM:2003, Bressaud-Durand-Maass:2005, Bressaud-Durand-Maass:2010, Durand-Frank-Maass:2015}, which require that every substitution is proper.

\begin{example}\label{ex:AR}
Consider the Arnoux-Rauzy substitutions on the three-letter alphabet  $\{a,b,c\}$ (see \cite{Arnoux-Rauzy:91}):
\begin{align*}
  & a   \stackrel{\alpha_1}{\mapsto}  a   \hspace{1.3em} a    \stackrel{\alpha_2}{\mapsto}  ba \hspace{1.3em} a    \stackrel{\alpha_3}{\mapsto}  ca
   \\&  b  \stackrel{\alpha_1}{\mapsto}  ab \hspace{1em}b  \stackrel{\alpha_2}{\mapsto}  b  \hspace{1.7em} b  \stackrel{\alpha_3}{\mapsto}  cb\\
    &c  \stackrel{\alpha_1}{\mapsto}  ac
 \hspace{1em} c   \stackrel{\alpha_2}{\mapsto} bc  \hspace{1.2em} c    \stackrel{\alpha_3}{\mapsto}  c.
\end{align*}
Consider a directive sequence  where   each of the three substitutions occurs infinitely often; then    the directive sequence is primitive, the shift is aperiodic (also for subword complexity reasons, as in the previous example), and as each substitution in  $\{\alpha_1,\alpha_2, \alpha_3 \}$ is  right-permutative,  the directive sequence is recognizable. 
Furthermore  any primitive directive sequence is straight since any product of length 3 of these substitutions is proper. For more on  eigenvalues
 of Arnoux-Rauzy  shifts, see   \cite{Cassaigne-Ferenczi-Messaoudi:08}.

 \end{example}

\begin{example} \label{ex:JP} 
Similarly, the infinite  family of {\em Jacobi-Perron substitutions}  $\{ \sigma_{jk}\}_{0\leq j\leq k, k\neq 0}$ \cite[Equation (6.6)]{BST:20}, where $\sigma_{jk}$ is defined by 
 \begin{align*}
  & a   \stackrel{\sigma_{jk}}{\mapsto}  b     \\&  b  \stackrel{\sigma_{jk}}{\mapsto}  c \\
    &c  \stackrel{\sigma_{jk}}{\mapsto}   ab^jc^k,
\end{align*}
is left-permutative and so the directive sequence generates a recognizable shift, provided the shift is aperiodic. A primitive directive sequence is straight
  and here we only need to telescope any directive sequence to every third level to achieve this. For, if we telescope in this way, then the composed substitutions 
 $   \sigma_{j_1k_1}\circ \sigma_{j_2k_2} \circ \sigma_{j_3k_3}     $ will be  such that for each each $\sigma$ and each letter $a$,  $\sigma(a)$ starts with $a$.  \end{example}
 
 \begin{example}  \label{ex:Brun}
  Almost the
  same can be said of the family of (unordered) {\em Brun substitutions } (see \cite{Delecroix-Hejda-Steiner} and  \cite[Equation (6.7)]{BST:20}), defined on an arbitrary alphabet as 
   \begin{align*}
   j  \stackrel{\sigma_{ij}}{\mapsto}  ij    , k \stackrel{\sigma_{ij}}{\mapsto}  k \mbox{ for $k\neq j$}  .
\end{align*}
   Each substitution is right-permutative, so any primitive  directive sequence that generates an aperiodic shift is recognizable. And, provided that each substitution appears infinitely often, there is only one right-infinite limit word. 
  \end{example}

\subsection{Eigenvalues and coboundaries for substitution shifts }\label{subsec:ecs}
We recall that  $ \mathbb S^{1}$ stands for  the unit circle in the complex plane. We say that $h: {\mathcal A}^*\rightarrow \mathbb S^{1}$ is a {\em morphism}   if $h(w_1w_2)=h(w_1)h(w_2)$ whenever $w_1$, $w_2$ and $w_1w_2$ belong to ${\mathcal A}^*$.

Host \cite{Host:1986} shows that for  the class of primitive substitution shifts, every measurable eigenvalue is continuous; we summarise his approach below, which is based on the notion of   coboundaries.
 
 \begin{definition}[Substitutive  coboundary] \label{subst-coboundary-combinatorial-definition}
 Let $\sigma$  be a substitution on  a  finite  alphabet ${\mathcal A}$.
A morphism $h: {\mathcal A}^*  \rightarrow \mathbb S^{1}$ is a {\em  coboundary} for $\sigma$ if,  for any $a\in \mathcal A$, $h(w)=1$ whenever $w$ is a return word to $a$.
\end{definition}
The  coboundary $h$ is said to be {\em trivial} if  $h(a)=1$ for each  $a\in\mathcal A$.

The next lemma  revisits Definition \ref{subst-coboundary-combinatorial-definition} in terms of two-letter words.

\begin{lemma} \cite{Host:1986}\label{f-h-lemma-one}
Let  $(\tilde{X}_{\sigma},T)$ be a  one-sided substitution shift, $\tilde{X}_\sigma\subset \mathcal A^{\N}$,  with $\sigma$ primitive.  A morphism $h$ is a  coboundary on $\mathcal L_\sigma$ if and only if
  there exists a function
 $\bar{f}: \mathcal A\rightarrow \mathbb S^{1} $ such that
 \begin{equation}\label{f-h-one}
   \mbox{ for every two-letter  word  $ab\in\mathcal L_\sigma$,    one has  $\bar{f}(b)=\bar{f}(a) h(a)$}.
  \end{equation} 
  \end{lemma}

\begin{proof} 
  Suppose first that $h:\mathcal A\rightarrow {\mathbb S}^1$ is a morphism and $\bar{f}:\mathcal A \rightarrow {\mathbb S}^1$ is a function where \eqref{f-h-one} is satisfied. Then for any letter $a$, if $au_1 \dots u_{n-1} u_n$ is a return word to $a$, we have
  \begin{align*}
 \bar{f}(a)&=   \bar{f}(u_n)   h(u_n)\\
 &=  \bar{f}(u_{n-1})   h(u_{n-1}) h(u_n)\\
 &\vdots \\
 &=  \bar{f}(a) h(a) h(u_1)\dots   h(u_{n-1}) h(u_n),
  \end{align*}
  so that $h(au_1\dots u_{n-1}u_n) =1 $, i.e., $h$ is a one-sided coboundary.  
Conversely, suppose that $h$ is a one-sided coboundary.
  Since $(\tilde{X}_{\sigma},T)$ is minimal, we have $\mathcal L_\sigma=\mathcal L_u$ for any  $u\in \tilde{X}_{\sigma}$, so fix such a $u$.  By assumption, for any two indices $m < n$ such that
 $u_n=u_m$, we have  $h(u_m) \cdots h(u_{n-1})=1$.  
We  define a function  $ g\colon {\mathbb Z} \rightarrow  {\mathbb S^{1}} $  as 
\begin{align*}
g(k):=
        \begin{cases}
         h(u_0)\cdots h(u_{k-1}) & \mbox{  if $ k > 0$} \\ 
        1  & \mbox{  if $ k = 0$}   \\
           h(u_{k})^{-1 }\cdots h(u_{-1})^{-1}  &\mbox{ if $k < 0$}.   \\
        \end{cases}
\end{align*}
   We then define $\bar{f}$ as $\bar{f}(a):=g(k)$  if  $k$ is such that 
   $u_k=a$.  One checks that the map $\bar{f}$ is well defined and satisfies \eqref{f-h-one}. Note that the map $\bar{f}$ depends on the choice of $u_0$, with 
$ \bar{f}(u_0)=h(u_0)$.
\end{proof}

\begin{remark}
Note that if,  for some $a\in \mathcal A$, $aa\in \mathcal L_\sigma$, then $h(a)\bar{f}(a)=\bar{f}(a)$, so $h(a)=1$. In particular, for any non-trivial substitution shift on a two-letter alphabet, any coboundary $h$ satisfies  $h\equiv1$.
\end{remark}

We now  recall  the  seminal result by  Host  originally stated in \cite{Host:1986} for primitive substitutions.
It  is phrased in terms of  one-sided shifts and
we state it only for strongly  straight substitutions $\sigma$, as any primitive  substitution has a power which is strongly  straight. Recognizability of  $X_\sigma$ is a key requirement. Note that Host assumed that $\sigma$ is injective on letters, but this may be relaxed. For, if $\sigma$ is not injective on letters, we can introduce an equivalence relation $\sim$ on $\mathcal A$ where $a\sim b$ if and only if $\sigma(a) = \sigma(b)$.  We then work with   $\tilde{\sigma}$ defined on $\mathcal A/\sim$,  and if $\sigma$  is recognizable then so is  $\tilde\sigma$. Finally  $(X_\sigma,T)$ and $(X_{\tilde\sigma},T)$  are topologically conjugate. See \cite{Blanchard-Durand-Maass} for details. Recall that  a primitive substitution shift is uniquely ergodic; we use $\mu$ to denote the unique invariant measure of  such a  shift.

\begin{theorem}[\cite{Host:1986}]\label{Host}
Let $\sigma$ be a primitive substitution on the alphabet $\mathcal A$ which is strongly straight.  Suppose that the one-sided shift $(\tilde{X}_\sigma, T,\mu)$ is recognizable.
Let $ h_n(a)= |\sigma^n( a)|$ for all $n$ and all $a\in \mathcal A$. 
If for each $a\in \mathcal A$ the limit \begin{equation}\label{condition1}h(a):=\lim_{n\rightarrow \infty}\lambda^{h_n(a)}\end{equation} exists and defines a  coboundary $h$, then $\lambda$ is a continuous  (and hence measurable) eigenvalue of $(\tilde{X}_\sigma, T)$. Conversely, if  $\lambda \in {\mathbb S}^1$ is a measurable eigenvalue of $(\tilde{X}_\sigma, T,\mu)$, then it  also satisfies \eqref{condition1} for some  coboundary $h$. 
\end{theorem}

We   give an  intuition for the proof of  Theorem \ref{Host} in  the continuous  case.
If $f$ is a continuous eigenfunction for  the eigenvalue $\lambda$, taking values in $\mathbb S^1$, then the assumption that $\sigma$ is strongly  straight allows us to 
define $\bar f:\mathcal A\rightarrow \mathbb S^1$ by $\bar f(a):=f(u)$, where $u$ is the one-sided fixed point such that  $\sigma^n(a)\rightarrow u$.  
If $ab\in \mathcal L_\sigma$ is such that $u\in[ab]$ (where here it can happen that $a=b$), then by
continuity of the eigenfunction $f$, one gets    \[\bar f(a)\lambda^{h_n(a)}\rightarrow \bar f(b),\]
which implies the existence of $  \lim_n \lambda^{h_n(a)}$;  this yields   a well-defined function $h:\mathcal A\rightarrow \mathbb S^{1}, \ a \mapsto \lim_n \lambda^{h_n(a)}$.
 Further, if  $a b_1 \dots b_k\in \mathcal L_\sigma$ is a   return word to $a$,
  then
   since $T^{h_n(w)}u \rightarrow u$, we have
 \[\lambda^{h_n(a)+ h_n(b_1) + \dots + h_n(b_k)
  }\rightarrow 1, \mbox{ i.e., } h(a)h(b_1) \dots h(b_k)=1.\]
  This kind of argument leads to the definition of a coboundary.

Conversely, given a  coboundary $h$ and a function $\bar f$ guaranteed by Lemma \ref{f-h-lemma-one},
we can use $\bar f$ to  define a map  $f$ on the shift orbit of a fixed point $u$, as $f(T^n u):=\lambda^n \bar f(u_0)$.
Note that   if $u$ has a dense orbit, the definition of $f$
 by $f(T^nu) = \lambda$  completely determines the continuous function, if it extends by continuity.  From the fact that $M_\sigma^n/\lambda^n$ converges geometrically to the projection on the Perron-Frobenius eigenspace defined by $\lambda$, we can deduce that the convergence  of \eqref{condition1}  is fast enough, so  that we can extend $f$ to a continuous eigenfunction on $X_\sigma$. For details, see \cite [Lemma 5, Proposition 1]{Host:1986}.

 To complete the proof of the theorem, Host shows  that any measurable eigenvalue must also define a coboundary.   Suppose that $f\in L^2(X_\sigma,T,\mu)$ is a measurable eigenfunction for  the eigenvalue $\lambda$. Let $B_n(a)= \sigma_{[0,n)}([a])$ be the base of the $\sigma^n$-tower defined by the letter $a$. The restriction $f|_{\mathcal T_n(a)}$  of $f$ to 
 this tower is completely determined by $
f|_{B_n(a)}$.
 Let $\beta_{n,a}:=\E_n(f)|_{B_n(a)}$ be the value  that the  conditional expectation of $f$, conditioned on the partition ${\mathcal Q}_n$, assumes on $B_n(a)$. Then if $ab\in \mathcal L_\sigma$, once again we should have
\begin{equation*}\label{approximate-coboundary}\beta_{n,a}\lambda^{h_n(a)}\approx \beta_{n,b}, \end{equation*}
and 
\begin{equation*}\lambda^{h_n(a) + h_n(b_1) + \dots + h_n(b_k)}\approx 1\end{equation*}
for large $n$, whenever $ab_1\dots b_k$ is a return word to $a$. The arguments are more delicate, but here again it can be shown that the heights of the towers end up defining a coboundary. We revisit this in more detail in Section \ref{measurable}.

The definition of a coboundary is purely combinatorial. Let us see  that more  is needed in order to get eigenvalues with   the  next example. \begin{example}\label{example 1}
There  exists $h$  which satisfies the conditions of Definition \ref{subst-coboundary-combinatorial-definition} without defining an eigenfunction. In particular, for a coboundary to define an eigenvalue, Equation \eqref{condition1} must also be satisfied.

     Consider the strongly   straight substitution $\sigma$ defined over $\{a,b,c\}$ by   
$$
\begin{array}{c c l}
a & \mapsto & abc \\
b & \mapsto & aabc \\
c & \mapsto  & aaa.\\
\end{array}
$$

  For this substitution, the words of length two that belong to $\mathcal L_{\sigma}$ are
     $\{aa,ab, bc, ca\}$. 
     Setting $h(a)=1$ and   $h(b)=h(c)=-1$ and  defines a coboundary,  with  the  corresponding $\bar{f}$  being  $\bar{f}(a)=\bar{f}(b)=1$ and  $\bar{f}(c)=-1$. 
         However, since 
     $h_{n+1}(c)=3h_n(a)$, there is no value of $\lambda$ such that \eqref{condition1} is true for this function $h$, as otherwise  it  would imply that
     \[-1=h(c)= \lim_{n\rightarrow \infty}\lambda^{h_{n+1}(c)} =  \lim_{n\rightarrow \infty}\lambda^{3h_{n}(a)} =  ( \lim_{n\rightarrow \infty}\lambda^{h_{n}(a)})^3 =h(a)^3  =1,  \]
     a contradiction.

     \end{example}
    
  \begin{remark}\label{rem:extra-coboundary-def}
     The   reason why     $h$ does not define an eigenvalue in Example \ref{example 1} is that $\bar{f}:\mathcal A\rightarrow {\mathbb S}^1$ is non-constant whereas $\sigma$ admits only one fixed point.  Indeed, the map  $\bar{f}$ encodes what values an eigenfunction $f$ can take on the fixed points for $\sigma$, namely, $\bar{f}(a)$ should be the value that $f$ assigns to the fixed point defined by  the letter $a$. Since $\sigma$ only has one fixed point, $\bar{f}$ must be constant. This example shows that for a coboundary  $h$ to define an eigenvalue, it must be accompanied by the appropriate $\bar{f}$. This explains why, in Definition
     \ref{def:Sadic-coboundary-one-sided} below, we define a coboundary to be a pair $(h,\bar{f})$ such that $\bar{f}$  takes the same values for letters that define the same limit words.   
           
 \end{remark}

 The following lemma tells us that $\bar{f}$ must take the same value on letters that define the same limit words.
\begin{lemma}\label{lem:coboundary-function}
 Let $\sigma$ be a   strongly  straight substitution on the alphabet $\mathcal A$.  Suppose that the one-sided shift $(\tilde{X}_\sigma, T,\mu)$ is recognizable, and that 
$h$ is a coboundary which satisfies \eqref{condition1}, with associated $\bar{f}$. If ${\boldsymbol a} ={\boldsymbol b}$,   then $\bar{f}(a)=\bar{f}(b)$. 
\end{lemma}

 \begin{proof}  
 Suppose that $a,b\in \mathcal{A}$ satisfy $\boldsymbol{a}=\boldsymbol{b}$. Primitivity implies the existence of a word $w$ such that $awb\in \mathcal L_{\sigma}$. Since $h$ is a coboundary, one has on the one hand $\bar{f}(b)=\bar{f}(a)h(aw)$. On the other hand, Host's theorem tells us that the $\lambda$ from  \eqref{condition1} is a continuous eigenvalue. Let $f$ be the continuous eigenfunction associated to $\lambda$. Now
$\sigma^n(aw)\sigma^n(b)\in \mathcal L_\sigma$, then by continuity of $f$
 $\lim_{n\rightarrow \infty}f( T^{\sigma^n(aw)}   {\boldsymbol a}  ) = f({\boldsymbol a}),$ so that $\lambda^{ | \sigma^n(aw)  |  }\rightarrow 1$.  Since  $h$ satisfies \eqref{condition1}, we have  $h(aw)=1$ and thus $\bar{f}(b)=\bar{f}(a)$. \end{proof}

\subsection{Height} \label{subsec:height} 
Host's definition of   a coboundary is motivated by the combinatorial definition  of  height for a constant-length substitution,
a definition  due to Kamae \cite{Kamae:1972} and Dekking \cite{Dekking:1977}. Its strength is that it allows  a  complete description of the 
set of eigenvalues  as $\{ j/q^n\tilde{h}: j\in \Z, n\in \N\}$  when $\tilde{h}$ stands for the height (see Corollary \ref{cor:Cobham0} below).
We assume below that  the substitution $\sigma$ has a fixed point, as otherwise we consider a power of $\sigma$.

 \begin{definition}[Height]\label{def:height-substitution}  Let  $\sigma$ be a primitive constant-length substitution with length $q$,  and let  $u$  be  any one-sided fixed point for $\sigma$.  The {\em height $\tilde h$ } of  $\sigma$ is defined as
  \begin{align*} \label{height}
\tilde h&:= \max \{n\geq 1: \gcd(n,q)=1, n | \gcd\{k: u_k=u_0 \} \}\\
&= \gcd\{ |w|: w \mbox{ is a return word to $u_0$, and $|w|$ is coprime to $q$}\}.
\end{align*}
   \end{definition} 
   For the equivalence between the two   formulations above in the definition, see, e.g., \cite[Section 6.1.1]{Queffelec:10}.  
   
  Equation \eqref{condition1}
 tells us that 
  for a constant-length substitution,  a coboundary $h$ must be constant. The   substitution $\sigma$   has non-trivial height $\tilde{h} \neq 1$   if and only if  it defines a non-trivial coboundary   $h\neq 1$.   When the  height $\tilde{h}$  is  non-trivial,  the constant function $h\equiv  e^{2\pi i /\tilde h}$ is a non-trivial coboundary associated to an eigenvalue. See  also Corollary \ref{cor:Cobham0},  which recalls  the explicit  and classical above-mentioned relation with   eigenvalues.

Note that while a  primitive constant length substitution $\sigma$ admits  its length $q$ as   the dominant eigenvalue of its  incidence matrix $M_\sigma$, and that this eigenvalue leads to a dynamical eigenvalue $e^{2\pi i/q}$ of the shift $(X_\sigma,T)$, there exist primitive non-constant length substitutions such that the dominant eigenvalue of $M_\sigma$ is an integer, but where $(X_\sigma,T)$ is weakly mixing, i.e., it has  no dynamical eigenvalues; see e.g. \cite{ASY}.

        In this article we investigate extending Theorem \ref{Host} to  $S$-adic shifts. In particular, in  Section \ref{constant-length-S-adic}, we apply our results 
        to constant-length  $S$-adic shifts and  we define the appropriate version of  of height in this case    (see    Definition \ref{def:height-Sadic}).

\begin{remark}\label{rem:trivialc}There are  many important situations for which coboundaries  are always  trivial. 
In particular   any coboundary of  a {\em Pisot   irreducible} substitution is trivial  \cite{Barge-Kwapisz:06}; see   Section \ref{sec:Pisot} for definitions.
More generally,  and as noticed by Clemens M\"ullner in private communication,  a coboundary of 
a primitive substitution for which   the measures of the   cylinders  associated to  letters are rationally independent   is 
trivial. Also,  left  proper substitutions  have only trivial coboundaries; see Proposition  \ref{prop:proper-implies-trivial-coboundary}. 
 These situations are a manifestation   of  a so-called property of coincidence  (see  \cite{AkiBBLS} for more  on  the subject, see also Section \ref{Toeplitz-S-adic}).
\end{remark}

\section{Coboundaries for directive sequences and continuous eigenvalues}\label{sec:coboundary-directive}

In this section, we  investigate   continuous eigenvalues.
 We first  define $S$-adic coboundaries  in Section \ref{subsec:sac}; the main results  for the continuous case are then  stated  and proved in Section \ref{sec:coboundary-continuous}.
Lastly, in Section \ref{sec:Pisot} we  investigate the  relation  between    continuous eigenvalues  and measures of letters in the case of  a trivial coboundary.
 
\subsection{$S$-adic coboundaries} \label{subsec:sac}
 
We define a  coboundary as follows: 
\begin{definition}[$S$-adic coboundary]\label{def:Sadic-coboundary-one-sided} Let $\boldsymbol{\sigma}=(\sigma_n)_{n\geq 0}$ be a  straight  directive sequence  on $\mathcal A$.
A  coboundary for $\boldsymbol{\sigma}$ is  a morphism $h:\mathcal A^*\rightarrow \mathbb S^{1}$,  and a map $\bar f:
\mathcal A \rightarrow \mathbb S^{1}$ satisfying
\begin{itemize}
\item       $   \bar f(b )  =  \bar f(a )$ whenever    $\boldsymbol {  a} =\boldsymbol{ b}$, and     \item   ${\bar f(b)  =  \bar f(a ) \lim_{n_k\rightarrow \infty}h(w_{n_k})}$ 
 for any $a,b\in \mathcal A$,  and  for any sequence 
  $(w_{n_k})_{n_k}$   of  transition words from $a$ to $b$ such that $w_{n_k} \in \mathcal L_{n_k} $
 for all $k $ and  $\lim_{n_k\rightarrow \infty}h(w_{n_k})$ exists.  \end{itemize}
 If we replace the second requirement above by the  weaker  condition that 
$\bar f(b)  =  \bar f(a ) \lim_{n\rightarrow \infty}h(w_{n_k})$
for any $a,b\in \mathcal A$,  and  for any sequence 
  $(w_{n_k})_{n_k}$   of   transition words  from $a$ to $b$ of bounded length such that $w_{n_k} \in \mathcal L_{n_k} $
 for each $k $ and  $\lim_{n_k\rightarrow \infty}h(w_{n_k})$ exists, then we say that $(h,\bar f)$ is a {\em weak coboundary}.
\end{definition}

\begin{remark}\label{remark:coboundary}\hfill
\begin{enumerate}
\item
Sometimes we will simply refer to a coboundary as $h$, dropping mention of $\bar f$. In particular, when we write that
 $h$ defines a constant   coboundary, we mean that there exists a function $\bar f$ such that 
 $(h,\bar f)$ is a  coboundary, with $h$ taking constant values. Moreover, in the case where   the map  $h\equiv 1$,  the coboundary is said to be {\em trivial}.
\item
Let   $\sigma$ be a substitution that   occurs infinitely often in {$\boldsymbol \sigma$}. 
Let $a$ be a letter,    and    let $wb$ be  a prefix of $\sigma (a)$. Then  existence of a coboundary implies that $\bar f(b)= h(w)\bar f(a)$.  To see this,  suppose that $w$ starts with $w_0$, then $\bar f (b) = h(w) \bar f(w_0)$. Also, the supposed conditions imply that 
$\boldsymbol {  a} =\boldsymbol{ w_0}$. The claim follows.
\item
Note that in the substitutive case, that is when {$\boldsymbol \sigma$} is a stationary directive sequence, if $(h,\bar{f})$ satisfies Definition \ref{def:Sadic-coboundary-one-sided}, then $h$ also satisfies Definition
\ref{subst-coboundary-combinatorial-definition}.   Conversely, if $h$ is a coboundary that satisfies Definition
\ref{subst-coboundary-combinatorial-definition} and {\em also} defines a dynamical eigenvalue via \eqref{condition1}, then the function $\bar f$  guaranteed by Lemma \ref{f-h-lemma-one}  satisfies the  conditions of  Definition \ref{def:Sadic-coboundary-one-sided}.  Note that  coboundaries, as defined in Definition
\ref{subst-coboundary-combinatorial-definition},  that do not lead to eigenvalues do not have to satisfy the first condition of  Definition \ref{def:Sadic-coboundary-one-sided}; see Example \ref{example 1}. 
\item
Coboundaries are supposed to reflect eigenvalues, and, in Definition \ref{def:Sadic-coboundary-one-sided},  the function $\bar f$ is supposed to define the putative eigenfunction on one-sided limit words. This is why we impose the first requirement in Definition \ref{def:Sadic-coboundary-one-sided}.
See also  Remark \ref{rem:extra-coboundary-def}  and Lemma \ref{lem:coboundary-function}.
   \item 
Although the coboundaries that we have defined seem to be suited to the study of one-sided shifts, because they are defined in terms of one-sided limit words, nevertheless, they are sufficient to work with two-sided shifts. This is mainly due to Proposition \ref{two-one-sided-continuous}.

\item   Observe that in the substitutive case, if $h$ is a coboundary defined by \eqref{condition1}, then $h(\sigma(w))= h(w)$ for any word. 
Similarly, 
 let $\boldsymbol{\sigma}=(\sigma_n)_{n\geq 0}$ be a  primitive directive sequence  on $\mathcal A$.
Suppose that each $\sigma_n$ that appears in $\boldsymbol{\sigma} $ appears infinitely often, and that for each $a\in \mathcal A$, the limit
\begin{equation*}h(a):= \lim_{n\rightarrow \infty}\lambda^{  h_n(a)}\end{equation*}
exists. Then   $h(\sigma(a))=h(a)$ for each $\sigma$ that appears in  $\boldsymbol{\sigma} $ and each $a\in \mathcal A$. 
To see this, if $\sigma$ appears in $\boldsymbol{\sigma} $ then for some $(n_j)$, $\sigma=\sigma_{n_j}$ for each $j$. 
Recalling the notation $h_n(a):= |\sigma_{[0,n)}(a)|$, one has, for any $a\in \mathcal A$,
\[ h(\sigma(a)) = \lim_{j\rightarrow \infty} \lambda^{  h_{n_j}(\sigma (a))} =\lim_{j\rightarrow \infty}    \lambda^{ | \sigma_{[0,n_j)}(\sigma (a))|}   = \lim_{j\rightarrow \infty} \lambda^{ | \sigma_{[0,n_j +1)}(a)|}      = h(a) . \]

\end{enumerate}
\end{remark}

\subsection{$S$-adic coboundaries and continuous eigenvalues}\label{sec:coboundary-continuous}
In this section we  describe the connection between coboundaries and continuous eigenvalues. 
There are two types of statements, namely necessary conditions  and sufficient conditions  stated in terms of the existence of 
limits  of the form $ \lim_{n\rightarrow \infty}\lambda^{h_n(a)}$,  together with a coboundary condition. 
In Theorem \ref{thm:fully-essential-coboundary} we show that   the existence of a  continuous   eigenvalue  implies the  existence of a  weak coboundary, under some combinatorial condition.  For sufficient conditions, in
Theorem \ref{Host-continuous-bounded-onesided} we show that    the existence of a   coboundary associated to a  rational $\lambda$ implies the existence  of a  continuous   eigenvalue, and in
Proposition \ref{Host-continuous-bounded-rw} we provide the return word  version of Theorem \ref{Host-continuous-bounded-onesided}. Lastly, we consider   the case of a  finitary, straight, recognizable,  directive sequence, where each substitution appears infinitely often.  With these restrictions, we show in  Theorem \ref{thm:sufficient-continuous-A}   that, provided that the terms $\lambda^{h_n(a)}$ converge sufficiently fast, $\lambda$ is a continuous eigenvalue.

In the following theorem, we require that there is a fully essential word of length two. It is easy to find examples where this is satisfied, for example, we can take a directive sequence where there exists at least one substitution which appears infinitely often, so that one can telescope to a directive sequence with a fully essential word.  Also, this condition is satisfied if the substitutions in our directive sequence each have a word of length two in common in the images of letters by each substitution. However, note for instance with the case of an Arnoux-Rauzy directive sequence  on a  three-letter alphabet, Example \ref{ex:AR}, we would require   that  one of the substitutions occurs 
boundedly often,  to apply Theorem \ref{thm:fully-essential-coboundary}. By Proposition~\ref{two-one-sided-continuous}  and Remark~\ref{remark:coboundary}(v), 
the following result can be applied to either one- or two-sided shifts.

\begin{theorem}\label{thm:fully-essential-coboundary} 
Let $\boldsymbol{\sigma} = (\sigma_n)_{n\ge0}$ be a  straight  directive
sequence on $\mathcal A$.
 Suppose that for each $a\in\cA$, there exists $\ell\in\cA$ such that $a\ell$ is a fully essential word.
 If $\lambda\in   {\mathbb S}^1$ is a continuous  eigenvalue,  then 
$$
h(a):= \lim_{n\rightarrow \infty}\lambda^{h_n(a)}
$$
exists   and defines a weak coboundary for $\boldsymbol{\sigma}$.  
\end{theorem}

\begin{proof}
Suppose that $\lambda$ is a continuous eigenvalue. By Proposition \ref{two-one-sided-continuous}, it is   a one-sided  continuous eigenvalue; let $f$  be a corresponding eigenfunction for the one-sided shift ${\widetilde X}_{\boldsymbol \sigma}$; since $\lambda\in   {\mathbb S}^1$, then $|f(x)|$ is constant on any orbit. Primitivity of  $\boldsymbol{\sigma}$ implies that
 ${\widetilde X}_{\boldsymbol{\sigma}}$ is minimal, and now continuity implies that $|f|$ is constant on 
  ${\widetilde X}_{\boldsymbol \sigma}$. Without loss of generality,  we assume $|f|=1$.
Let $a\in\cA$ and let $a\ell $ be a fully essential word. Since $a\ell\in \mathcal L_{\boldsymbol{\sigma}}^{(n)}$ for each $n$, the set $\sigma_{[0,n)}([ a\ell]) \subset
{\widetilde X}_{\boldsymbol{\sigma}}$ is non-empty for each $n$, and so $\sigma_{[0,n)}([ a\ell])$ belongs to the language $\mathcal L_{\boldsymbol{\sigma}}^{(0)}$. Let $(x^{(n)})_{n\in\N}$ be a sequence of points in ${ \widetilde  X}_{\boldsymbol{\sigma}}   $  
such that $x^{(n)}\in \sigma_{[0,n)}([ a\ell])$ for each $n\in \N$. 
Note that $ x^{(n)} \rightarrow  {\boldsymbol a}$ by straightness, and $f(x^{(n)})\rightarrow f( {\boldsymbol a})$ since $f$ is continuous.
Note also that $T^{ h_n(a)    }(x^{(n)})\rightarrow {\boldsymbol \ell}$ by straightness. Since $f \circ T^{h_n(a)} (x^{(n)})= \lambda ^{h_n(a)} f(x^{(n)})$,
we obtain that
\begin{equation*}
 \lim_{n\to \infty} \lambda^{ h_n(a) } f(x^{(n)})=f({\boldsymbol \ell}),
\end{equation*}
and since $ \lim_{n\to \infty} f(x^{(n)})=f(   {\boldsymbol a }\boldsymbol     )\neq 0$, we conclude that  $h(a):= \lim_{n\rightarrow \infty} \lambda ^{h_n(a) }$ exists. Since each $a\in\cA$ admits a fully essential word of the form $a\ell$, 
$h$ is  defined on  all of  $\cA$, i.e.,  $ \lim_{n\rightarrow \infty} \lambda ^{h_n(a) }$ exists for all $a$.

Define $\bar f: \mathcal A\rightarrow \mathbb S^{1}$ to be $\bar f( a )= f(\boldsymbol { a})$. 
We  now verify that 
$h$ and $\bar f$ satisfy the conditions of  a weak coboundary in Definition \ref{def:Sadic-coboundary-one-sided}. Clearly $\bar f(b )  =  \bar f(a )$ whenever    $\boldsymbol { b} =\boldsymbol{ a} $. Suppose that   $(w_{n_k})$ is  a sequence of transition words   from $a$ to $ b$  in  $\mathcal L_{\boldsymbol{\sigma}}^{(n_k)}$, of bounded length.
  As before, we have
 \[  f(\boldsymbol {b})=f(\boldsymbol {a}) \lim_{k\rightarrow \infty} \lambda^{h_{n_k }(w_{n_k}) }.\]
 Since   by definition   $\bar f( a )= f(\boldsymbol {a})$ for each letter $a$, one has 
 \[\bar f (b) = f(  \boldsymbol { b} )= 
 f(  \boldsymbol  { a})  \lim_{k\rightarrow \infty} \lambda^{ h_{n_k} (w_{n_k}) } =  \bar f(a  ) \lim_{k \rightarrow \infty} h(w_{n_k})   .\]

  \end{proof}

As a   sufficient condition  in the rational case  for the existence of a continuous eigenvalue, we have the following.

 \begin{theorem}\label{Host-continuous-bounded-onesided}
 Let $\boldsymbol{\sigma} = (\sigma_n)_{n\ge0}$ be a  straight and {recognizable}  directive sequence  
 on $\mathcal A$. Let $\lambda \in {\mathbb S}^1$ be rational. If
$$
  {h(a)}:= \lim_{n\rightarrow \infty}\lambda^{  h_n(a)}
  $$
exists for each $a\in \mathcal A$  and if the map $h: a \mapsto h(a)$ defines a    coboundary $(h,\bar f)$  for $\boldsymbol{\sigma} $,
then $\lambda$ is a continuous eigenvalue.

\end{theorem}

\begin{proof}
Write   $\lambda = e^{\frac{2\pi i p}{q}}$ where $p$ and $q$ are non-zero coprime integers (with $q\geq 1$). Fix a   limit word $u$, and,   for $0\leq j\leq q -1$, define
\[ A_j:=\overline{\{ T^{i}(u): i\in  {\Z},\quad  i\equiv j \mod q \}}.\]
  Minimality tells us that the union of the sets $A_j$ is {$X_{\bf \sigma}$}. If we show that the sets  $A_j$ are pairwise disjoint, then this means that  the closed sets $A_j$ are also  open, i.e., $\{A_0, \dots , A_{q-1}\}$ forms a clopen partition with    $T(A_i)= T(A_{i+1\mod q})$. We will see below that this implies that this  partition defines a continuous dynamical eigenvalue.
  
  Let us first  prove that  the sets $\{ A_j: 0\leq j \leq q -1\}$ are pairwise disjoint. 
  Since $\lambda$ is a $q$-th root of unity  and $\lim  \lambda^{  h_n(a)}$ exists for all $a$,   $ \lambda^{  h_n(a)}= h(a)$  for all  $n$  large enough and each $a\in \mathcal A$.
   The map  $ a\mapsto \lambda ^{h_n(a)}$  is a morphism for all $n$. Since $h$ is also a morphism, we get that for all $w\in \cA^+$, $h(w)=\lambda^{h_n(w)}$ for all $n$ large enough.

The assumption of recognizability implies the following: for each $\varepsilon>0$ and  $n\in \mathbb N$, there exists $N_n$ such that if $w=w_{-N_n} \dots w_{N_n}\in \mathcal L_{\boldsymbol\sigma}$, then there exists $M_n\geq N_n(1-\varepsilon)$ such that $w_{-M_n} \dots w_{M_n}$  has  a unique centred representation at level $n$ for  $w_{-M_n} \dots w_{M_n} $  in $ \mathcal L_{\boldsymbol\sigma}^{(n)}$.
By a centred representation of a word at level $n$, we mean   
that there is a  (finite) word $v^{(n)}= v_{-k}^{(n)} \dots v_0^{(n)} \dots v_{\ell}^{(n)}\in \mathcal L_{\boldsymbol\sigma}^{(n)}$,
there exist a proper suffix $s$ of $\sigma_{[0,n)}(v^{(n)}_{-k})$,     a proper prefix $p$  of $\sigma_{[0,n)}(v^{(n)}_{\ell})$,  a proper prefix  $p'$ of $\sigma_{[0,n)}(v^{(n)}_{0})$,
  and a proper suffix $s'$ of $\sigma_{[0,n)}(v^{(n)}_{0})$,  where $p's'=\sigma_{[0,n)}(v^{(n)}_{0})$,
such that \[w_{[-M_n, -1]}= s\, \sigma_{[0,n)} (v^{(n)}_{-k+1} \dots v^{(n)}_{ -1})p' \mbox{ and }
w_{[0, M_n]}= s' \sigma_{[0,n)} (v^{(n)}_{1} \dots v_{\ell -1}^{(n)})\,p .\]
   By uniqueness we mean that there is a unique such pair  $(v^{(n)},|p'|)$.
 Otherwise, by compactness, we would find a point in $X_{\boldsymbol \sigma}$ with two desubstitutions at some level $n$, and this contradicts recognizability.

Now suppose that $x\in A_j\cap A_{j'}$, $j\neq j'$. This means that there exist sequences $(m_k)$ and $( m_k')$ such that $m_k\equiv j \mod q$, 
 $ m_k'\equiv j' \mod q$, $T^{m_k} u \rightarrow x$ and $T^{ m_k'} u \rightarrow x$. Without loss of generality, we can assume that $m_k<m_k'$ for each $k$. Fix $\varepsilon>0$. For each $n$,
 choose $N_n$, $M_n$, $(v^{(n)},|p'|)$ as above (here $p'=(p') ^{(n)}$).   By dropping to a subsequence if necessary, we may assume that each $v^{(n)}_0$  equals a fixed letter $b$ and each $|s|$ equals a fixed value $\ell \bmod q$.
 Next choose $k_n$ such that $ (T^{m_{k_n}}u)_{[-N_n,N_n]} =(T^{m_{k_n}'}u)_{[-N_n,N_n]}     =x_{[-N_n,N_n]}$.  
  Let $u^{(n)}\in X_{\boldsymbol \sigma} ^{(n)}$ be such that $u=\sigma_{[0,n)} ( u^{(n)}    )$.
   In $u^{(n)}$, we see two occurrences of $v^{(n)}$,
   where  the  image under  $\sigma_{[0,n)}$ of the first  occurence of $v^{(n)}$  in $u$  starts at an index congruent to $m_{k_n}-N_n-\ell$,  and the image of the second one  starts  at an  index congruent to $m'_{k_n}-N_n-\ell$;  they  are separated by a concatenation $R^{(n)}$ of return words to the letter $b$   in $u^{(n)}$.
       As $h$ is assumed to be a cocycle, we have $h_n( v^{(n)}R^{(n)})=
|\sigma_{[0,n)} (v^{(n)}R^{(n)})|   
   \rightarrow 0 \mod q$ as $n\rightarrow \infty$, and so $
|\sigma_{[0,n)} (v^{(n)}R^{(n)})|   \equiv 0 \mod q$ for $n$ large. This contradicts the fact that $|  \sigma_{[0,n)} (v^{(n)}R^{(n)}) |= m_{k_n}'-m_{k_n} \not \equiv  0 \mod q $ for all $n$.
We thus have proved that  the sets $A_j$ are disjoint.

Let $a$ be such that  $u=\bm{a}$, and set  $f(u):= \bar f(a)$, where $\bar{f}:\mathcal{A}\to {\mathbb S}^1$   is  a function associated to $h$.
We define $f(x)= f(u)\lambda ^j $ if $x\in A_j$. Since $f$ is constant on each $A_j$,    and since the   sets $\{ A_j: 0\leq j \leq q-1\}$   are open,   $f$ is continuous. Since $f(T (T^n(u)))= \lambda f(T^n(u))$ and the orbit of $u$ is dense by minimality, $f$ is an eigenfunction for $\lambda$.

 \end{proof}

We will  need  in Section 
  \ref{subsec:clmeasurable} a slightly modified version of Theorem \ref{Host-continuous-bounded-onesided}, stated in terms of return words, which  avoids the use of  coboundaries
 (in the flavour of  \cite{Ferenczi-Mauduit-Nogueira}).
We use  below the notation from (\ref{eq:noth}) for $h_{n_k}$. The proof of the following proposition is contained in that of Theorem \ref{Host-continuous-bounded-onesided}.

\begin{proposition}\label{Host-continuous-bounded-rw}
Let $\boldsymbol{\sigma} = (\sigma_n)_{n\ge0}$ be a  straight  recognizable 
 sequence of 
 substitutions on $\mathcal A$. If $\lambda$ is rational and 
 $$ \lim_{k\rightarrow \infty}\lambda^{  h_{n_k}(w_{n_k})}=1$$
whenever $(w_{n_k})$ is a sequence of   return words to some $a\in \cA$,  with  $w_{n_k} \in \mathcal L_{n_k}$, then $\lambda$ is a continuous eigenvalue.
  \end{proposition}

We next  consider   the analogue of \cite[Proposition 7]{CDHM:2003},  namely Theorem \ref{thm:sufficient-continuous-A}. 
  Note that  \cite[Proposition 7]{CDHM:2003} makes no mentions of coboundaries whilst Theorem  \ref{thm:sufficient-continuous-A}  does; this is because 
    the assumption that an $S$-adic representation is left-proper simplifies some issues, as the next proposition shows,  namely,  that all coboundaries are trivial for 
    left-proper substitutions. We state and prove this theorem for left-proper substitutions, but we note that there is an  analogous statement and proof, for 
     right-proper substitutions.
        
        \begin{proposition}\label{prop:proper-implies-trivial-coboundary}
       Suppose that the   primitive directive sequence ${\boldsymbol \sigma}$  consists of left-proper substitutions. If $\lambda$ is a continuous eigenvalue of $(X_{\boldsymbol \sigma}, T)$, then  it must satisfy
$$\lim_n \lambda^{h_n(a)}=1$$ for each letter $a$.
    \end{proposition}
    \begin{proof}
    By Proposition \ref{two-one-sided-continuous}, we need only work with the one-sided shift $(\tilde{X}_{\boldsymbol \sigma}, T)$.
    The assumption that each $\sigma_n$ is left-proper means that for each $n$ there is a letter $a_n$ such that $\sigma_n(b)$ starts with $a_n$ for each letter $b$.  This implies  that the sequence of words
    $(\sigma_{[0,n)}( a_{n}))$ is nested and converges to 
     the unique right-infinite limit word $u$ for the one-sided shift  $\tilde{X}_{\boldsymbol \sigma}$. Furthermore, for any letter $b$ and any $n>1$,  $\sigma_{[0,n)}(b)$ starts with $ \sigma_{[0,n-1)}(a_{n-1}) $. For each $b\in \cA$, let $(x^{(n)})_{n\in \N}$ be a sequence of points in $\tilde{X}_{\boldsymbol \sigma}$ such that $x^{(n)}\in \sigma_{[0,n)}([b])$ for all $n$, and hence 
     $x^{(n)}$ starts with $ \sigma_{[0,n)}(b )\sigma_{[0,n-1)}(a_{n-1} )$ for all $n$.
     Then $x^{(n)}\to u$ and $T^{h_n(b)}(x^{(n)})\to u$. Let $f$   be a continuous eigenfunction associated to $\lambda$.  (In particular, $f\neq 0$.) We obtain that
$$\lim_{n\to\infty} \lambda^{h_n(b)}f(x^{(n)})=f(u),$$
thus, $\lambda$ must satisfy $\lim_{n\to\infty}\lambda^{h_n(b)}=1$ for each letter $b\in\cA$.

\end{proof}

Because of Proposition \ref{prop:proper-implies-trivial-coboundary},
  the statements of  \cite[Propositions 7, 8]{CDHM:2003}
are stated in terms of convergence of the terms  $\lambda^{h_n(b)     }$ to 1; we  discuss in Section \ref{sec:Pisot}   consequences of such  a convergence property  in terms of 
the relation between   eigenvalues and  measures of one-letter cylinders; see Theorem \ref{theo:Pisot} and Lemma \ref{lem:eigenvectorbis}. In general  $\lambda^{h_n(b)     }$ converges to a term $h(b)$, for some coboundary $h$.

\begin{theorem} \label{thm:sufficient-continuous-A}
Let $\boldsymbol{\sigma} = (\sigma_n)_{n\ge0}$ be a  finitary,  straight, recognizable  directive sequence on $\mathcal A$,
 where each substitution occurs infinitely often. 
 Suppose that   $(X_{\boldsymbol{\sigma}},T)$ is
aperiodic. 
 Let $(h,\bar f)$ be a weak coboundary for $ (X_{\boldsymbol \sigma}, T)  $. Let $\lambda\in {\mathbb S}^1$. If 
 \[ \sum_{n=1}^{\infty} | \lambda^{ h_n(a) }  -h(a)  |<\infty\]
 for each $a\in \mathcal A$, then $\lambda$ is a continuous eigenvalue of $
 (X_{\boldsymbol \sigma}, T) $.
\end{theorem}
\begin{proof}

 Since  all  the substitutions appear infinitely often, all the  factors  of images of letters  are essential. Now fix a substitution  $\sigma$ that occurs   in the  directive sequence
  and fix  a  pair of  letters $a,b$ such that $b$ occurs in $\sigma(a)$. Write $\sigma(a)=wbw'$. 
The word $w$ is  an  essential    transition word from  the first letter of $w$  to $b$. This means we can apply (ii) of Remark \ref{remark:coboundary},  to write $h(w)\bar f(w_0) = \bar f(b)$.

     Pick and fix $x\in     X_{\boldsymbol\sigma}   $. We use the notation of Section \ref{sec:generating-partitions} and work with the sequence of partitions ($\mathcal{Q}_n$). For each $n$, there exist a letter $a_n  = a_n(x) $ and $j_n= j_n(x)$, with
$0\leq j_n<h_n(a_n)$ such that $x\in T^{j_n}B_n(a_n)$. In other words, $(j_n,a_n)_{n\geq 0}$ is a $(\mathcal{Q}_n)$-address for $x$. Note that this forces $a_n$ to appear in $\sigma_n(a_{n+1})$: if $\sigma_n(a_{n+1})=b_n^{(1)}\dots b_n^{(L_n)}$, then there exists $0\leq k_n<L_n$ such that $b_n^{(k_n+1)}=a_n$. If $k_n=0$, then $j_{n+1}=j_n$. If $k_n\geq 1$, then $j_{n+1}=\sum_{i=1}^{k_n} h_n(b_n^{(i)})+j_n$. See Figure \ref{fig:partitionQ}. 

  Note also that  $b_n^{(1)}\dots b_n^{(k_n)}$ is an essential  transition word from $b_n^{(1)}$ to $a_n$, from the remark above. Moreover,   since ${\boldsymbol \sigma}$ is straight, one has  ${\boldsymbol a }_{n+1}={\boldsymbol   b_n^{(1)}}$ (we use the fact that  the substitutions $\sigma_n$  occur infinitely often). This implies in particular  that 
$ \bar{f} ( b_n^{(1)})=  \bar{f} (a_{n+1})$. Since $h$ is a weak coboundary,   for each  fixed morphism $\sigma$,     by taking the limit on the subset of indices $n$ where
$\sigma$ occurs   in the directive sequence, i.e., the $n$ such that $\sigma_n=\sigma$, we have that    \[ 
\left|        h(b_n^{(1)}\dots b_n^{(k_n)}) \bar{f}(b_n^{(1)}) - \bar{f}(a_{n})     \right| \rightarrow 0 , \]
and so equals 0 since the expression is constant. We have shown that $ |h(b_n^{(1)}\dots b_n^{(k_n)}) \bar{f}(a_{n+1}) - \bar{f}(a_n)|=0$ for all $n$.

For each $n$, define $f_n(x)=\lambda^{j_n}\bar{f}(a_n)$. 

  {\bf Claim:} the sequence $(f_n)$ converges  uniformly to a continuous function $f$.
We have,  as  $\lambda$ has modulus $1$,
\[|f_{n+1}(x)-f_n(x)| = |\lambda^{j_{n+1}} \bar{f}(a_{n+1})- \lambda^{j_n} \bar{f}(a_n)| =    |\lambda^{ \sum_{i=1}^{k_n} h_{n}(b_n^{(i)}) } \bar{f}(a_{n+1})-  \bar{f}(a_n)|. \]
 Also, since  $|\bar{f}|=1$,
\begin{align*}
&|\lambda^{ \sum_{i=1}^{k_n}h_{n}(b_n^{(i)}) } \bar{f}(a_{n+1})-   \bar{f}(a_n)|
 \leq \\&
 |  \lambda^{ \sum_{i=1}^{k_n} h_{n}(b_n^{(i)}) }\bar{f}(a_{n+1}) -   
     \lambda^{ \sum_{i=1}^{k_n-1}h_{n}(b_n^{(i)})} h(b_n^{(k_n)}) \bar{f}(a_{n+1})| + |   \lambda^{ \sum_{i=1}^{k_n-1}h_{n}(b_n^{(i)})}h(b_n^{(k_n)}) \bar{f}(a_{n+1}) -   \bar{f}(a_n)| \\&=
      |  \lambda^{ h_{n}(b_n^{(k_n)}) }-   
      h(b_n^{(k_n)}) | + |   \lambda^{ \sum_{i=1}^{k_n-1}h_{n}(b_n^{(i)})}h(b_n^{(k_n)}) \bar{f}(a_{n+1}) -   \bar{f}(a_n)|, 
\end{align*}
and recursively,
\begin{align*}
|f_{n+1}(x)-f_n(x)|&= |\lambda^{ \sum_{i=1}^{k_n}h_{n}(b_n^{(i)}) } \bar{f}(a_{n+1})-   \bar{f}(a_n)|\\&  \leq \left(  \sum_{i=1}^{k_n} |\lambda^{     h_{n}(b_n^{(i)})} - h(b_n^{(i)})| \right) + |h(b_n^{(1)}\dots b_n^{(k_n)}) \bar{f}(a_{n+1}) - \bar{f}(a_n)| \\ &=   \sum_{i=1}^{k_n} |\lambda^{     h_{n}(b_n^{(i)})} - h(b_n^{(i)})| .   \end{align*}

Note that since $\boldsymbol\sigma$ takes values from a finite set, $k_n$ is uniformly bounded on $n$,  and this implies that the sum on the right hand is uniformly bounded for all $x$. Thus
  the series $\sum_{n\geq 1}||f_n-f_{n-1}||_{\infty}$ converges, and the sequence $(f_n)$ converges  uniformly to a continuous function $f$. 
  
   {\bf Claim:} $f$ is an eigenfunction for $\lambda$. 
 Note  first that the maps $x\mapsto j_n(x)$ and $x\mapsto a_n(x)$ are continuous.  For all $n$ and for all $x$ such that $0\leq j_n(x)<h_n(a_n(x))-1$, we have that $j_n(Tx)=j_n(x)+1$ and $a_n(Tx)=a_n(x)$.  Hence $f_n(Tx)=\lambda  f_n(x)$.
One has either  $0\leq j_n(x)=h_n(a_n(x))-1$ for all $n$ or  $0\leq j_n(x)<h_n(a_n(x))-1$ for $n$ large enough.
By   the same  argument as in the proof of Lemma \ref{lem:finitely many limit words} there are finitely many  points $x\in X_{\boldsymbol{\sigma}}$ such that $j_n(x)=h_n(a_n(x))-1$ for all $n$, and by aperiodicity, they are well approximated by points $y$ such that $j_n(y)<h_n(a_n(y))-1$. This proves that $f$ is an eigenfunction on a dense subset of $X_{\boldsymbol \sigma}$, from which the result follows.

\end{proof}

\subsection{Convergence, eigenvalues and  measures} \label{sec:Pisot}
We discuss in this section  the  convergence properties involved in Theorem  \ref{thm:sufficient-continuous-A}. 
We focus on the case where the incidence matrices are invertible  (i.e.,  they have non-zero determinant).  With the assumption of strong convergence discussed below, the 
$S$-adic shifts are uniquely ergodic; see e.g. \cite[Theorem 5.7]{Berthe-Delecroix}. Here, we relate   eigenvalues to the measures of  one-letter cylinders.
We recall from  Section   \ref{subsec:subs} that    additive continuous eigenvalues  form  a subset of the   group   generated by the   measures of cylinders in the uniquely ergodic case,
 i.e., $E(X,T) \subset I(X,T)$ by   \cite[Proposition 11]{CortDP:16} and   \cite[Corollary 3.7]{GHH:18}. In this section we show that the absence of nontrivial coboundaries  allows us to deduce  a relationship between the   additive continuous eigenvalues 
 and  the measures of one-letter cylinders.  By Remark  \ref{rem:trivialc}, the  restriction  to trivial coboundaries  covers a large range of cases.  Note that  for  proper  minimal unimodular  $S$-adic subshifts,   measures of one-letter cylinders  are known to generate the full set $I(X,T)$
  \cite[Theorem 4.1]{BCDLPP:19}.

We consider a directive   sequence of substitutions  $\boldsymbol{\sigma} $   with  invertible incidence matrices over a $d$-letter alphabet $\{1, \dots , d\}$.
We recall  that $M_{[0,n)}= M_0 M_{k+1} \dots M_{n-1}$  stands for the incidence matrix of  $\sigma_{[0,n)}$.
Let ${\bf 1}$ stand for the vector  in $ \Z^d$ all of whose entries  equal  $1$.
Noting that the  $a$-indexed column of $M_{[0,n)}$ sums to 
$h_{n}(a)$, the existence of the limit  $h(a):= \lim_{n\rightarrow \infty}\lambda^{h_n(a)}$, with $\lambda= {e^{2\pi i t}}$, is equivalent to 
\[\lim_{n\rightarrow \infty }t{\bf 1} M_{[0,n)} \equiv \arg \boldsymbol{h} \bmod \Z^d,\]
where the $a$-indexed entry of  $\boldsymbol{h}$ equals $h(a)$ and $\arg \boldsymbol{h}$ denotes the column vector whose $a$-th index entry equals the argument of $h(a)$  (where the argument is considered modulo $2\pi$).
We focus here on the condition $\lim_{n\rightarrow \infty }t{\bf 1} M_{[0,n)} \equiv  0 \bmod \Z^d$, which is  the case when     the  coboundary is trivial, i.e., $h\equiv 1$ and $\bar f \equiv 1$,  and 
we relate it with the notion of convergence for the products of matrices $M_{[0,n)}$. Let $\boldsymbol e_a$ denote the vector of the canonical basis associated to the letter $a$.

\begin{definition}[Strong  and exponential  convergence]
Let  $\boldsymbol u  \in \R_{\ge 0}^d $ have 
entries which sum  to $1$. We say that the directive  sequence  $\boldsymbol{\sigma} $ on  the alphabet $\{1, \dots, d\}$  is {\em strongly convergent}  with   respect  to  $\boldsymbol u  \in \R_{\ge 0}^d $ 
 if
$$
\lim_{n\to+\infty} \mathrm{dist}(M_{[0,n)}\boldsymbol e_a,\R  \boldsymbol u) = 0\quad \hbox{for all }a\in \{1, \dots, d\},$$
 where  the distance $\mathrm{dist}$ refers to the usual  distance of a point to a line, i.e.,  $\mathrm{dist}({\boldsymbol  x} ,
 {\mathbb R}{ \boldsymbol  u})$  is the infimum of the set $\{\Vert {\boldsymbol x} -  \lambda   {\boldsymbol u} \Vert_2  \mid  \lambda \in\mathbb R \}$. 
 In this case we call  $\boldsymbol u$ a   {\em generalised  right  eigenvector} for $\boldsymbol{\sigma}$.   
The vector $\boldsymbol u$ is {\em normalised}  if the sum of its entries equals $1$.
 
 We say that the  directive sequence $\boldsymbol{\sigma} $  is {\em exponentially  convergent}  with   respect to  $\boldsymbol u  \in \R_{\ge 0}^d $ 
 if there exist  $C, \alpha 
 >0$ such that $$
 \mathrm{dist}(M_{[0,n)}\boldsymbol e_a,\R\boldsymbol u)   < C e^{-\alpha n}\quad \hbox{for all } n \mbox{ and   all } a\in \{1,\ldots,d\}.
$$ 
 \end{definition} 
  Note that  strong convergence  implies that
$$\bigcap _n M_{[0,n)} {\mathbb R}^d_+={\mathbb R}^d_+ u,$$ with the latter property being called cone convergence, see for example \cite{Pytheas-2020}. Note also  that exponential convergence implies  the convergence of $\sum_n  \mathrm{dist}(M_{[0,n)}\boldsymbol e_a,\R\boldsymbol u) $ for each $a$; the latter is called sum convergence.

  A generalised  right  eigenvector for $\boldsymbol{\sigma}$ can be seen as the generalisation of the Perron--Frobenius eigenvector of a primitive matrix.

  Note that the property that $\boldsymbol{\sigma} $ is exponentially convergent with respect to a generalised eigenvector has been identified and studied in the literature. For example, see  \cite[Theorem 6.3]{Berthe-Delecroix},  and also \cite{BST:19,BST:20}.   This   property has emerged in the Pisot  case.   Let us first  recall a few definitions. A {\em Pisot number}   is an algebraic integer greater than~$1$  whose Galois conjugates are all contained in the open unit disk.  A {\em Pisot   irreducible matrix} is such that  its characteristic polynomial  is the minimal polynomial   of a Pisot  number;  such a matrix  is primitive  by \cite{CanSie:2001}.  A {\em Pisot irreducible substitution} is one whose incidence matrix is  Pisot   irreducible. A stationary directive sequence   given by a Pisot 
  irreducible substitution is  exponentially convergent  to the Perron eigenvector \cite{Host:1986}. We recall  that the only  coboundary $h$ of a  Pisot irreducible substitution is the  trivial one,   by Remark \ref{rem:trivialc}, that is, $h\equiv 1$.
Also, if the directive sequence is made of  substitutions with the same Pisot  irreducible  incidence matrix, then  the  directive sequence will converge exponentially to the Perron eigenvector.   Sets  of substitutions with the same  incidence matrix  have been studied in
  \cite{Arnoux-Mizutani-Sellami:2014}, \cite{Baake-Spindler-Strungaru:2018} (where it is called semi-compatibility), and \cite{Rust:2020} (where it is called compatibility).

 In particular, if the incidence matrix  $M_\sigma$ of a  Pisot irreducible substitution $\sigma$  has unit determinant, then the eigenvalues of $(X_\sigma,T)$ lie in the linear span of the Perron eigenvector for $M_\sigma$ \cite[Example (6.2)]{Host:1986}. In this section we show that a similar result holds if our directive sequence  is finitary, recognizable, straight, under the condition  of   strong convergence, with invertible matrices. 
 
We start with
Lemma \ref{lem:eigenvectorbis}, which  extends 
 Assertion 1 of Exercice 7.3.30 from \cite{Fog02} and   \cite[Lemma 1]{Host:1986}. Similar extensions can be found e.g. with 
\cite[Lemma 15]{Bressaud-Durand-Maass:2005}.
\begin{lemma}\label{lem:eigenvectorbis} Let
$(M _n)_n \in {\mathcal M}^{\mathbb N}$ where ${\mathcal M}$ is a finite set of $d\times d$ invertible matrices with integer entries.  Assume that $(M_n)_n$ is strongly convergent with respect to a normalized vector  $\boldsymbol u  \in  {\mathbb R}_{\ge 0}^d$.   Let $t\in {\mathbb R}$.
 Suppose that
\[\lim_{n\rightarrow \infty }t{\bf 1}  M_{[0,n)} \equiv {\bf 0} \bmod \Z^d.\]
Then  the following holds.

\begin{enumerate}
\item   The vector  $t{\boldsymbol 1}$  can be decomposed   as $t{\bf 1} = { \boldsymbol t }_1+ {\boldsymbol t}_2$ where  ${\boldsymbol t} _1 M_{[0,n)}
\rightarrow 0$ and 
$ {\boldsymbol t}_2 M_{[0,n)}
 \in \Z^d$ for all $n$ large. 
 
 \item The vector  $ { \boldsymbol t }_1$ is orthogonal to  $\boldsymbol u$.

\item   
The   number $ t$ is a rational  combination of the  entries of  $\boldsymbol u$.
 If  the matrices $M_n$ are all unimodular, then   $t$ is an integer   combination of the  entries of  $\boldsymbol u$.
 \item  For each $a$, let $ h_n(a) = (1...1)M_{[0,n)}e_a$. There exists $C>0$ such that  for each letter $a$ and each $n$, then 
 \[          |\lambda^{h_{n}(a)}   - 1|  \leq
    C    \mathrm{dist}(M_{[0,n)}\boldsymbol e_a,\R\boldsymbol u) .\]

 \end{enumerate}

\end{lemma}
\begin{remark}
Note that the condition $\lim_{n\rightarrow \infty }t{\bf 1}  M_{[0,n)} \equiv {\bf 0} \bmod \Z^d$ is equivalent to    
$ \lim_n  \lambda ^{h_n(a) }= 1 \mbox{ for all } a$ where $\lambda = \exp(2i \pi t)$. Note that Assertion (iv) above also implies that if sum convergence holds, i.e., if the series  
$\sum_n    \mathrm{dist}(M_{[0,n)}\boldsymbol e_a,\R\boldsymbol u) $  converges, then  the series  $\sum_n    |\lambda^{h_{n}(a)}   - 1|$ converges.

\end{remark}

\begin{proof}

{\bf Proof of (i)}.
 By assumption, we can write
\[
t {\boldsymbol  1}   M_{[0,n)}      = {\boldsymbol u}_n  +{\boldsymbol v}_n\]
where ${\boldsymbol u}_n\in \Z^d$ and ${\boldsymbol v}_n\rightarrow  {\bf 0}$ as $n\rightarrow \infty$.
Note that for each $n$
\begin{align*}
{\boldsymbol u}_{n+1} + {\boldsymbol v}_{n+1} &= t {\bf 1}  M_{[0,n+1)}      = \left(  {\boldsymbol u}_n  +{\boldsymbol v}_n\right)    M_{{n}} \\&    =   {\boldsymbol u}_n  M_{{n}}  +{\boldsymbol v}_n M_{{n}}  ,
\end{align*}
so 
\[ {{\boldsymbol u}_{n+1} -  {\boldsymbol u}_n            M_{{n}}              =  {\boldsymbol v}_n M_{{n}}  - {\boldsymbol v}_{n+1}} ,\]
and since  $M_{{n}}$ is one of finitely many matrices, the right hand side of this last expression converges to ${\bf 0}$, so  therefore $ {\boldsymbol u}_{n+1} -  {\boldsymbol u}_n     M_{_{n}}     \rightarrow {\bf 0}$  as $n\rightarrow \infty$.
But  $ {\boldsymbol u}_{n+1} -  {\boldsymbol u}_n   M_{{n}}  $ is an integer valued row vector. We conclude that there exists $N$ such that  \[ {\boldsymbol u}_{N+1} -  {\boldsymbol u}_N   M_{{N}}  = {\bf 0}\mbox{ and }
 {\boldsymbol u}_{N+m} -  {\boldsymbol u}_N    M_{[N,N+m)}  = {\bf 0}
 \mbox{ for all } m\geq 1.\]

 As each of the matrices $M_n$ is invertible, we can find a vector $\boldsymbol t_2$ such that $\boldsymbol t_2  M_{[0,N)} = {\boldsymbol u}_N$, and so
 \[  {\boldsymbol u}_{n} = \boldsymbol t_2
  M_{[0,n)}   
   \mbox{ for } n\geq N. \] 
 
 Write $\boldsymbol t_1 :=  t {\bf 1}  - \boldsymbol t_2$. Then,  for every $n \geq N$, one has 
 \[    \boldsymbol t_1  M_{[0,n)}       = t  {\bf 1}    M_{[0,n)}  -  \boldsymbol t_2    M_{[0,n)}  = t  {\bf 1}   M_{[0,n)} - {\boldsymbol u}_{n} = {\boldsymbol v}_{n}. \]
 But  ${\boldsymbol v}_n \rightarrow {\bf 0}$ as $n\rightarrow \infty$, so  $\boldsymbol t_1    M_{[0,n)} \rightarrow  {\bf 0}$ as $n\rightarrow \infty$.

 \bigskip
{\bf Proof of (ii)}. 
By the assumption of strong convergence, we have 
$\lim_{n\to+\infty} \mathrm{dist}(M_{[0,n)}\boldsymbol e_a,\R  \boldsymbol u) = 0$ for each $a$. In particular, cone convergence holds. By continuity of the scalar product, one has 
$\lim_n \frac{\langle {\bf  t}_1, M_{[0,n)}\boldsymbol{u}\rangle}{\vert \langle {\bf  t}_1, M_{[0,n)}\boldsymbol{u}\rangle \vert } = 
\frac{\langle {\bf  t}_1, \boldsymbol{u}\rangle}{\vert \langle {\bf  t}_1, \boldsymbol{u}\rangle\vert }$. Moreover, $\lim_n \langle {\bf  t}_1, M_{[0,n)}\boldsymbol{u}\rangle =0$ from above.
 Therefore $\langle{\bf  t}_1,{\boldsymbol u}\rangle = 0$.

\bigskip

{\bf  Proof of (iii).} 
One has   $t {\bf 1 }= {\boldsymbol t} _1 + {\boldsymbol t} _2$, 
$ t =  \left\langle   t { \bf 1}, {\boldsymbol u} \right\rangle =
\langle \boldsymbol t_1+\boldsymbol t_2, {\boldsymbol u} \rangle =
\langle \boldsymbol t_2,{\boldsymbol u}\rangle$,
and
 $ \boldsymbol t_2= {\boldsymbol u}_{N+1}  \left(   M_{[0,N+2)}           \right)^{-1}$. Here ${\boldsymbol u}_{N+1} $ has integer entries,  as well as the invertible matrices
 $M_n$, so the entries of $\boldsymbol t_2$ are also   rational numbers. Hence $t$ is a  rational combination of the entries of ${\boldsymbol u}$, and if  the matrices are unimodular, then it is even  a  linear  combination of the entries of ${\boldsymbol u}$.

\bigskip

{\bf  Proof of (iv).}   
By Part (i), $ t{\bf 1} M_{[0,n)} \equiv  {\bf t}_1  M_{[0,n)} \mod  \Z^d$  for all $n$.
Write  $\lambda^{h_{n}(a)}= \exp(2i \pi  t h_n(a))$.
For each $a$ and each $n$, one has
 $$  |\lambda^{h_{n}(a)}  -1| \leq 2 \pi  | \langle \tr{ \boldsymbol t}_1 ,M_{[0,n)}\boldsymbol e_a \rangle |  
  \mbox{ and }   |  \langle \tr{ \boldsymbol t } _{1} , M_{[0,n)}\boldsymbol e_a \rangle | \leq  \Vert { \boldsymbol t } _{1}\Vert_2 \,   \mathrm{dist}(M_{[0,n)}\boldsymbol e_a,\R\boldsymbol u)
, $$
by the Cauchy-Schwartz inequality together with the fact that $ \boldsymbol t  _{1}$ is orthogonal to ${\boldsymbol u}$ for the second inequality.
\end{proof}

 The following result   relates  eigenvalues  and measures of letter  cylinders.

\begin{theorem} \label{theo:Pisot}
Let $\boldsymbol{\sigma} = (\sigma_n)_{n\ge0}$ be a finitary, straight, recognizable directive sequence  on $\mathcal A$, 
 where each substitution occurs infinitely often, and where each incidence matrix is invertible. 
Suppose that   $(X_{\boldsymbol{\sigma}},T)$ is
   aperiodic, and that 
    $\boldsymbol{\sigma} $ is  strongly  convergent.  Let $\mu$ be the (unique) shift-invariant measure on $(X_{\boldsymbol{\sigma}}, T)$. Let $\lambda=e^{2i\pi t}$  in ${\mathbb S}^1$ such that
$\lim_n  \lambda ^{h_n(a)}=1$ for every $a \in {\mathcal A}$. 
If  
\[\sum_n    \max_a  \mathrm{dist}(M_{[0,n)}\boldsymbol e_a,\R\boldsymbol u) <\infty,\] then
$\lambda$  is a continuous eigenvalue, and $t$ is a $\mathbb Q$-linear combination of   the  measures of the letter cylinders $\mu[a]$, and  even a $\mathbb Z$-linear  combination if 
   $\boldsymbol{\sigma}$ is unimodular.     Finally, $ \mu[a]$  is a continuous eigenvalue for  every letter $a$.
\end{theorem}
 \begin{proof}
   Strong convergence, together with the assumption of having invertible matrices,  implies unique ergodicity \cite[Theorem 5.7]{Berthe-Delecroix}, and  similarly as in the substitutive case,    the  entries of the   generalised   normalised right eigenvector $\boldsymbol u$ are given by  the measures $\mu[a]$ of  one-letter cylinders.
  The fact that $\lambda$ is a continuous eigenvalue comes from    Theorem   \ref{thm:sufficient-continuous-A}, with the constant trivial coboundary  $(h,\bar f)$,
  with $h\equiv 1$ and $\bar f \equiv 1$, 
   and 
 Condition (iv) of  Lemma \ref{lem:eigenvectorbis}.  
  Lemma  \ref{lem:eigenvectorbis}
 then  implies that  $t$ is a $\mathbb Q$-linear combination of   the  measures of the letter cylinders $\mu[a]$, and if 
 $\boldsymbol{\sigma}$ is unimodular,   $t$ is a $\mathbb Z$-linear  combination. 
 
 It remains to prove that  for each letter $a$,  $t=\mu[a]$ satisfies  $\lim_n  \lambda ^{h_n(a)}=1$, where $\lambda= e^{2\pi i t}$. The  projection of $M_{[0,n)}\boldsymbol e_a$ along    the hyperplane orthogonal to the vector ${\bf 1}$ onto   $ {\mathbb R}^+\boldsymbol u$ is   $ h_n(a) \boldsymbol u$, whose $a$th entry is
 $\mu[a]h_n(a)$.  Since the matrix $M_{[0,n)}$ has integer entries, the  convergence to $0$ 
of $   \mathrm{dist}(M_{[0,n)}\boldsymbol e_a,\R\boldsymbol u) 	$  implies    the convergence of 
$\mu[a]h_n(a)$ to $0$ modulo $1$.  We  again conclude with Theorem   \ref{thm:sufficient-continuous-A},
 together with Condition (iv) of  Lemma \ref{lem:eigenvectorbis} that $\mu[a]$ is a continuous eigenvalue.
\end{proof}

 \begin{remark} We can also use the fact that 
$\sum _n \max_a  \mathrm{dist}(M_{[0,n)}\boldsymbol e_a,\R\boldsymbol u) < + \infty 	$  implies the so-called balance property under the  finitary  assumption
 (see  \cite[Theorem 5.8]{Berthe-Delecroix}). We  then again deduce that $\mu[a]$ is a continuous eigenvalue, via  \cite[Theorem 1]{BerCecchi:2018}, which relies on  the Gottshalk-Hedlund theorem.
\end{remark} 
 
 \begin{remark}
 The convergence condition  of Theorem  \ref{theo:Pisot} holds  in particular if one has   exponential convergence for the directive sequence but we can go beyond here: 
the convergence  just needs to be sufficiently fast for  the series to converge. This has to be compared  with the measurable case where the convergence is slower (see Theorem \ref{thm:sufficient-mble}). This is a classical phenomenon observed for some   proper  directive sequences; see  \cite{DFM:19} or  \cite{Cassaigne-Ferenczi-Messaoudi:08} for Arnoux-Rauzy   shifts; see also \cite[Corollary 15.57]{Nadkarni-1998}.
\end{remark}

\section{ Measurable eigenvalues}\label{measurable}

 In this section we equip an $S$-adic shift with an invariant measure and  study the existence of measurable eigenvalues. For this, Theorem \ref{thm:Bratteli-Vershik} is particularly useful, as, with the conditions that the directive sequence is recognizable and everywhere growing, it allows us to use and compare results concerning  Bratteli-Vershik systems. In particular
  we state results from  \cite{BKMS:13},  \cite{Mentzen:1991} and \cite{CDHM:2003}    which are couched in terms of Bratteli diagrams, or cutting-and-stacking transformations, but  which we rephrase   in our language. Our main result in this section is Theorem \ref{thm:sufficient-mble}. It requires a strengthening of the notion of straightness to that of strong straightness, in  Definition \ref{def:strongly-straight}. The combination of  strong primitivity and finitary ensures that our systems are of {\em exact finite rank},
a condition which has been extensively studied in the literature  \cite{Mentzen:1991, Boshernitzan-1992}, and one which is useful in the proof of  Theorem \ref{thm:sufficient-mble}. We recall the partitions $(\mathcal Q_n)$ that were defined in  Section \ref{sec:generating-partitions}.

\begin{lemma}\cite[Lemma 6.3]{BSTY}\label{partition-generates-mble}
  Let $\boldsymbol{\sigma} = (\sigma_n)_{n\ge0}$ be a recognizable and everywhere growing directive sequence on
  $\mathcal A$.
If  $\mu$ is a shift invariant probability measure on~$X_{\boldsymbol{\sigma}}$, then one has  $\mu(\bigcap_{n=0}^{\infty}   \bigcup_{a\in\mathcal A} B_n(a)) = 0$, and $(\mathcal Q_n)$ is generating in measure.
\end{lemma}

We give an indication of the proof of Lemma \ref{partition-generates-mble}. In fact, the partitions  $\mathcal{Q}_n$ determine the past of any point which is not in the $T$-orbit of a limit word. This is because for such a point $x$, we have that infinitely often $x$ belongs neither to the base   $\sigma_{[0,n)}([a])$ of any $\mathcal T_n(a)$, nor to the top level         $T^{h_n(b)-1} \sigma_{[0,n)}([b])$ of any $\mathcal T_n(b)$. Therefore, if $n$ is such that $x\not \in \bigcup_{a\in \mathcal A}\sigma_{[0,n)}([a])\cup \bigcup_{b\in \mathcal A}  T^{h_n(b)-1} \sigma_{[0,n)}([b]) $, then we have complete knowledge of 
$x_{[-k_n,k_n)}$, where $k_n= \min_{a\in \mathcal A}h_n(a)$. We then  let  $n$ become large and use the hypothesis that $\boldsymbol{\sigma}   $ is everywhere growing.

 \bigskip
 
We now  introduce the following notion, which is key to the proof of Theorem \ref{thm:sufficient-mble}.

\begin{definition}[$S$-adic exact finite rank]
We say that the directive sequence  ${\boldsymbol \sigma}$ is of  {\em exact finite
rank} if it is primitive and recognizable,   and if there is an  invariant probability measure $\mu$ on $X_{\boldsymbol \sigma}$ and a constant $\delta>0$ such that  $\mu(  \mathcal T_n( a))  \geq \delta$ for all $n\geq 0$ and all $a \in \mathcal A$.  
\end{definition} 
We keep in mind that $\mu(  \mathcal T_n( a)) =h_n(a) \mu(B_n(a))$. Thus if the heights of the towers  $\mathcal T_n( a)$ and  $\mathcal T_n( b)$ are comparable in an exact finite rank system,  then so is the $\mu$-mass of their bases. Note also  that   the vector with entries $(\mu(B_n(a)))_a$ is equal to  the generalised  normalised right  eigenvector   of  the sequence of matrices 
$(M_{k})_{ k \geq n}$,  i.e.,  $ \mu([a])_a= M_{[0,n)} (\mu(B_n(a)))_a $ for a finitary,  primitive and strongly convergent  directive sequence made of invertible matrices  (see e.g. \cite[Section 6.8.1]{CANT} and \cite[Theorem 3.10]{Berthe-Delecroix} for the existence of frequencies under these assumptions). Indeed  one has $\mu(B_n(a))= (M_{n} \mu(B_{n+1}(b))_b)_a$.

Even though the 
notion of exact finite rank is essentially one that describes the given dynamical system,  nevertheless whether it holds  is a function of the given $S$-adic representation of the system, for, as discussed in 
\cite[Remark 6.9]{BKMS:13}, a dynamical system may have two $S$-adic representations, only one of which is of  exact finite rank.
Boshernitzan showed that dynamical systems which have an exact finite rank representation are uniquely ergodic by   \cite[Theorem 1.2]{Boshernitzan-1992}, so we will denote the unique invariant measure by  $\mu$.

 Finally we repeat here a condition,  \cite[Proposition 5.7]{BKMS:13}, which identifies a large family of exact finite rank directive sequences. Note that most of the results in \cite{BKMS:13} concerning exactness do not depend on the ordering of the Bratteli diagram, i.e., there is no requirement that the substitutions in the directive sequence be proper.

 \begin{proposition}\label{prop:exact}
Let ${\boldsymbol \sigma}$ be a recognizable directive sequence   on $\mathcal A$, where the incidence matrix  $M_{\sigma_n} =\left(    m_{ab}^{(n)}\right)_{a,b}$ of $\sigma_n$ is positive for each $n$. If there is a constant $c>0$ such that 
\[ \frac{\min_{a,b} m_{ab}^{(n)}}{ \max_{a,b} m_{ab}^{(n)} }\geq c\]
for each $n$, then ${\boldsymbol \sigma}$ is of exact finite rank, and  $ h_{n}(a)/   h_{n}(b)\geq c$ for each $n$ and pair of letters $a,b$.
\end{proposition}
As a direct consequence of Proposition \ref{prop:exact}, we get that if the system is finitary  and strongly primitive, then ${\boldsymbol \sigma}$ is of exact finite rank, and thus uniquely ergodic  as recalled above by  \cite[Theorem 1.2]{Boshernitzan-1992}.  Exact finite rank systems and linearly recurrent shifts enjoy similar properties, though they are different.
Linear recurrence  requires specific  additional combinatorial properties.

  \begin{theorem}\cite[Proposition 1.1]{Durand:00b}
The subshift $(X,T)$ is linearly recurrent if and only if there exists a strongly primitive, finitary and  proper directive sequence ${\boldsymbol \sigma}$ such that $(X,T)=(X_{\boldsymbol \sigma }, \sigma)$.
\end{theorem}

 In particular,  linearly recurrent systems have an exact finite rank representation.
   However, the next example shows that  exact finite rank directive sequences are not necessarily linearly recurrent.   
\begin{example}\label{ex:running-fourth}

 We have shown that  the directive sequence
 \[{\boldsymbol \sigma}=\sigma, \tau , \sigma, \sigma, \tau,  \sigma, \sigma , \sigma, \tau,  \sigma, \sigma,\sigma \dots \]
 from Example~\ref{ex:running-first}
 is recognizable. 
 Note that  $\tau \circ \sigma$ has a positive incidence matrix, as do $\sigma\circ \sigma$ and $\sigma\circ  \tau \circ \sigma$. Hence we can boundedly telescope the  directive sequence ${\boldsymbol \sigma}$ as 
   \[\sigma\circ \tau \circ \sigma,  \sigma\circ \tau \circ \sigma, \sigma \circ \sigma,\tau \circ \sigma, \sigma\circ \sigma, \sigma\circ \tau \circ \sigma, \sigma\circ \sigma, \sigma\circ \sigma,\tau \circ \sigma, \dots \]
   where this directive sequence takes values in $\{    \sigma\circ \tau \circ \sigma,   \sigma \circ \sigma, \tau \circ \sigma   \}$. We conclude by Proposition \ref{prop:exact} that $(X_{\boldsymbol \sigma},T)$ is of exact finite rank, although it is not linearly recurrent  such as stressed in~\cite{Durand:00b}.
   
\end{example}

We now introduce the $S$-adic counterpart of the notion of strong straightness from Definition \ref{def:straight-substitution} (stated in the stationary case). 
This notion allows one to relax the condition of properness, which we often do not have with a given directive sequence, even if it defines a linearly recurrent shift. For example, the conditions of Theorem \ref{thm:sufficient-mble} imply that the shift is linearly recurrent, by \cite[Lemma 3.1]{Durand:00b}, without assuming properness.

\begin{definition}[Strong straightness]\label{def:strongly-straight}
Let ${\boldsymbol \sigma}$ be a  straight directive sequence on $\mathcal A$. For each letter $a\in \mathcal A$  and for each $n$, let  $a_n$ be the    first letter of $\sigma_n(a)$. We call  ${\boldsymbol \sigma}$ {\em strongly straight} if \begin{enumerate} \item for each letter $a\in \mathcal A$ and each non-negative  $n$,  $\sigma_n(a_{n+1})$ starts with $a_n$,\item  if $ {\boldsymbol a}= {\boldsymbol b} $, then $(a_n)=(b_n)$, and
\item there exists $r$ such that for each $n$, for each two-letter word $\alpha\beta\in \mathcal L^{(n)}_{\boldsymbol \sigma}$, there exists a letter $\gamma$ such that $\alpha \beta$ is a subword of $\sigma_{[n, n+r)}(\gamma)$.
\end{enumerate}\end{definition}

Note that if ${\boldsymbol \sigma}$ is stationary and  if  its constant term $\sigma$ is strongly straight in the sense of Definition \ref{def:straight-substitution}, then it satisfies the conditions of Definition \ref{def:strongly-straight}. Moreover a strongly  straight  directive sequence is a  straight directive sequence, i.e., each letter $a$ defines a unique right-infinite limit word  ${\boldsymbol a}$.
Indeed Condition (i) guarantees that the prefix $\sigma_{[0,n)}(a_n)$ of $\sigma_{[0,n+1)}(a)$ is also a prefix of $\sigma_{[0,n+k)}(a)$ for each $a$, $n\geq 1$ and $k\geq 0$, and  it is also a prefix of the  unique right-infinite  limit word ${\boldsymbol a}$ defined by $a$, since   the sequence of words
 $(\sigma_{[0,n)}(a_{n}))$ is a nested sequence of words that increases in length.  See Figure \ref{fig:strongly-straight} for an illustration. Note also  that  Condition (ii) holds if ${\boldsymbol \sigma}$ is one-sided recognizable.
 We stress the fact that strongly  straight  directive sequences  are distinct from proper ones. For instance,  a directive sequence made of  substitutions that are such that   the image  of each  letter  $a$ starts with $a$  is an example of  a strongly   straight   sequence.  See  also
 Condition (i) in Example  \ref{ex:nontrivial height}.

 \begin{example}
 The directive sequence  ${\boldsymbol \sigma}$ from our running example~\ref {ex:running-first} is not strongly straight since ${\mathbf  a}={\mathbf  b}$,
 $a_n=a$ for all $n$,  and $b_n=b$ for all $n$ such that $\sigma_n=\sigma$. Moreover one checks that  Condition (iii) does not hold.
 \end{example}

The following example shows that there exist $S$-adic systems  that are  finitary, recognizable, strongly prefix-straight  and  not linearly recurrent. However note that the combination of the conditions
finitary, strongly  straight and strongly primitive implies linear recurrence, since factors of length $2$ occur with  uniform bounded  gaps, by  \cite[Lemma 3.1]{Durand:00b}.
\begin{example} 
 Consider  the substitutions  $S=\{\sigma, \tau\}$ with
 \begin{align*}
   a   \stackrel{\sigma}{\mapsto}  aaaa  \ \  & a    \stackrel{\tau}{\mapsto}  bbbb
   \\  b  \stackrel{\sigma}{\mapsto}  abaa \ \  & b  \stackrel{\tau}{\mapsto}  babb,
\end{align*}
and let
\[  {\boldsymbol{\sigma}}= \tau, \sigma, \tau, \sigma, \sigma, \tau, \sigma, \sigma , \sigma, \tau , \sigma, \sigma ,\dots \,. \]
We follow here the same scheme  of proof as in \cite[Section 2]{Durand:00b}.
One verifies that the difference between two successive occurrences of the word $abb$ in $\tau \circ  \sigma^n    (z) $ is greater than $ 4^n$   for any sequence $z$.
For $ n\geq 1$, let $\rho_n=  \tau \circ  \sigma \circ  \tau \circ  \sigma^2  \circ  \cdots \circ \tau \circ  \sigma^{n-1} $. One checks that 
if $w$ is a strict  return word of  $abb$ in $\tau \circ  \sigma^n  (z) $, then   $\rho_n(w)$ is a strict return word of $\rho_n(abb)$  in $X_{\boldsymbol{\sigma}}$. 
Thus
there exists $C>0$ such that   $\frac{|\rho_n(w)|} {|\rho_n(abb)|} \geq  C 4^n$. 
This implies that $X_{\boldsymbol \sigma}$ is not linearly recurrent.\end{example}

\begin{figure}

 \definecolor{red(munsell)}{rgb}{0.95, 0.0, 0.24}
\definecolor{mediumelectricblue}{rgb}{0.01, 0.31, 0.59}
 \definecolor{lavender(floral)}{rgb}{0.71, 0.49, 0.86}
\definecolor{lavenderblue}{rgb}{0.8, 0.8, 1.0}
\definecolor{darkgreen}{rgb}{0.0, 0.2, 0.13}
\definecolor{cadmiumgreen}{rgb}{0.0, 0.42, 0.24}

\begin{tikzpicture}
[-,>=stealth',shorten >=2pt, shorten <=2pt,auto,node distance=2.0cm,semithick] 
 \tikzstyle{every state}=[fill=white,draw=black,text=black]

  \node[state] 		(D)       {$a$};
  \node[state] 		(E)  [right of=D]       {$a_{n+2}$};
  \node[state] 		(F)  [right of=E]       {$b$};

  \node[state] 		(D1)  [above of=D]     {$a$};
  \node[state] 		(E1)  [right of=D1]     {$a_{n+1}$};
 \node[state] 		(F1)  [right of=E1]     {$b$};
 
  \node[state] 		(D2)  [above of=D1]     {$a$};
  \node[state] 		(E2)  [right of=D2]     {$a_n$};
 \node[state] 		(F2)  [right of=E2]     {$b$};

  \node[state] 		(D3) [above of=D2]      {$a$};
  \node[state] 		(E3)  [right of=D3]       {$a_{n-1}$};
  \node[state] 		(F3)  [right of=E3]       {$b$};
      \path
  
	(E)  edge [line width=1.2]   node[yshift=-10pt, xshift=12pt, color=black]  { $0$}  (E1)
 (E1)  edge [line width=1.2]   node[yshift=-10pt, xshift=12pt, color=black]  { $0$}  (E2)
 (E2)  edge [line width=1.2]   node[yshift=-10pt, xshift=12pt, color=black]  { $0$}  (E3)
 
      (D) edge[yshift=2pt, line width =1.2,  dashed] node[yshift=-10pt, xshift=12pt, color=black, ]  { $0$}  (E1)
      (D1) edge[yshift=2pt, line width =1.2, dashed] node[yshift=-10pt, xshift=12pt, color=black, ]  { $0$}  (E2)
      (D2) edge[yshift=2pt, line width =1.2,  dashed] node[yshift=-10pt, xshift=12pt, color=black, ]  { $0$}  (E3)
      
       (F) edge[yshift=2pt, line width =1.2,  dashed] node[yshift=-10pt, xshift=12pt, color=black, ]  { $0$}  (E1)
      (F1) edge[yshift=2pt, line width =1.2,  dashed] node[yshift=-10pt, xshift=12pt, color=black, ]  { $0$}  (E2)
      (F2) edge[yshift=2pt, line width =1.2, dashed] node[yshift=-10pt, xshift=12pt, color=black, ]  { $0$}  (E3)

	 ;
                            \end{tikzpicture}
  \caption{Part of the  natural Bratteli diagram for a strongly  straight directive sequence, where the letters $a$ and $b$ determine the same right-infinite limit word. Dashed edges correspond to minimal edges from $a$ or $b$ which lead directly to  the limit word ${\boldsymbol a}={\boldsymbol b}$.
   The solid edges depict the limit word   ${\boldsymbol a}$.  }
   \label{fig:strongly-straight}
   \end{figure}

We remark that the notion of a strongly straight directive sequence is restrictive. However the conditions are not difficult to verify.
We have  gathered all notions required to  state  the following generalisation of Proposition \cite[Proposition 8]{CDHM:2003}.   Its topological counterpart is given by  Theorems \ref{thm:sufficient-continuous-A}.  
We start with a preliminary lemma. For all $n$, let $\mathcal B_n$ denote the $\sigma$-algebra generated by $\mathcal Q_n$ 
 and let 
$\E_n:L^1(X_{\boldsymbol{\sigma}},\mathcal B,\mu)\rightarrow L^1(X_{\boldsymbol \sigma},\mathcal B_n, \mu)$ be  the conditional expectation operator; write $f_n= \E_n(f)$. 
 The function $f_n$ is constant on elements of $\mathcal Q_n$. The sequence $(f_n)$ is a martingale and converges to $f$ in $L^2(\mu)$. 
\begin{lemma}\label{lem:bastard-referee}
Suppose that $(X_{\boldsymbol \sigma}, T,\mu) $ is of exact finite rank, with eigenfunction $f$  such that $|f|=1$. 
For each $a\in \mathcal A$, let 
 $v_n(a): = \E_n(f) $ on $B_n(a)$. Then 
 $|v_n(a)|\rightarrow 1.$
\end{lemma}
\begin{proof}
As $f$ is an eigenfunction, for the eigenvalue $\lambda$, then  $\E_n(f)=f_n=\lambda^j v_n(a)$ on $T^j (B_n(a))$ for  $0\leq j < \sigma_{[0,n)}(a)$,  and so  for all $x\in X_{\boldsymbol \sigma}$
$$f_n (x)= \sum_{a \in  {\mathcal A }} \sum _{j=0}^{|\sigma_{[0,n)(a)}|-1}  v_n(a) \lambda^j  \chi _{T^j B_n(a)}(x),$$
where $\chi$ refers to the indicator function.
Therefore
$||f_n||_2^2 = \sum_{a\in \mathcal A}|v_n(a)|^2 \mu({\mathcal T}_n(a))$. As  $|v_n(a)|= \left | \frac{1}{\mu(B_n(a))} \int_{B_n(a)}    f(x) \, d\mu(x)  \right |\leq 1$, then for any $a\in \mathcal A$, we have 
\begin{align*}
 ||f_n||_2^2 &\leq  |v_n(a)|^2\mu(   {\mathcal T}_n(a)  ) + \mu(  \cup_{b\neq a} \mu( {\mathcal T}_n(b) ) )\\
 & \leq   |v_{n}(a)|^2\mu(   {\mathcal T}_{n}(a)  ) + 1- \mu(   {\mathcal T}_{n}(a)  ) =    ( |v_{n}(a)|^2 -1)\mu(   {\mathcal T}_{n}(a)  ) + 1. 
\end{align*}
Now suppose that  $|v_{n_k}(a)| \rightarrow \rho<1$ for some  subsequence $(n_k)_k$.  The hypothesis of exact finite rank tells us that there is a  $\delta>0$ such that $\mu(   {\mathcal T}_n(a)  )\geq \delta$ for each $n$. Thus for $\varepsilon$ sufficiently small
 \[ \lim _k||f_{n_k}||_2^2 
 < (\rho + \varepsilon -1)\delta +1<1\] for $k$ large enough,
 contradicting the fact that $||f_n||_2^2 \rightarrow ||f||_2^2 =1$.
\end{proof}

Recall the comment after Proposition \ref{prop:exact}  that finitary, strongly primitive directive sequences define uniquely ergodic shifts $(X_{\boldsymbol \sigma}, T,\mu)$.

\begin{theorem}\label{thm:sufficient-mble}
Let ${\boldsymbol \sigma}$ be a finitary, strongly  straight, recognizable and strongly primitive   directive sequence on $\mathcal A$. If $\lambda\in  {\mathbb S}^1$ is a measurable eigenvalue of $(X_{\boldsymbol \sigma}, T,\mu)$,  then

\[   \sum_{n=1}^{\infty} | \lambda^{h_n(w_n(a,b))}   -1     |^2 < \infty \]
for  any  $a,b$  such that  ${\boldsymbol a}= {\boldsymbol b}$
and  any $w_n(a,b)\in \{w\in \mathcal L_{\boldsymbol \sigma}^{(n)}: w\mbox{ is a strict transition word from $a$ to $b$} \}$. 
\end{theorem}
\begin{proof}
We follow \cite{CDHM:2003}. We recall from Section \ref{sec:generating-partitions} that 
given $a,b\in \mathcal A$ and $n$, 
$$ t_n(a,b) = \{ 0\leq t<  h_n(a): T^t (B_n(a))\subset B_{n-1}(b)\},$$
 and 
the cardinality of $t_n(a,b)$ equals the $(b,a)$ entry of $M_{\sigma_{n-1}}$.  In particular $ t_n(a,b)\neq \emptyset$ if and only if $b$ appears in $\sigma_{n-1}(a)$, and $0\in
 t_n(a,b)$ means that $b$ is the first letter of $\sigma_{n-1}(a)$. See Figure \ref{fig:partitionQ} for an illustration.

 Let $f$ be an  eigenfunction for $\lambda$. 
 As $|\lambda|=1$, and $(X_{\boldsymbol \sigma}, \sigma)$ is minimal, we can take $|f|=1$ almost everywhere. 
 Recall that 
 $f_n= \E_n(f)$ where  $\E_n$ is the conditional expectation defined by the $\sigma$-algebra generated by $\mathcal Q_n$.
  The functions $f_n-f_{n-1}$ are mutually orthogonal (see for example \cite{Doob-1953} for an exposition on martingales), so that $\sum_{n=1}^{\infty}  ||
  f_{n-1}-f_n||_2^2<\infty$.

 Since ${\boldsymbol \sigma}$ is  strongly primitive, there  exists $r$ such that for each $n$, 
$\sigma_{[n,n+r)}$ has a  positive incidence matrix. As ${\boldsymbol \sigma}$  takes values from a finite set, we can telescope ${\boldsymbol \sigma}$ so that  it takes values from a finite set and  all incidence matrices are positive.
 This implies that there exists $c$ such that
$$\frac{\min_{a,b}m_{ab}^{(n)}}{\max_{a,b}m^{(n)}_{ab}}\geq c,$$ 
which   allows us to conclude, by  Proposition \ref{prop:exact}, that  ${\boldsymbol \sigma}$ is of exact finite rank, so that there is $\delta>0$ with $h_n(a) \mu (B_n(a))\geq \delta$ for each $a$.

 Let $j\in t_n(a,b)$. Recall
 $v_n(a): = \E_n(f) $ on $B_n(a)$. Then, for each $0\leq p < h_{n-1}(b)$, we have $T^{j+p}B_n(a) \subset T^{p}B_{n-1}(b)$, and if  $x\in T^{j+p}B_n(a)$, then $f_n(x)= \lambda^{j+p}v_n(a)$, $f_{n-1}(x) = \lambda^{p}v_{n-1}(b)$.
 So
 \begin{equation}\label{eq:prel} ||f_n-f_{n-1}  ||_2^2\geq h_{n-1}(b) \mu(B_{n}(a)) |     \lambda^{j}v_n(a)-v_{n-1}(b)    |^2 \geq \delta  \frac{h_{n-1}(b)}{h_n(a)}     |     \lambda^{j}v_n(a)-v_{n-1}(b)    |^2 .\end{equation}

Let $L$ be the maximum sum   of the entries of a column of any of the incidence matrices, and let $b^{*}$ be such that $h_{n-1}(b^{*})= \min_{b\in \mathcal A} h_{n-1}(b)$. One has 
 \begin{align*}
  \max_{a,b}\frac{h_{n}(a)}{       h_{n-1}(b)} &= \frac{ \max_{a}h_{n}(a)}{  \min_{b}     h_{n-1}(b) }  =  \frac{ \max_{a}h_{n}(a)}{       h_{n-1}(b^{*}) }   \leq \frac{L \max_{b\in \mathcal A} h_{n-1}(b)}{         h_{n-1}(b^{*})   } \leq \frac{L}{c},
\end{align*}
 where the last statement follows from Proposition \ref{prop:exact}. This implies that $|\lambda^{j}v_n(a)-v_{n-1}(b)    |^2 \leq  \frac{L}{\delta c}     ||f_n-f_{n-1}  ||_2^2$. Taking the maximum, over all $j\in t_n(a,b)$ and $a,b\in \mathcal A$, we have  
\begin{equation} \label{eq:general-l2}  \sum_{n=1}^{\infty} \max_{a,b} \max_{j\in t_n(a,b)}    |     \lambda^{j}v_n(a)-v_{n-1}(b)    |^2       < \infty .\end{equation}

We use the notation of Definition \ref{def:strongly-straight}, writing that $\sigma_{n-1}(a)$ starts with $a_{n-1}$.
This gives 
$0\in t_{n}(a,a_{n-1})$,  and letting $j=0$ in \eqref{eq:general-l2}, we get $$\sum_{n=1}^{\infty}    |     v_n(a)-v_{n-1}(a_{n-1})    |^2       < \infty.$$
 Now assume that     ${\boldsymbol a}= {\boldsymbol b}$. Since ${\boldsymbol \sigma}$ is strongly straight,
     then in addition to  $0\in t_{n}(a_{n},a_{n-1})$ for each $n$ (i.e., $\sigma_{n-1}(a_n)$ starts with $a_{n-1}$), 
  we  also have  $0\in t_{n}(b,a_{n-1})$ for each $n$ (i.e., $\sigma_{n-1}(b)$ starts with $a_{n-1}$), since $b_{n-1}=a_{n-1}$. See also Figure \ref{fig:strongly-straight}.    
   
Using \eqref{eq:general-l2} two more times with $j=0$, we get    \[    \sum_{n=2}^{\infty}    |     v_{n-1}(a_{n-1})-v_{n-2}(a_{n-2})    |^2       < \infty   \mbox{ and } \sum_{n=2}^{\infty}    |     v_{n-2}(a_{n-2})-v_{n-1}(b)    |^2       < \infty, \]
   and we obtain
 \begin{equation} \label{eq:l2}  \sum_{n=1}^{\infty} \max_{a,b:    {\boldsymbol a}= {\boldsymbol b}         }     |     v_n(a)-v_{n-1}(b)    |^2       < \infty .\end{equation}
As  $\boldsymbol{\sigma}$ is of exact finite rank, 
 then by Lemma 
\ref{lem:bastard-referee},
 $|v_n(a)|\rightarrow 1$ for each $a\in\mathcal A$. Therefore from \eqref{eq:general-l2} and \eqref{eq:l2}, we obtain 
  \begin{align}\label{eq:square-summable}   \sum_{n=1}^{\infty} \max_{a,b:    {\boldsymbol a}= {\boldsymbol b}         }  \max_{j\in t_n(a,b)}   |     \lambda^{j}- 1    |^2       < \infty  .\end{align}  
  Since all incidence matrices are positive,  any strict transition word $aw_n\in \mathcal L^{(n)}$ from $a$ to $b$  must  appear in $\sigma_{n}(\alpha\beta)$ for some word $\alpha\beta \in \mathcal L^{(n+1)}_{\boldsymbol{\sigma}}$ of length 2. Also, by the assumption (iii) of the definition of strong straightness, there exists $r'$ such that  the word $\alpha\beta$ appears in $\sigma_{[n,n+r')}(\gamma)$ for some $\gamma$. Since all incidence matrices are positive, $\alpha\beta$ appears in $\sigma_{[n,n+r'+1)}(a)$.
  Thus by telescoping boundedly,  we can assume that $aw_nb$ appears in $ \sigma_n(a)$.
  We write $\sigma_n(a)=p_naw_nbs_n$. Then  $j_n:= | h_{n} (p_n)|\in t_{n+1}(a,a)$,  $J_n:= | h_{n} (p_naw_n)|\in t_{n+1}(a,b)$, so that since ${\boldsymbol a} = {\boldsymbol b}$, and by \eqref{eq:square-summable}, we have
   \[   \sum_{n=1}^{\infty}  |     \lambda^{j_n}- 1    |^2       < \infty  \mbox{ and } \sum_{n=1}^{\infty}  |     \lambda^{J_n}- 1    |^2       < \infty .\] Thus
  \[   \sum_{n=1}^{\infty}  |     \lambda^{h_n(aw_n)}- 1    |^2    =    \sum_{n=1}^{\infty}  |     \lambda^{J_n - j_n}- 1    |^2  =  \sum_{n=1}^{\infty}  |     \lambda^{J_n }- \lambda^{j_n }    |^2 \leq   \sum_{n=1}^{\infty}  |     \lambda^{j_n}- 1    |^2       +  \sum_{n=1}^{\infty}  |     \lambda^{J_n}- 1    |^2  <\infty.\]
  \end{proof}

\section{The constant-length $S$-adic case}\label{constant-length-S-adic}

 We apply  the results of the previous sections
to
 the family of $S$-adic systems generated by constant-length directive sequences. This is a natural starting point; after all Host's work on coboundaries  generalises the notion of {\em height} in Dekking's and Kamae's works on constant-length substitutions; see \cite{Dekking:1977,Kamae:1972} and also \cite{Queffelec:10}.
 We  first discuss eigenvalues and constant coboundaries in  the continuous case  in Section \ref{subsec:consl}, we then  focus on  the Toeplitz   $S$-adic shifts   in Section 
 \ref{Toeplitz-S-adic}, and     we lastly consider  measurable eigenvalues  in Section \ref{subsec:clmeasurable}. 
 
\subsection{Eigenvalues and constant coboundaries} \label{subsec:consl}

We  first start  with the following remark, easily  providing continuous eigenvalues.
\begin{remark}\label{odometer-eigenvalues}
Let $(\sigma_n)$   be a  recognizable and constant-length directive  sequence with  sequence of lengths  $(q_n)$. Assume that $X_{\boldsymbol{\sigma}}$ is aperiodic.
Note that for any  $m$, 
 $\lambda=e^{\frac{2\pi i}{q_0\dots q_{m}}}$ is an eigenvalue of  the recognizable $S$-adic system  of constant length  $( X_{\boldsymbol{\sigma}},T)$. For, as $h_{m+1}(a) = q_0 q_1 \dots q_{m}$, for each letter $a$,
 we can use the partition
   $\mathcal Q_{m+1}=  \{ T^k \sigma_{[0,m+1)}([a]):\, a  \in \mathcal{A},  0\leq k<    q_0 q_1 \dots q_{m}    \}$ to  build an eigenfunction.  
 Namely,  let $f(x)=\lambda^{k}$ if 
 $x\in  T^k \sigma_{[0,m+1)}([a])$.  Then $f$ is a continuous eigenfunction for the eigenvalue $\lambda$. 
\end{remark}

We can translate the previous remark to information about equicontinuous factors of  $( X_{\boldsymbol{\sigma}},T)$.
Indeed, we    first recall that if  $\exp(  2i \pi /r)$ is an  eigenvalue of $(X,T)$, then 
 the addition of $1$ on the finite group ${\mathbb Z}/r {\mathbb Z}$  is a factor of $(X,T)$.  We now  extend this fact  to the case where 
  $\exp(  2i \pi /q_n)$  is an  eigenvalue of $(X,T)$ for each $n$. We first recall the definition of an odometer
 associated to a sequence $(q_n)_n$ of positive integers,  where we assume $q_n\geq 2$. 
Given a sequence $(q_n)_{n\geq1}$ of positive integers,  define the group \[\Z_{(q_n)} :=\prod_{n\geq 1} \Z/q_n\Z,\] where the group operation of   addition is performed with carry. For a detailed exposition of equivalent formal definitions, we refer the reader to \cite{Downarowicz:2005}. Endowed with the product topology over the discrete topology on each $\Z/q_n\Z$ ,  the group $\Z_{(q_n)}$ is a compact metrizable  topological group, 
where the multiples of the unit $z=\dots 0\,0\,1$, which we simply write as $z=1$, are dense. We write elements $(z_n)_{n\geq 1}$ of $\Z_{(q_n)}$ as left-infinite sequences $\dots z_2\,z_1$ where $z_n\in \Z/q_n\Z$,
so that addition in $\Z_{(q_n)}$ has the carries 
propagating to the left in the usual way as in $\Z$. 
If $q_n=p$ is constant, then  $\Z_{(q_n)}= \Z_{(p)}$ is the classical ring of $p$-adic integers,  and the addition of $1$ on the  $p$-adic group  ${\mathbb Z}_{(p)}$  is a factor if and only if
  $\exp(  2i \pi /p^n)$ is an eigenvalue  for every $n \geq 1$ (se e.g. \cite{Fog02}). 
Given $(q_n)$ and $\tilde h\in \N$, we use  $\Z_{\tilde h,(q_n)}$ to denote the group
defined by the sequence $\tilde h, q_1, q_2, \dots $.  Finally recall that an {\em odometer} is a dynamical system $(Z, +1)$ where 
 $ Z=\Z_{(q_n)}$ for some sequence $(q_n)$ and $+1$ denotes the homeomorphism of adding the unit $1$.

The following  arithmetic lemma  will be essential when finding either measurable or topological eigenvalues of a constant-length $S$-adic shift.
It can be considered as an analogue of  the discussion  in Section  (2.2) in \cite{Host:1986} (see also  Lemma \ref{lem:eigenvectorbis}). 
We  stress the fact that in  the constant-length case, coboundaries $h$  are constant they do not depend on    letters), since  they are  given by limits  of the form  $h= \lim_n  \lambda^{q_0\dots q_{n}} $. The following   statement   is motivated by the fact that    taking the quotient of  two successives terms  yields    $\lim_{n}\lambda^{q_0\dots q_{n}(q_{n+1} -1)}= 1$, from which  we  can deduce crucial information in the finitary case.

\begin{lemma}\label{rational-eigenvalue}
Let  $\lambda \in \mathbb S^{1}$, let the sequence $(q_n)$ of integers  take on  finitely many  values, and suppose that  $q_n\geq 2$ for all $n$.
If
 \begin{equation} \label{convergence}\lim_{n}\lambda^{q_0\dots q_{n}(q_{n+1} -1)}  =1,\end{equation}
  then  $\lambda$ is rational, and more precisely,   
  for each positive integer $m$   sufficiently large, there exists $k$ such that $$\lambda= e^{2\pi i \frac{ k}
{q_0 \dots q_m(q_{m+1}-1) }}.$$

\end{lemma}
\begin{proof}
Let $\lambda = e^{2\pi i t}$. By assumption
\begin{equation} \label{eq:r}
q_0\dots q_{n}(q_{n+1} -1)t  =  k_{n} +  \varepsilon_{n}
\end{equation}
where $k_{n}\in \Z$ and $ \varepsilon_{n}\rightarrow 0$. Let
\begin{align*}
r_n:=q_{n+1}\frac{q_{n+2}-1}{q_{n+1}-1}.
\end{align*}

Note that the sequence $(r_n)$ takes on finitely many values. 
We have that
\begin{align*}
k_{n+1}-r_n k_n &= q_0\dots q_{n+1}(q_{n+2}-1) t - \varepsilon_{n+1}
-r_n \left ( q_0\dots q_{n}(q_{n+1}-1) t - \varepsilon_{n} \right)\\
&= q_0\dots q_{n+1}(q_{n+2}-1) t  -  r_n  ( q_0\dots q_{n}(q_{n+1}-1) t ) - \varepsilon_{n+1}+r_n\varepsilon_n
\\
&= -\varepsilon_{n+1}+r_n\varepsilon_n. 
\end{align*}
Since the sequence $(r_n)_n$ takes on finitely many values and $\varepsilon_n\rightarrow 0$,
we get $k_{n+1}-r_nk_n\rightarrow 0.$ Since in addition $(k_n)$ is integer valued,  we have
 $k_{n+1}-r_n k_n=-\varepsilon_{n+1}+r_n\varepsilon_n=0$ for $n$ large enough. But then, recursively, we get $\varepsilon_{n+\ell} =r_{n+\ell-1}\dots r_{n+1}r_n \varepsilon_{n}  $ for all $\ell\geq 1$.  Since $r_j>1$ for each $j$, we conclude that 
   $\varepsilon_n=0$ for all $n$ large.  The statement of the lemma follows from   (\ref{eq:r}).

\end{proof}

With the notation of the previous lemma, then writing $k$ and $q_0 \dots q_m(q_{m+1}-1)$ in their unique prime factorization and simplifying the common factors, we  even  get an irreducible quotient $\frac{p}{q}$, where the prime factors of $q$ are in those of $\{ q_0,\dots, q_m, q_{m+1}-1\}$.

\begin{example} \label{non-rational}
We give an example to show that the  finitary assumption (i.e., of finitely many values $q_n$) in Lemma \ref{rational-eigenvalue} cannot  be relaxed without additional qualification. 
Take $\lambda = e^{2\pi i \sum_{k=1}^{\infty}  \frac{1}{10^{k!}}}$. The exponent   $t=\sum_{k=1}^{\infty}  \frac{1}{10^{k!}}$ is  the  Liouville number, known to be transcendental. Take $(q_n)$ such that $q_0q_1\dots q_n=10^{ n!}$ for $n\geq 1$. Then  $  q_0\dots q_n \sum_{k=1}^{\infty}  \frac{1}{10^{k!}}\mod \Z \equiv  10^{ n!}     \sum_{k=n+1}^{\infty}  \frac{1}{10^{k!}}\rightarrow_{n} 0$, so \eqref{convergence} is satisfied. 
\end{example}

 \begin{remark} \label{rem:rational}
If $ \lambda$ is rational,   i.e., if  $\lambda  = e^{2i \pi t}$ with $t\in \mathbb Q$, and  $ h(a)=\lim_{n\rightarrow \infty}\lambda^{h_n(a)} $   for some $a$, 
then  $h(a)$ is   rational.
Note  though that $h(a)$  rational does not necessarily imply that $\lambda $ is  rational: take e.g.  the case of a trivial coboundary $h\equiv 1$ with  $\sigma$   being  the Fibonacci substitution
$\sigma: a  \mapsto ab, b \mapsto a$. The  topological eigenvalues of this Pisot  irreducible substitution  belong to $ {\mathbb Z} [\varrho] +  {\mathbb Z}$,
 where $\varrho= \frac{1 + \sqrt 5}{2}$ (see Section \ref{sec:Pisot}). 
However, if
 we are in the finitary constant-length case, Lemma \ref{rational-eigenvalue} implies that  $\lambda$  has to be rational.
\end{remark}

In the case where the directive sequence consists of constant-length substitutions, inspection of the proof of Theorem 
\ref{thm:fully-essential-coboundary}     yields that
 a weak coboundary $(h, \bar f)$ associated to a continuous function  is constant, as in the constant-length substitution case, and also that
 we can lighten the hypothesis concerning fully essential words. Indeed, in Theorem \ref{thm:fully-essential-coboundary} we require that 
 for each $a\in\cA$, there exists $\ell\in\cA$ such that $a\ell$ is a fully essential word. Below, with  Theorem \ref{cor:fully-essential-constant-coboundarybis}, we  just need that there exists a fully essential word of length 2.

We can  now state  the following   theorem.

\begin{theorem}\label{cor:fully-essential-constant-coboundarybis} 
Let $\boldsymbol{\sigma} = (\sigma_n)_{n\ge0}$ be a  straight  constant-length directive sequence on $\mathcal A$, with length sequence $(q_n)_{n\geq 0}$, where each $q_n\geq 2$. Suppose that $X_{\boldsymbol{\sigma}}$ is aperiodic and 
 that  $\boldsymbol \sigma$ has a  fully essential word of length 2.

If $\lambda\in {\mathbb S}^1$ is a continuous  eigenvalue of $(X_{\boldsymbol{\sigma}}, T)$,  then 
\begin{equation}\label{constant-coboundary}h:= \lim_{n\rightarrow \infty}\lambda^{q_0 \dots q_{n}}\end{equation}
exists and 
  defines a constant   weak  coboundary.
  
Assume  in addition  that  $\boldsymbol{\sigma} = (\sigma_n)_{n\ge0}$ is  finitary and recognizable.
Suppose that  $h:= \lim_{n\rightarrow \infty}\lambda^{q_0 \dots q_{n}}$
exists  for some $\lambda\in {\mathbb S}^1$, and that $h$  is a  coboundary.   Then
   $\lambda$ is a continuous eigenvalue, and
 $e^{\frac{2\pi i}{q_0\dots q_{m}}}$ is  also a continuous  eigenvalue,  for any  $m$. Moreover,
each continuous eigenvalue  is  rational,   and there exists $\tilde h \in \N$  
\begin{itemize}\item
 which is coprime to each $q_n$, and 
 \item
 which divides $p: = \prod_{q\in \{ q_n:n\geq 0\}}{(q-1)}$, where the product is over the set of distinct values  of elements in $ \{ q_n:n\geq 0\}$,
 \end{itemize}
 such that the maximal equicontinuous factor of 
$(X_{\boldsymbol{\sigma}},T  )$ is the odometer $(\Z_{\tilde h, (q_n)}, +1)$.

\end{theorem}

\begin{proof}

To prove the first statement, we repeat the proof of Theorem \ref{thm:fully-essential-coboundary}, except working only with one fully essential word $ab$. We obtain
\begin{equation*}
 \lim_{n\to \infty} \lambda^{ q_0 q_1 \dots q_n} f( {\boldsymbol a})=  \lim_{n\to \infty} \lambda^{ h_{n}(a) }f( {\boldsymbol a})=f({\boldsymbol b}),
\end{equation*}
so that \eqref{constant-coboundary} holds. The rest of the proof is as in Theorem \ref{thm:fully-essential-coboundary}.

If ${\boldsymbol \sigma}$ is finitary  and recognizable,  and  $h=\lim_{n\rightarrow \infty}\lambda^{q_0 \dots q_{n}}$ exists, then by Lemma \ref{rational-eigenvalue}, $\lambda=e^{2\pi i t}$  where $t$ is rational; write  
$t= p/q$ with $q\geq 1$   and where $p$ and $q$ are non-zero coprime integers. The claim that $\lambda$ is a continuous eigenvalue follows by Theorem  \ref{Host-continuous-bounded-onesided}. 

Remark \ref{odometer-eigenvalues} tells us that $(\Z_{ (q_n)}, +1)$ is an equicontinuous factor of $(X_{\boldsymbol{\sigma}},T  )$. 
Finally, Lemma \ref{rational-eigenvalue} tells us that $\frac{p}{q}= \frac{ k}
{q_0 \dots q_m(q_{m+1}-1) }$ for some $m$. Therefore, if $\frac{p}{q}$ is an additional continuous eigenvalue, we need only add some factor of 
$(\Z/ (q_{m+1}-1)\Z, +1)$,  to $(\Z_{ (q_n)}, +1)$ to obtain possibly a larger equicontinuous factor of $(X_{\boldsymbol{\sigma}},T  )$.  
 The result follows.
 \end{proof}

The following definition  of  height in the $S$-adic case is a generalisation of Definition \ref{def:height-substitution}.   Recall the discussion there, of the relationship between the height of a constant length substitution, and the existence of a nontrivial coboundary.  The following is motivated by the  close relation between eigenvalues and height.
\begin{definition}[$S$-adic height]\label{def:height-Sadic} 
Let $\boldsymbol{\sigma} = (\sigma_n)_{n\ge0}$ be a  finitary, straight, aperiodic,  recognizable, constant-length directive sequence on $\mathcal A$,   with length sequence $(q_n)_{n\geq 0}$, where $q_n\geq 2$ for all $n\geq 2$. Suppose that  $\boldsymbol \sigma$ has a fully essential word of length two.
 Then, by  Theorem \ref{cor:fully-essential-constant-coboundarybis},    
 \begin{align*}\tilde h&:=  lcm\{q\in \mathbb N \,: \, \lambda = e^{2\pi i /q} \mbox{ is an eigenvalue 
 of  $(X_{\boldsymbol \sigma}, T)$, with $q$ coprime to each $q_n $     }\} \end{align*}
is finite. We call $\tilde h$ the {\em height}
of ${\boldsymbol \sigma}$. \end{definition}

As in the case of constant-length substitutions,  the factors of $\tilde h$  define the set of  eigenvalues  of $(X_{\boldsymbol \sigma},T)$ not determined by the sequence of lengths $(q_n)$
(as expressed in Theorem \ref{cor:fully-essential-constant-coboundarybis}).

\begin{example}\label{ex:running-fifth}
 Here we consider the constant-length directive sequence in Example~\ref{ex:running-fourth} which is recognizable, aperiodic and straight as we previously discussed.
  Since  $ba$ is a fully essential word,  Theorem \ref{cor:fully-essential-constant-coboundarybis}  tells us that if $\lambda$ is a continuous eigenvalue, then  
 $h:=\lim_{n}\lambda^{3^{n}}$ exists.  
 As we computed earlier in  Example \ref{ex:running-third}, 
  both ${\boldsymbol{ a}}$ and 
 ${\boldsymbol{ b}}$   consist of the single limit word $u= \lim_{n\rightarrow \infty} \sigma_{[0,n)}(a)$.  
 Therefore if  $f$ is an eigenfunction  associated to $\lambda$, this gives,  according to the proof of  Theorem \ref{cor:fully-essential-constant-coboundarybis}, 
 \[ f( {\boldsymbol{ b}}) h=f({\boldsymbol{ a}}),\]
 so that we must have $h= 1$.  This implies that $h$   is a strong coboundary. By Theorem \ref{cor:fully-essential-constant-coboundarybis}, all continuous eigenvalues are rational. Also, any eigenvalue must be of the form  $ \lambda= e^{\frac{2\pi i}{3^{n}\tilde h}}$where $\tilde h \in \{1,2\}.$ Since   $\lim_{k}\lambda^{3^{n}}=1$,
 one has $\tilde h=1$. This $S$-adic system therefore has trivial height,
 and so  its maximal equicontinuous factor is $(\Z_{(3)},+1)$.

 \end{example}

Recall the definition of height in Definition \ref{def:height-substitution}.
From Theorem \ref{cor:fully-essential-constant-coboundarybis} we 
recover its role in the case of a constant-length substitution, as seen in \cite{Kamae:1972, Dekking:1977}, \cite[Theorem 6.1]{Queffelec:10}.

\begin{corollary}\label{cor:Cobham0}
 Let $\sigma$ be a   straight aperiodic substitution of constant-length $q$. 
 Then the continuous eigenvalues of $(X_\sigma,T)$ are generated by $\{e^{\frac{2\pi i}{q^n}}: n\in \N \}\cup \{ e^{\frac{2 \pi i}{\tilde h}}\}$, where $\tilde h$ is the height of $\sigma$. 
 Furthermore $\tilde h$ divides $q-1$. 
\end{corollary}

Similarly, if a finitary directive sequence satisfies the conditions of Theorem \ref{cor:fully-essential-constant-coboundarybis}, and the shift is aperiodic, then  the continuous eigenvalues of $(X_{\boldsymbol{\sigma}},T)$ are generated by $\{e^{\frac{2\pi i}{q_0 \dots q_n}}: n\in \N \}\cup \{ e^{\frac{2 \pi i}{\tilde h}}\}$, where $\tilde h$ is the height of $\boldsymbol{\sigma}$.

From Theorem \ref{cor:fully-essential-constant-coboundarybis} we 
 also obtain a version of Cobham's theorem.  For more on its classical  substitutive  version, see \cite{Cobham}.

 \begin{corollary}\label{cor:cobham}
Let $\boldsymbol{\sigma} = (\sigma_n)_{n\ge0}$  and $\boldsymbol{\tilde{\sigma}} = (\tilde{\sigma}_n)_{n\ge0}$ be two  finitary, straight, aperiodic, recognizable,  constant-length directive sequences  on $\mathcal A$,    with length sequences $(q_n)_{n\geq 0}$ and
$(\tilde{q}_n)_{n\geq 0}$, where each  $q_n\geq 2$ and $\tilde {q}_n\geq 2$. Suppose  that
  each directive sequence possesses a fully essential word of length 2.  
  If there is a prime factor $p$ of infinitely many of the  lengths $q_n$ that is a prime factor of only finitely many of the lengths $\tilde q_n$,
   then $(X_{\boldsymbol{\sigma}},T  )$ and $ (X_{\boldsymbol{\tilde \sigma}},T  )$ cannot be topologically conjugate.
 \end{corollary}

\begin{proof}
The conditions we have imposed allow us to deduce from Theorem \ref{cor:fully-essential-constant-coboundarybis}  that 
 the maximal equicontinuous factor of 
$(X_{\boldsymbol{\sigma}},T  )$ is $(\Z_{h_1, (q_n)}, +1)$, and that of  $\boldsymbol{\tilde{\sigma}} = (\tilde{\sigma}_n)_{n\ge0}$ is $(\Z_{h_2, ({\tilde q_n})}, +1)$, for some finite natural numbers   $ h_1$ and $h_2$. The fact that the prime factor $p$ appears infinitely often in the lengths $q_n$ implies that $(\Z_p,+1)$ is an equicontinuous factor of $(X_{\boldsymbol{\sigma}},T  )$, while the fact that $p$ divides  only finitely often in the lengths $\tilde q_n$   implies that $(\Z_{p},+1)$ is not an equicontinuous factor of $(X_{\boldsymbol{\tilde{\sigma}}},T)$.
 Now equicontinuous factors
encode  continuous eigenvalues. In particular, for each $n$, $e^{2\pi i /p^n}$ is an eigenvalue of $(X_{\boldsymbol{\sigma}},T  )$, but only finitely many  of the $e^{2\pi i /p^n}$ can be eigenvalues of $(X_{\boldsymbol{\tilde{\sigma}}},T)$.
 The statement follows since  the set of continuous eigenvalues is a topological invariant.
\end{proof}

To see why the requirement that $p$ has to divide $q_n$ infinitely often is necessary, recall that the height of the second system can in principle absorb a $p$ which divides $q_n$ only finitely many times. For example, consider the situation where $q_n=3$ for each $n$, and $(\tilde {q}_n)= 2,3,3,3 \dots$; then it can happen that both $(X_{\boldsymbol{\sigma}},T  )$ and  $(X_{\boldsymbol{\tilde \sigma}},T  )$ have $\Z_{2,(3)}$ as maximal equicontinuous factor, in particular if $(X_{\boldsymbol{\sigma}},T  )$ has height 2. Take for example the stationary
$\boldsymbol{\sigma}$ generated by the substitution $a\mapsto aba$, $b\mapsto bac$, $c\mapsto cab$. It can be verified that this substitution has height 2, so that its maximal equicontinuous factor is $\Z_{2,(3)}$.

The following example gives a recipe  to find constant-length directive sequences which have a non-trivial height. We formulate it to generate straight, recognizable directive sequences, by modifying the first two ingredients appropriately.

\begin{example}\label{ex:nontrivial height}
Let $\mathcal S$ be any finite set of aperiodic, primitive constant-length substitutions on $\mathcal A=\{a,b,c,d \}$ where
\begin{enumerate}
\item for  all $\sigma\in \mathcal S$
  for  each  $\alpha \in \{a,b\}$,  $\sigma(\alpha)$  begins with the same letter in \{a,b\} for  all $\sigma\in \mathcal S$, and for  each  $\alpha \in \{c,d\}$,  $\sigma(\alpha)$  begins with the same letter in  $ \{c,d\}$ for  all $\sigma\in \mathcal S$,
\item each substitution in $\mathcal S$ is rotationally conjugate to a left- or right- permutative substitution, 
\item
there is a word $w$ of length two such that for each $\sigma\in \mathcal S$ there is  $\alpha \in \mathcal A$ where  $w$  occurs in  $\sigma(\alpha)$,
 \item
for each $\sigma\in \mathcal S$, any occurrence of  a letter in $\{a, b\}$ in  the image of  a letter by a  substitution is always followed by a letter in $\{c, d\}$, and any occurrence of a letter in $\{c, d\}$ is always followed by a letter  $\{a, b\}$, and 
\item each substitution in $\mathcal S$ has odd length at least three.
\end{enumerate}

Then we claim that $-1$ is an eigenvalue, provided that the resulting shift is aperiodic.
 Indeed, Condition (i) ensures that any primitive directive sequence from $\mathcal S$ is straight.
   Condition (ii) ensures that any directive sequence is recognizable, using Theorem \ref{c:rec}.
  Condition (iii) ensures that  there is a fully essential word of length two.  
  Therefore Theorem \ref{cor:fully-essential-constant-coboundarybis} applies.
  Finally, Conditions (iv) and (v) ensure that in all words of the language $\mathcal L_{\boldsymbol \sigma}$, a letter from $\{a,b\}$ is always followed by a letter from $\{c,d\}$ and conversely. Thus  -1 is an eigenvalue, as
  \[ \mathcal P = \{ [a]\cup[b], [c]\cup [d]\}\]
  is then a clopen partition which forms a Rokhlin tower of height 2, that is $X_{\boldsymbol \sigma}=A\cup B$ with  $T(A)=B$, $T(B)=A$, here with $A=[a]\cup[b]$, $B= [c]\cup [d]$.

  Take for example
 $S=\{\sigma, \tau\}$ with

\begin{align*}
     a  \stackrel{\sigma}{\mapsto} bda,  \ \ &   a  \stackrel{\tau}{\mapsto}  bcad  a
  \\
          b  \stackrel{\sigma}{\mapsto} bcb,  \ \ &   b  \stackrel{\tau}{\mapsto}  bd adb  \\
             c  \stackrel{\sigma}{\mapsto} cac,  \ \ &   c  \stackrel{\tau}{\mapsto}  cacbc
     \\    d  \stackrel{\sigma}{\mapsto} cbd,  \ \ &   d  \stackrel{\tau}{\mapsto}  cb dad.
  \end{align*}

  The above general remarks tell us that $-1$ is an eigenvalue.  The conditions of Theorem \ref {cor:fully-essential-constant-coboundarybis}  are satisfied, so 
 we could in principle have an eigenvalue $\lambda$ which is a fourth root of unity.  We show that this does not occur. If $\lambda $ is a continuous eigenvalue, then by Theorem \ref {cor:fully-essential-constant-coboundarybis},   $\lambda $ defines a weak coboundary $(h, \bar f)$.
  Assume first that $\tau$ appears infinitely often in the directive sequence. 
  Now  the fact that $\tau(b)$ appears infinitely often and contains the subword  $dad$ tells us that  $\bar f (d) h^2 = \bar f (d)$, and  this means that $h$ cannot be a fourth root of unity, and hence the eigenvalue -1 is the only eigenvalue that appears in addition to those in Remark \ref{odometer-eigenvalues}.
  Finally if  $\tau$ appears only finitely often, 
  then, since    $\boldsymbol{a}= \boldsymbol{b}$,
  we can repeat the argument above, but with the essential word $bda$, to arrive at the same conclusion.
  
  For the above family, if  $X_{\boldsymbol \sigma}$ is aperiodic, then we can  use Theorem \ref{cor:fully-essential-constant-coboundarybis} to  deduce the explicit form of the maximal equicontinuous factor space of the given $S$-adic shift. For example,  if $\sigma$ and $\tau$ above each appear infinitely often in a directive sequence from $S$, then the appropriate  group is isomorphic to  $\Z_{ 2,(15)}$.

Note that  any directive sequence chosen from  the specific $S=\{\sigma, \tau\}$
  will give a strongly  straight directive sequence (see Definition \ref{def:strongly-straight}); in general though  examples satisfying (i) above are not necessarily strongly  straight.

  Finally we remark that this kind of construction is essentially the only kind that yields height in the constant length $S$-adic case, see \cite[Corollary 5.5]{BMY}.  Also, this technique can be extended to obtain nontrivial coboundaries for  non-constant length directive sequences.

  \end{example}

\subsection{Toeplitz $S$-adic shifts and discrete spectrum}\label{Toeplitz-S-adic}

In this section we focus on a special family of constant-length $S$-adic shifts which has been extensively studied, namely those that are  {\em almost-automorphic}. We apply our previous results to this family and then compare our results to existing results in the  literature.

We  first start  with a   few definitions.    A factor $(Z,S)$  of $(X,T)$ via a map $\pi:(X,T) \rightarrow (Z,S)$ is {\em almost one-to-one} if the set 
$\{x\in X: \pi^{-1}(\{\pi(x) \})=\{x\}\}$ is dense in $X$. 
In this case, we call the system $(X,T)$ an almost one-to-one extension of $(Z,S)$.  
A system $(X,T)$ is {\em almost automorphic} if it is an almost one-to-one extension of 
a minimal equicontinuous system. Almost automorphic systems are necessarily minimal, and for minimal systems $\{x\in X: \pi^{-1}(\{\pi(x) \})=\{x\}\}$ is a dense $G_\delta$ if it is non-empty. Thus if there is some $x\in X$ such that $|\pi^{-1} (\pi(x))|=1$, then $(X,T)$ an almost one-to-one extension of $(Z,S)$.
A {\em  Toeplitz shift} is a symbolic shift $(X,\sigma)$, $X\subset \mathcal A^{\Z}$ with $\mathcal A$ finite, which is an almost automorphic extension of an odometer
$(Z,+1)$, via some factor  $\pi:(X,\sigma)\rightarrow ( Z, +1)$. Note that Gjerde and Johansen \cite{Gjerde-Johansen-2002} give a nice characterisation of Toeplitz shifts in terms of the  equal path number property for 
  their Bratteli-Vershik representations.

 Given a substitution $\sigma_n: \mathcal A\rightarrow \mathcal A^{q_n}$, we write it as  $\sigma_n= \sigma_{n,0} \sigma_{n,1} \dots \sigma_{n,q_n-1}$, a    concatenation of $q_n$ maps $\sigma_{n,i}:\mathcal A\rightarrow \mathcal A$. We say that 
  {\em $\sigma_n$ has a coincidence at i} if $|\sigma_{n,i} (\mathcal A)|=1$. We say that the directive sequence $\boldsymbol \sigma$  has a coincidence if each of of the substitutions $\sigma_n$ has  a coincidence at some $i$.
We have the following result, a straightforward generalisation of a result due  to Kamae and Dekking \cite[Theorem 7]{Kamae:1972}; see also  \cite[Theorem  7]{Dekking:1977}, which states that a measurable constant length substitution shift has discrete spectrum if and only if the {\em pure base} of the substitution has a coincidence; see the aforementioned papers for definitions and details.

\begin{theorem}\label{thm:discrete-spectrum} 
Let $\boldsymbol{\sigma} = (\sigma_n)_{n\ge0}$ be a finitary, straight, recognizable, constant-length directive sequence on $\mathcal A$,  with length sequence $(q_n)_{n\geq 0}$, where each  $q_n\geq 2$.
Suppose that 
 $\boldsymbol{\sigma} = (\sigma_n)_{n\ge0}$  has a  fully essential word of length 2,  that 
 infinitely many of the substitutions $\sigma_n$  have a coincidence, and that $X_{\boldsymbol{\sigma}} $ is aperiodic.   Then $(X_{\boldsymbol{\sigma}},T )$ is Toeplitz and uniquely ergodic. If $\mu$ is the unique invariant measure, then the measure preserving system $(X_{\boldsymbol{\sigma}},T, \mu)$   has discrete spectrum, $\boldsymbol{\sigma}$ has trivial height,  and
 the maximal equicontinuous factor of 
$(X_{\boldsymbol{\sigma}},T  )$ is  $(\Z_{(q_n)}, +1)$.
\end{theorem}

\begin{proof}
We assume that each $\sigma_n$ has a coincidence; this does not change the generality of our arguments (otherwise we just need to add a layer in the indexing of  subsequences).
By Theorem \ref{cor:fully-essential-constant-coboundarybis}, the maximal equicontinuous factor of  $(X_{\boldsymbol{\sigma}},T  )$ is  $(\Z_{\tilde h, (q_n)}, +1)$ for some $\tilde h$ which
divides $\prod_{j=1}^{N }(q_{j} -1)$ for some $N$.

 Note that  $(\Z_{ (q_n)}, +1)$ is an equicontinuous factor of the system.
 Recognizability means that for each $n$ and each $x\in X_{\boldsymbol{\sigma}}$,
there are  a unique  $m_i$ with $0\leq m_i < q_i$ and a unique $x^{(n)}\in X_{\boldsymbol{\sigma}}^{(n)}$ such that $x= T^{m_0}\sigma_0 T^{m_1}\sigma_1 \dots T^{m_{n-1}}\sigma_{n-1}(x^{(n)})$, and furthermore the data $m_0, \dots ,m_{n+1}$,  which let us desubstitute $x$ to $X_{\boldsymbol{\sigma}}^{(n+1)}$, agree with the data $m_0, \dots , m_n$, which let us desubstitute $x$ to $X_{\boldsymbol{\sigma}}^{(n)}$.
This allows us to define a maximal equicontinuous factor map
 $\pi: X_{\boldsymbol{\sigma}}\rightarrow \Z_{ (q_n)}$, where $\pi(x)=\dots m_2\,m_1\,m_0$. Let $\lambda$ be the Haar probability measure on  $\Z_{ (q_n)}$, and let $\mathcal Z$ be the set of points in $ \Z_{ (q_n)}$ which are not invertible under $\pi$. Denote by $C_j$ the set of indices $0<i<q_j -1$ such that $\sigma_j$ has a coincidence, i.e.,  $|\sigma_{j,i}(\mathcal A)|=1$. If needed to ensure that this set is nonempty, we can telescope two substitutions at a time, to ensure that the set of coincidences indices is not ``extremal". We have that $z\not \in \mathcal Z$ if  $z_j\in C_j$ infinitely often. In other words, if $[C_j]^c:=\{z\in   \Z_{ (q_n)}:  z_j \not\in [C_j]\}$, then 
\[ \mathcal Z\subset \bigcup_{k=1}^{\infty} \bigcap_{j\geq k}  [C_j]^c,\]
so that 
\begin{align*}
\lambda (\mathcal Z) &\leq \lim_{k\rightarrow \infty }\lambda \left(\bigcap_{j\geq k}  [C_j]^c\right) \leq  \lim_{k\rightarrow \infty } \prod_{j\geq k}\frac{q_j -1}{q_j} = 0,
\end{align*}
with the second inequality following because we assumed that each $\sigma_n$ has a coincidence, and the last equality following since the sequence $(q_n)$ takes on a finite number of values. It follows that  $(X_{\boldsymbol{\sigma}},T  )$ is uniquely ergodic, see for example \cite[Theorem 4.12]{ABKL-2015}. Let $\mu$ be the unique invariant measure.

To show that $(X_{\boldsymbol{\sigma}},T, \mu  )$  has discrete spectrum, we must show that $\mu(\pi^{-1}(\mathcal Z))=0$.
Define $D_j:=\{ x: (\pi (x))_j\not\in C_j\}$. Then  as above, $\pi^{-1} (\mathcal Z) \subset   \bigcup_{k=1}^{\infty}   \bigcap_{j\geq k} D_j $, and as for each $j$, $\sigma_j$ defines a $q_j$-cyclic partition of $X_{\boldsymbol{\sigma}}^{(j+1)}$, hence each element of this partition has measure $1/q_j$ and we have that $\mu ( \bigcap_{j\geq k} D_j)  
 \leq \prod_{j\geq k}\frac{q_j -1}{q_j}=0$. Therefore $\mu(\pi^{-1}(\mathcal Z))=0$.
 
 Any point $z\in \Z_{ (q_n)}$ such that $z_j\in C_j$ infinitely often satisfies $|\pi^{-1}(z)|=1$.  
For minimal systems $\{x\in X: \pi^{-1}(\pi(x)) = \{x\} \}$  is a dense $G_\delta$ set if it is non-empty (which holds from the above)
Thus $(X_{\boldsymbol{\sigma}},T, \mu  )$ is almost automorphic and  the fact that $\pi$ is somewhere injective implies that it is a maximal equicontinuous factor map, by \cite[Proposition 1.1]{Williams:84}. Thus
 its maximal equicontinuous factor is $(\Z_{ (q_n)}, +1)$; this forces $\tilde h = 1$.

\end{proof}

\begin{example}\label{ex:running-sixth}

We have shown that  the directive sequence
 \[\sigma, \tau , \sigma, \sigma, \tau,  \sigma, \sigma , \sigma, \tau,  \sigma, \sigma\dots \]
 is recognizable in Example~\ref{ex:running-first}, and also straightin Example~\ref{ex:running-third}.
   The $S$-adic shift   $(X_{\boldsymbol{\sigma}},T )$ has a  unique  ergodic measure $\mu$, see for example   Proposition \ref{prop:exact} and the comment below it.  The substitution 
$\tau$ has a  coincidence at index $0$. Since   $\tau$  appears infinitely often in   the directive sequence,  Theorem \ref{thm:discrete-spectrum} tells us that  $(X_{\boldsymbol{\sigma}},T )$ is Toeplitz and that  $(X_{\boldsymbol{\sigma}},T,\mu)$ has discrete spectrum.

\end{example}

We compare our results on constant-length $S$-adic shifts  to those in the literature   for  Toeplitz shifts, of whose spectrum   there is an extensive study. The maximal equicontinuous factor of a Toeplitz shift  is always an odometer \cite{Williams:84}.
In other words, each continuous eigenvalue of a Toeplitz shift is rational.
Furthermore, it was known that there could exist rational continuous eigenvalues   that 
are not given by the period sequence  $(p_n)_{n\geq 0}$  \cite{Downarowicz:2005}. Our contribution in the case of Toeplitz shifts is to quantify  the eigenvalues  if the Toeplitz shift is given as a constant-length $S$-adic shift, as given by Theorem \ref{cor:fully-essential-constant-coboundarybis}, and to describe them as heights, as in the constant-length substitution case, and finitely   extending Cobham's theorem to this setting.

In the measurable setting, Downarowicz and Lacroix showed that  if $K\subset {\mathbb S}^1$ is a countable subgroup containing infinitely many rationals, then there exists a minimal, uniquely ergodic Toeplitz shift  whose measurable eigenvalues equals $K$ \cite{Downarowicz-Lacroix:1996}. Therefore we see that the set of measurable eigenvalues of a Toeplitz shift can be much larger than the set of continuous eigenvalues. This is further studied by  Durand, Frank and Maass 
 \cite{Durand-Frank-Maass:2015}, who use the characterisation of  Toeplitz shifts  as expansive symbolic systems which have proper Bratteli-Vershik representations  defined by  a sequence of constant-length morphisms \cite{Gjerde-Johansen-2002}. In this setting recognizability is built into the representation, but the morphisms need to be assumed {\em proper}, which is a restriction we do not impose, and automatically implies that the $S$-adic shifts considered in  \cite{Durand-Frank-Maass:2015} are Toeplitz, whereas constant-length $S$-adic shifts are generally not; see also \cite{ADE:23} where the proper condition is extended  in terms of  coincidences. Finally, in moving to the setting of proper Bratteli-Vershik representations, the notion of height disappears.
If one considers a Toeplitz shift with a proper Bratteli-Vershik representation in which the morphisms are taken from a finite set and all of whose incidence matrices are positive, then these systems are linearly recurrent \cite{Durand:00b}, and  all their measurable eigenvalues are continuous \cite{Bressaud-Durand-Maass:2005}. 

 Mentzen studies a family of measure preserving systems on the unit interval, those of {\em exact uniform rank} \cite{Mentzen:1991};  these are the closest dynamical systems to constant-length $S$-adic shifts, as
they can be seen
 as  a geometric  realisation of such symbolic systems. We study this next.

\subsection{Measurable spectrum} \label{subsec:clmeasurable}

 The notion of {\em uniform} exact rank is investigated in \cite{Mentzen:1991}  - these are measure preserving finite rank dynamical systems on the unit interval which have a sequence of generating partitions $(\mathcal R_n)$, where $\mathcal R_n$ consists of a fixed number of Rokhlin towers, all of which are of  the same height, and where this representation is of exact finite rank.
 By Theorem \ref{thm:Bratteli-Vershik}, if an   $S$-adic shift $(X_{\boldsymbol \sigma}, T)$  is defined using a  recognizable,  directive sequence of constant length, such that $(X_{\boldsymbol \sigma}, T, \mu)$ is of exact finite rank, then  it is measurably conjugate to a uniform exact rank system as defined by Mentzen.
Mentzen shows that such a system must necessarily have rational measurable spectrum, i.e., translating his result to the $S$-adic setting, we have the following which should be compared with  Theorem \ref{cor:fully-essential-constant-coboundarybis}; his proof techniques are different.
 
 \begin{theorem}\label{mentzen} \cite[Corollary 1]{Mentzen:1991}
 Let $\boldsymbol{\sigma} = (\sigma_n)_{n\ge0}$ be a recognizable,  constant-length directive sequence  on $\mathcal A$, with length sequence $(q_n)_{n\geq 0}$.  If  $\mathsf{\boldsymbol{\sigma}}$  is of exact finite rank, then 
 every measurable eigenvalue is rational.  \end{theorem}

If we specialise to constant-length directive sequences the statement of Theorem \ref{thm:sufficient-mble}, we may relax the condition that the directive sequence be strongly  straight in a variety of ways. 
We illustrate with one that we will use.  Note that the condition that the directive sequence admits a unique right-infinite limit word, say ${\boldsymbol a}$,   and    moreover,  that $\sigma_n(a)$ starts with $a$ for  each $n$, is weaker than the condition that there is a letter $a$ such that for each $n$ and each $\alpha$, $\sigma_n(\alpha)$ starts with $a$, which is the required   condition of properness in \cite{CDHM:2003}.
For example, recall the substitutions $\sigma$ and $\tau$ from Example \ref{ex:first}, and note that any directive sequence ${\boldsymbol \sigma}$ selected from
$\{\sigma, \tau \}$ satisfies this condition, provided that $\tau$ appears infinitely often.

Recall that strongly primitive finitary directive  sequences generate uniquely ergodic shifts. We denote the unique invariant measure by $\mu$ in the following.

\begin{theorem}\label{thm:sufficient-mble-constant}
Let $\boldsymbol{\sigma} = (\sigma_n)_{n\ge0}$ be a  finitary,  
strongly primitive constant-length directive sequence  on $\mathcal A$, with length sequence $(q_n)_{n\geq 0}$,  where each $q_n \geq 2$.
Suppose that the directive sequence admits a unique right-infinite limit word, say ${\boldsymbol a}$,
   and  moreover that  $\sigma_n(a)$ starts with $a$ for each $n$. If
 $\lambda\in {\mathbb S}^1$ is a measurable eigenvalue of $(X_{\boldsymbol \sigma}, T,\mu)$,  then 
\[   \sum_{n=1}^{\infty} | \lambda^{q_0 \dots q_n}   -1     |^2 < \infty. \]
\end{theorem}

\begin{proof}

We follow the proof of Theorem \ref{thm:sufficient-mble}, to the step which yields \eqref{eq:general-l2}, i.e., 
\begin{equation*}   \sum_{n=1}^{\infty} \max_{a,b} \max_{j\in t_n(a,b)}    |     \lambda^{j}v_n(a)-v_{n-1}(b)    |^2       < \infty .\end{equation*}

Since  
 $\sigma_{n-1}(a)$ starts with $a$ for each $n$, then  $0\in 
 t_n(a,a)$,  and letting $j=0$ in \eqref{eq:general-l2}, we get
  \begin{equation}      \label{eq:l2-constant}            
   \sum_{n=1}^{\infty}    |     v_n(a)-v_{n-1}(a)    |^2       < \infty .\end{equation}

By Lemma \ref{lem:bastard-referee},
 $\min_{a\in \mathcal A}|v_n(a)|\rightarrow 1$. Therefore from  \eqref{eq:l2-constant}, we obtain, since   $q_0 \dots q_n$ belongs to $t_{n+1}(a,b)$ for some $b$, and any such $b$ satisfies $\boldsymbol a= \boldsymbol b$,   \[   \sum_{n=1}^{\infty}    |     \lambda^{ q_0 \dots q_n}- 1    |^2       < \infty .\]

\end{proof}

\begin{corollary}\label{cor:finitary-constant-sadic}
Let $\boldsymbol{\sigma} = (\sigma_n)_{n\ge0}$ be a  finitary,   recognizable and strongly primitive  constant-length directive sequence  on $\mathcal A$, with length sequence $(q_n)_{n\geq 0}$,  where each $q_n \geq 2$.
 Suppose that either
\begin{itemize}
\item ${\boldsymbol \sigma}$ is strongly straight, or
\item  there is a unique right-infinite limit word ${\boldsymbol a}$, and such that  $\sigma_n(a)$ starts with $a$ for each $n$.
\end{itemize}
Then every measurable eigenvalue for the uniquely ergodic $(X_{\boldsymbol \sigma}, T, \mu)$ is continuous.
\end{corollary}
\begin{proof}
By Proposition \ref{prop:exact} and Theorem \ref{mentzen}, every measurable eigenvalue $\lambda$ is rational. With the first hypothesis, we have that  the
 sum   \[\sum_{n=1}^{\infty}\max_{a,b: {\boldsymbol a}= {\boldsymbol b}} | \lambda^{h_n(w_n(a,b))}   -1     |^2\] is finite for strict transition words $w_n(a,b)$ by Theorem 
\ref{thm:sufficient-mble}. We conclude in particular that $\lambda^{h_n(w_n(a,a))}   =1$  $n\geq n_0$ and each $a$, since $\lambda$ is rational.  

Furthermore, the third condition of the definition of strongly straight tells us that any strict return word to $a$ can be contained in $\sigma_n(\alpha)$ for some $\alpha$, by telescoping boundedly if needed. This means that there is a finite set of words to which a strict return word at any level belongs. Thus $n_0$ can be taken to be independent of the given sequence of strict return words. This also implies that $\lambda^{h_n(w_n(a,a))}   -1=0$ when $w_n(a,a)$ is a (not necessarily strict) return word  to $a$.  Now,  Proposition \ref{Host-continuous-bounded-rw} gives the result.

With the second hypothesis, Theorem \ref{thm:sufficient-mble-constant} implies
\[ \sum_{n=1}^{\infty}    |     \lambda^{ q_0 \dots q_n}- 1    |^2       < \infty ,\]
  or $    \lambda^{ q_0 \dots q_n}- 1  =0$ for large enough $n$.
 Now,   Theorem \ref{Host-continuous-bounded-onesided} 
 implies that $\lambda$ is a continuous eigenvalue.
\end{proof}

\begin{example}\label{ex:running-seventh} 
We conclude by recapping the eigenvalues, measurable and topological of the $S$-adic shift of our running Example~\ref{ex:running-first}.
We have shown that  the directive sequence
 \[\sigma, \tau , \sigma, \sigma, \tau,  \sigma, \sigma , \sigma, \tau,  \sigma, \sigma,\sigma \dots \]
 is recognizable,  straight, and  if we re-write it as 
  \[\sigma, \tau \circ \sigma, \sigma,\tau \circ \sigma, \sigma , \sigma,\tau \circ \sigma, \sigma,\sigma\dots ,\]
  it is strongly primitive.
  We have also seen that   $(X_{\boldsymbol \sigma},T)$ is of exact finite rank, although it is not linearly recurrent  ~\cite{Durand:00b}, and that   its maximal equicontinuous factor is $(\Z_{(3)},+1)$. Finally, while  it is not strongly  straight,  it satisfies the second condition of Corollary \ref{cor:finitary-constant-sadic}, so $(\Z_{(3)},+1)$ is also  its Kronecker factor,  which we  define  as the maximal measure theoretic factor of the system  that is isomorphic to a rotation on a compact abelian group.

\end{example}

\bibliographystyle{amsalpha}
\bibliography{Spectrebib}

 \end{document}